\documentclass[microtype]{gtpart}

\usepackage[T1]{fontenc}
\usepackage[utf8]{inputenc}
\usepackage[UKenglish]{babel}

\usepackage{graphicx}
\usepackage{bookmark}
\usepackage{xcolor}
\colorlet{darkgreen}{green!70!black}
\usepackage{rotating}

\usepackage[labelfont=up,margin=5pt]{subcaption}
\usepackage{wrapfig}
\newcommand{\myfixwrapfig}{~\vspace{-28pt}}

\usepackage{pstricks,pstricks-add,pst-pdf}

\usepackage{ifpdf}
\ifpdf
\usepackage{tikz}
\usepackage{tikz-cd}
\fi

\usepackage{cancel}
\usepackage{enumerate}

\usepackage{mathtools}
\mathtoolsset{mathic=true}

\title{Kauffman states and Heegaard diagrams for tangles}
\author{Claudius Zibrowius}
\givenname{Claudius}
\surname{Zibrowius}
\address{Department of Mathematics,
	The University of British Columbia,
	1984 Mathematics Road,
	Vancouver, BC,
	Canada, V6T 1Z2}
\email{claudius.zibrowius@posteo.net}
\urladdr{https://www.dpmms.cam.ac.uk/~cbz20/}


\arxivreference{1601.04915}

\newtheorem{theorem}{Theorem}[section]
\newtheorem*{theorem*}{Theorem}
\newtheorem{lemma}[theorem]{Lemma}

\newtheorem{conjecture}[theorem]{Conjecture}

\newtheorem{corollary}[theorem]{Corollary}
\newtheorem{proposition}[theorem]{Proposition}
\theoremstyle{definition}
\newtheorem{definition}[theorem]{Definition}
\newtheorem{example}[theorem]{Example}
\newtheorem{observation}[theorem]{Observation}
\newtheorem{Remark}[theorem]{Remark}

\DeclareMathOperator{\CFT}{\widehat{CFT}}
\DeclareMathOperator{\HFT}{\widehat{HFT}}
\DeclareMathOperator{\BSD}{\widehat{BSD}}
\DeclareMathOperator{\BSA}{\widehat{BSA}}

\DeclareMathOperator{\SFH}{SFH}

\DeclareMathOperator{\HFL}{\widehat{HFL}}

\DeclareMathOperator{\pip}{\pi^\partial_2}
\DeclareMathOperator{\Spinc}{Spin^\mathit{c}} 
\DeclareMathOperator{\Gru}{\underline{Gr}}

\DeclareMathOperator{\gru}{\underline{gr}}
\DeclareMathOperator{\gr}{gr}

\DeclareMathOperator{\lk}{lk} 
\DeclareMathOperator{\m}{m} 
\DeclareMathOperator{\rr}{r} 
 
\DeclareMathOperator{\im}{im}

\newcommand{\A}{\boldsymbol{\alpha}}
\newcommand{\B}{\boldsymbol{\beta}}
\newcommand{\Ac}{\boldsymbol{\alpha}^c}
\newcommand{\Aa}{\operatorname{\boldsymbol{\alpha}}^a}
\newcommand{\aaa}{\operatorname{\alpha}^a}

\newcommand{\x}{\boldsymbol{x}}
\newcommand{\y}{\boldsymbol{y}}
\newcommand{\z}{\boldsymbol{z}}

\def\co{\colon\thinspace}

\begin{document}
	\psset{arrowsize=3pt 2}
	
	\begin{abstract}
		We define polynomial tangle invariants~$\nabla_T^s$ via Kauffman states and Alexander codes and investigate some of their properties. In particular, we prove symmetry relations for~$\nabla_T^s$ of 4-ended tangles and deduce that the multivariable Alexander polynomial is invariant under Conway mutation.
		The invariants~$\nabla_T^s$ can be interpreted naturally via Heegaard diagrams for tangles. This leads to a categorified version of~$\nabla_T^s$: a Heegaard Floer homology~$\HFT$ for tangles, which we define as a bordered sutured invariant. We discuss a bigrading on~$\HFT$ and prove symmetry relations for~$\HFT$ of 4-ended tangles that echo those for~$\nabla_T^s$. 
	\end{abstract}
	\maketitle
	
	%%%%%%%%%%%%%%%%%%%%   Start of main body of article
	
	% Intro.tex
\section*{Introduction}\label{sec:intro}\pdfbookmark[section]{Introduction}{intro}
Let $L$ be a link in the 3-sphere~$S^3$. Consider an embedded closed 3-ball~$B^3\subset S^3$ whose boundary intersects~$L$ transversely. Then, modulo a parametrization of the boundary $\partial B^3$, the embedding $L\cap B^3\hookrightarrow B^3$ is essentially what we call a tangle, see definition~\ref{def:tangle}. 
The first half of this paper is concerned with the definition and study of polynomial invariants~$\nabla_T^s$ for such tangles. These invariants should be viewed as a generalisation of the Conway potential function, ie the normalised Alexander polynomial, which is a classical knot and link invariant \cite{Alexander}. Using Heegaard diagrams for tangles, we interpret $\nabla_T^s$ as the graded Euler characteristics of some homological Heegaard-Floer-type invariants~$\HFT(T,s)$, which we define in the second half of this paper.

\subsection*{The polynomial tangle invariants $\nabla_T^s$} We start from Kauffman's combinatorial definition of the Alexander polynomial \cite{Kauffman} and adapt it to tangles. In general, from an oriented tangle diagram~$T$, we obtain a finite set of Laurent polynomials $\nabla_T^s$ in the same number of variables as there are tangle components. This finite set of invariants is indexed by some additional input data for tangles, which we call the sites $s$ of~$T$. For example, a tangle diagram with four ends, such as the one in figure~\ref{fig:INTRO2m3pt}, has four sites, one for each open region of the diagram. For the general definition of sites, see definition~\ref{def:basic}.

\begin{theorem}[\ref{thm:nablaisaninvariant} and~\ref{thm:twoended}]
	For each site \(s\) of an oriented tangle \(T\), \(\nabla_T^s\) is an invariant of \(T\). Furthermore, if \(T\) is a tangle with two ends, there is exactly one site \(s\), namely \(s=\emptyset\), and \(\nabla^\emptyset_T\) is (up to some factor) equal to the Conway potential function \(\nabla_L\) of the link or knot \(L\) obtained by joining the two ends of \(T\).
	\end{theorem}

\begin{wrapfigure}{r}{0.3333\textwidth}
	\centering
	\vspace*{-10pt}
	{\psset{unit=0.55}
		\begin{pspicture}[showgrid=false](-4.2,-3.1)(2.2,3.1)
		\psecurve(-2.5,1.5)(0,2)(0.75,1)(-0.75,-1)(0,-2)(0.97,-2.24)(2,-2)
		\psecurve(2,2)(0.97,2.24)(0,2)(-0.75,1)(0.75,-1)(0,-2)(-2.5,-1.5)(-3.25,0)(-2.5,1.5)(0,2)(0.75,1)
		\psecurve(-6,1.5)(-3.3,1.85)(-2.5,1.5)(-1.85,0)(-2.5,-1.5)(-3.3,-1.85)(-6,-1.5)
		\pscircle*[linecolor=white](-2.5,1.5){0.2}
		
		\psecurve(0.75,-1)(0,-2)(-2.5,-1.5)(-3.25,0)(-2.5,1.5)
		
		\pscircle*[linecolor=white](0,2){0.2}
		\pscircle*[linecolor=white](0,0){0.2}
		\pscircle*[linecolor=white](0,-2){0.2}
		
		\psecurve(0.75,1)(-0.75,-1)(0,-2)(0.97,-2.24)(2,-2)
		\psecurve(0,2)(-0.75,1)(0.75,-1)(0,-2)
		\psecurve(-2.5,-1.5)(-3.25,0)(-2.5,1.5)(0,2)(0.75,1)(-0.75,-1)
		
		\pscircle*[linecolor=white](-2.5,-1.5){0.2}
		\psecurve(-2.5,1.5)(-1.85,0)(-2.5,-1.5)(-3.3,-1.85)(-6,-1.5)
		
		\psline(-3.25,-0.1)(-3.25,0.1)
		\pscircle[linestyle=dotted](-1,0){3.05}
		
		\uput{2.4}[180](-1,0){$a$}
		\uput{2.3}[-90](-1,0){$b$}
		\uput{2.1}[0](-1,0){$c$}
		\uput{2.3}[90](-1,0){$d$}
		\end{pspicture}
	\caption{A diagram of the $(2,-3)$-pretzel tangle with its four open regions $a$, $b$, $c$ and~$d$ corresponding to its sites}\label{fig:INTRO2m3pt}}
	\bigskip
	\begin{pspicture}(-2.1,-1.5)(2.1,1.05)
	\rput(-1.1,0){
		\SpecialCoor
		\psline{<-}(0.6;45)(1;45)
		\psline{->}(0.6;-45)(1;-45)
		
		\psline{->}(0.6;135)(1;135)
		\psline{<-}(0.6;-135)(1;-135)

		\uput{0.6}[0](0,0){$c$}
		\uput{0.6}[90](0,0){$d$}
		\uput{0.6}[180](0,0){$a$}
		\uput{0.6}[270](0,0){$b$}
		
		\pscircle[linestyle=dotted](0,0){1}
		
		\rput(0,-1.3){\textit{type 1}}
	}
	\rput(1.1,0){
		\psline{->}(0.6;45)(1;45)
		\psline{<-}(0.6;-135)(1;-135)
		
		\psline{->}(0.6;135)(1;135)
		\psline{<-}(0.6;-45)(1;-45)
		
		\uput{0.6}[0](0,0){$c$}
		\uput{0.6}[90](0,0){$d$}
		\uput{0.6}[180](0,0){$a$}
		\uput{0.6}[270](0,0){$b$}
		\pscircle[linestyle=dotted](0,0){1}
		\rput(0,-1.3){\textit{type 2}}
	}
	\end{pspicture}
	\caption{Two orientations on a 4-ended tangle}\label{fig:INTROorientations}
	\bigskip
	{\psset{unit=0.25}
		\begin{pspicture}(-8.01,-3.01)(8.01,3.01)
		\rput(-5,0){
			\pscircle[linestyle=dotted](0,0){3}
			\psline(3;45)(3;-135)
			\psline(3;-45)(3;135)
			\pscircle[fillcolor=white,fillstyle=solid](0,0){1.5}
			\rput(0,0){$R$}
		}
		\rput(0,0){$\longrightarrow$}
		\rput(5,0){
			\pscircle[linestyle=dotted](0,0){3}
			\psline(3;45)(3;-135)
			\psline(3;-45)(3;135)
			\pscircle[fillcolor=white,fillstyle=solid](0,0){1.5}
			\psrotate(0,0){180}{\rput(0,0){$R$}}
		}
		\end{pspicture}
		\caption{Conway mutation}\label{fig:MutationIntro}}
\end{wrapfigure}

The definition of $\nabla_T^s$ is a very straightforward generalisation of Kauffman's construction. However, I am unaware of any reference in the literature where these invariants have been studied. The underlying idea is perhaps similar to the one in \cite{ouka}, where state polynomials of tangle universes are introduced as a tool for proving the duality conjecture in formal knot theory.

Thus, studying the basic properties of $\nabla_T^s$ will be the first objective of this paper. We show that the invariants $\nabla_T^s$ satisfy glueing formulas (propositions~\ref{prop:annularglueing} and~\ref{prop:glueing}) which generalise the connected sum formula for knots and links. In section~\ref{sec:basicpropertiesofnabla}, we show that our tangle invariants also enjoy some other basic properties, similar to those of the Conway potential function. In particular, the invariants $\nabla_T^s$ are well-behaved under orientation reversal of the tangle strands and taking mirror images. In section \ref{sec:4endedandmutation}, we study $\nabla_T^s$ for 4-ended tangles and prove the following symmetry relations between different sites.

\begin{theorem}[\ref{thm:fourendedonecolour}]\label{thm:INTROfourendedonecolour}
	Let \(T\) be an oriented 4-ended tangle. Then up to cyclic permutation of the sites, it carries one of the two types of orientations shown in figure~\ref{fig:INTROorientations}. Let \(\rr(T)\) denote the same tangle with the opposite orientation on all strands. After identifying the colours of the two open components of \(T\), we have
	$$\text{\(\nabla^{a}_{T}=\nabla^{c}_{\rr(T)}=\nabla^{c}_{T}\)
	and
	\(\nabla^{d}_{T}=\nabla^{b}_{\rr(T)}=\nabla^{b}_{T}\)}$$
	for type 1. For type 2, all identities but the last hold true.
\end{theorem}

\begin{definition}[Conway mutation]\label{def:INTROmutation}
	Given a link $L$, let $L'$ be the link obtained by cutting out a tangle diagram~$R$ with four ends from a diagram of~$L$ and glueing it back in after a half-rotation, see figure~\ref{fig:MutationIntro} for an illustration. We say $L'$ is a \textbf{Conway mutant} of~$L$ and we call~$R$ the \textbf{mutating tangle} in this mutation. If $L$ is oriented, we choose an orientation of $L'$ that agrees with the one for~$L$ outside of~$R$. If this means that we need to reverse the orientation of the two open components of~$R$, then we also reverse the orientation of all other components of~$R$ during the mutation; otherwise we do not change any orientation. For an alternative, but equivalent definition, see definition~\ref{def:mutation} and remark~\ref{rem:defmut}.
\end{definition}

Theorem~\ref{thm:INTROfourendedonecolour} along with the glueing formula for $\nabla_T^s$ gives rise to the following result.

\begin{corollary}[\ref{cor:mutation}]\label{cor:IntroMutationDECAT}
	The multivariate Alexander polynomial is invariant under Conway mutation after identifying the variables corresponding to the two open strands of the mutating tangle.
\end{corollary}
This result has long been known for the univariate Alexander polynomial, see for example~\cite[proposition~11]{LickorishMillett}, but I have been unable to find a corresponding result for the multivariate polynomial in the literature. The fact that mutation invariance follows so easily from the symmetry relations of theorem~\ref{thm:INTROfourendedonecolour} suggests that $\nabla_T$ is well-suited for studying the ``local behaviour'' of the Alexander polynomial.

\subsection*{The homological tangle invariant $\HFT$}
Heegaard Floer homology theories were first defined by Ozsváth and Szabó in~2001 \cite{OSHF3mfds}. With an oriented, closed 3-dimensional manifold~$M$, they associated a family of homological invariants, the simplest of which is denoted by~$\widehat{\text{HF}}(M)$. Given an oriented (null-homologous) knot or link $L$ in~$M$, Ozsváth and Szabó, and independently J.\,Rasmussen, then defined filtrations on the chain complexes which give rise to the respective flavours of knot and link Floer homology \cite{OSHFK,Jake,OSHFL}, the simplest of which is denoted by~$\widehat{\text{HFL}}(L)$. The Alexander polynomial can be recovered from these groups as the graded Euler characteristic.

Given corollary~\ref{cor:IntroMutationDECAT}, it is only natural to ask for a Heegaard-Floer theoretic categorification of~$\nabla_T^s$. To this end, we define a homology theory $\HFT$ as follows: given a Heegaard diagram for a tangle $T$ (see definition~\ref{def:HDsfortangles}) along with a site $s$ of $T$, we define a finitely generated Abelian group which comes with two gradings: a relative homological $\mathbb{Z}$-grading and an Alexander grading, which is an additional relative $\mathbb{Z}$-grading for each component of the tangle:
$$\CFT(T,s)=\bigoplus_{\substack{h\in\mathbb{Z}~\leftarrow \text{homological grading}\hspace{-2.96cm}\\ a\in\mathbb{Z}^{\vert T\vert}~\leftarrow \text{Alexander  grading}\hspace{-2.7cm}}}\CFT_h(T,s,a).$$
Here, $\vert T\vert$ denotes the number of components of $T$. 
One can then define a differential on this group which preserves the Alexander grading and decreases homological grading by 1. In sections~\ref{sec:HDsForTangles} and~\ref{sec:Gradings}, we prove the following result.  
\begin{theorem}[\ref{thm:HFTiswelldefandinvariant} and \ref{thm:Eulercharagreeswithnabla}]\label{thm:IntroHFT}
	Given an oriented tangle \(T\) and a site \(s\) for \(T\), the bigraded chain homotopy type of \(\CFT(T,s)\) is an invariant of \(T\). We denote its homology by \(\HFT(T,s)\) and call it the \textbf{non-glueable Heegaard Floer homology of  the tangle \(T\) with respect to the site $s$}. Its graded Euler characteristic 
	$$
	\chi(\HFT(T,s))=\sum_{h, a} (-1)^h\operatorname{rk}(\HFT_h(T,s,a))\cdot t_{1\phantom{\vert}}^{a_{1\phantom{\vert}}}\!\!\cdots t_{\vert T\vert}^{a_{\vert T\vert}}\in\mathbb{Z}[t_1^{\pm1},\dots,t_{\vert  T\vert}^{\pm1}]
	$$
	is well-defined up to multiplication by a unit and agrees with $\nabla_T^s$ up to some factor.
\end{theorem}
Actually, we define $\HFT$ for tangles within arbitrary 3-manifolds $M$ with spherical boundary, see the comment at the beginning of section~\ref{sec:HDsForTangles}. This is done using Zarev's bordered sutured Heegaard Floer theory~\cite{ZarevThesis}. However, in general, the gradings on his invariants are rather complicated. To obtain $\HFT$ with the gradings described above, we restrict ourselves to tangles inside $\mathbb{Z}$-homology 3-balls $M$. In this case, the first two parts of the theorem above still hold, whereas the final part could be viewed as a definition of $\nabla_T^s$ for tangles $T$ in $M\neq B^3$.

In \cite{Juhasz}, Juhász defined sutured Floer homology $\SFH$, a Heegaard Floer homology for balanced sutured manifolds, certain 3-manifolds with non-empty boundaries which carry some additional structures, so-called sutures (see definition~\ref{def:SuturedManifold}). 
We show in theorem~\ref{thm:HFTasSFT} that for any fixed site $s$, we can identify $\HFT(T,s)$ with the sutured Floer homology $\SFH$ of the tangle complement with a particular choice of sutures which depends on~$s$. This can be regarded as the analogue of the fact that link Floer homology $\HFL$ can be computed as the sutured Floer homology of the link complement with meridional sutures \cite[proposition~9.2]{Juhasz}. Moreover, by work of Friedl, Juhász and Rasmussen~\cite{DecatSFH}, the graded Euler characteristic of $\SFH$ coincides with sutured Turaev torsion. Thus, we obtain a geometric interpretation of sites and the invariants~$\nabla_T^s$ themselves. Note that this interpretation can also be retraced directly, without referring to any categorified invariants, see~\cite[section~I.4]{thesis}.

\subsection*{Towards $\delta$-graded mutation invariance of $\HFL$}
We know from~\cite{OSmutation} that knot and link Floer homology is, in general, not invariant under mutation. However, Baldwin and Levine conjectured the following \cite[conjecture~1.5]{BaldwinLevine}.
\begin{conjecture}\label{conj:MutInvHFL}
	Let \(L\) be a link and let \(L'\) be obtained from \(L\) by Conway mutation. Then \(\HFL(L)\) and \(\HFL(L')\) agree after collapsing the bigrading to a single \(\mathbb{Z}\)-grading, known as the \(\delta\)-grading. In short: \(\delta\)-graded link Floer homology is mutation invariant.
\end{conjecture} 
Following the strategy for proving mutation invariance for the polynomial invariants, we study symmetry relations for the categorified invariants of 4-ended tangles. This is were the interpretation of $\HFT$ in terms of sutured Floer homology $\SFH$ becomes very useful. Using Juhász's surface decomposition formula~\cite[proposition~8.6]{SurfaceDecomposition}, we show the following.

\begin{theorem}[\ref{thm:fourendedHFT} and \ref{rem:SymRelCAT}]\label{thm:INTROfourendedHFT}
	Let \(T\) be an oriented 4-ended tangle and let \(\rr(T)\) denote the same tangle with the opposite orientation on all strands. Then, after collapsing the Alexander gradings of the two open components of \(T\), we have
	$$
	\CFT(T,a)\cong\CFT(\rr(T),c)\quad\text{and}\quad \CFT(T,b)\cong\CFT(\rr(T),d),
	$$
	as (relatively) bigraded invariants.
\end{theorem}
As one can see, the symmetry relations for $\HFT$ are not quite as strong as those for~$\nabla_T^s$. In fact, in example~\ref{exa:pretzeltangle} we show that, in general, the stronger relations do not hold. This offers a satisfying explanation of \textit{why} bigraded link Floer homology fails to be mutation invariant, while at the same time giving further evidence towards conjecture~\ref{conj:MutInvHFL}. 

Unfortunately, the symmetry relations for $\HFT$ are not enough to prove the conjecture. This is because, unlike $\nabla_T^s$, $\HFT$ \emph{alone} is insufficient to state a glueing formula. In general, $\HFT$ can be upgraded to a glueable theory by modifying its differential. This approach is described in my PhD thesis \cite[section~II.3]{thesis} and uses more complicated arc diagrams from Zarev's bordered sutured theory~\cite{ZarevThesis}. For 4-ended tangles, one can define a slightly different glueing structure, which turns out to be very similar to Hanselman, J.\,Rasmussen and Watson's immersed curve invariant for 3-manifolds with torus boundary~\cite{HRW}. This approach is described in~\cite{pqMod}, building on~\cite[chapter~III]{thesis}. 

\subsection*{Similar work by other people. }
It is interesting to compare the ideas described in this paper to those of several other groups of people who have defined generalisations of the Alexander polynomial or its categorification via Heegaard Floer theory to tangles. 

As a classical invariant of knots and links, the Alexander polynomial can be defined and interpreted in a number of different ways, depending on one’s preferred point of view. Many of these different interpretations have been used as starting points for generalisations of the Alexander polynomial to tangles: For example, Polyak \cite{Polyak} uses skein theory to define his invariant; Bigelow \cite{Bigelow} and Kennedy \cite{Kennedy} adopt a diagrammatic approach; Sartori \cite{Sartori14} uses representation theory; Archibald \cite{Archibald}, Bigelow-Cattabriga-Florens \cite{Florens} and Damiani-Florens \cite{Damiani} work with suitable generalisations of Alexander matrices. $\nabla_T^s$ from this paper fits into this collection of invariants, as it is based on the purely combinatorial definition of the classical Alexander polynomial via Kauffman states and Alexander codes. In my thesis \cite{thesis}, I explain yet another definition of a polynomial tangle invariant, namely in terms of the maximal Abelian cover of the tangle complement. Up to normalisation, this invariant can be identified with $\nabla_T^s$, so it offers a very natural geometric interpretation of $\nabla_T^s$. I refer the interested reader to~\cite[section~I.4]{thesis}, where I also discuss how one might be able to use this point of view to relate $\nabla_T^s$ to some of the other tangle invariants mentioned above. 

In 2014, Petkova and Vértesi defined a combinatorial tangle Floer homology using grid diagrams and ideas from bordered Floer homology \cite{cHFT}. They use a more general definition of tangles, namely two-sided ones. In \cite{DecatCTFH}, they and Ellis show that the decategorification of their invariant agrees with Sartori's generalisation of the Alexander polynomials to two-sided tangles via the representation theory of $U_q(\mathfrak{gl}(1\vert 1))$ \cite{Sartori14}. Thus, Petkova and Vértesi's theory fits nicely into the Reshetikhin-Turaev framework~\cite{ReshetikhinTuraev}, making it analogous to Khovanov's tangle invariant~\cite{KhovanovTangles}. 

In 2016, Ozsváth and Szabó developed a completely algebraically defined knot homology theory, which they conjecture to be equivalent to knot Floer homology \cite{OSKauffmanStates1,OSKauffmanStates2}. Like Petkova and Vértesi, they cut up a knot diagram into elementary pieces, associate with each piece a bimodule and then tensor these bimodules together to obtain a knot invariant. Implicitly, they also define an invariant for two-sided tangles, since their proof of invariance under Reidemeister moves is entirely local. Ozsváth and Szabó's theory seems to be frightfully powerful: from a computational point of view, since they can compute their homology from diagrams with over 50 crossings; but also from a more theoretical point of view, since their theory includes the hat- as well as the more sophisticated ``$-$''-version of knot Floer homology without reference to holomorphic curves or grid diagrams. Interestingly, the generators in their theory correspond to Kauffman states like in ours. A decategorified invariant has been studied by Manion~\cite{ManionDecat} and related to the representation theory of $\mathcal{U}_q(\mathfrak{gl}(1\vert 1))$.

Finally, I want to mention some impressive work of Lambert-Cole \cite{LambertCole1,LambertCole2}, where he confirms conjecture~\ref{conj:MutInvHFL} for various families of mutant pairs.

\subsection*{Acknowledgements}
This paper grew out of an essay for the Smith-Knight \& Rayleigh-Knight Prize Competition 2015 \cite{essay}, which later formed the first two chapters of my PhD thesis~\cite{thesis}. I would therefore like to take the opportunity to thank my PhD supervisor Jake Rasmussen for his generous support. I consider myself very fortunate to have been his student.

My PhD was funded by an EPSRC scholarship covering tuition fees and a DPMMS grant for maintenance, for which I thank the then Head of Department Martin Hyland.

I thank my examiners Ivan Smith and András Juhász for many valuable comments on and corrections to my thesis. I also thank Adam Levine, Andy Manion, Ina Petkova and Vera Vértesi for helpful conversations. My special thanks go to Liam Watson for his continuing interest in my work.

Last but not least, I am very grateful to the anonymous referee for their many valuable and detailed comments on and corrections to an earlier draft of this paper. 

	% Definitions.tex

\section{\texorpdfstring{The polynomial tangle invariants $\nabla_T^s$}{The polynomial tangle invariants ∇}}\label{sec:basicdefinitions}

First of all, we define what we mean by a tangle. Our definition is based on Conway's notion of tangles, see for example \cite[section~2.3]{Adams}.

\begin{definition}\label{def:tangle}
A \textbf{tangle} $T$ is a smooth embedding of a disjoint union of intervals and circles into the closed 3-ball $B^3$,
$$T\co\left(\coprod I \amalg \coprod S^1,\partial\right)\hookrightarrow \left(B^3,{\red S^1}\subset \partial B^3\right),$$
such that the endpoints of the intervals lie on a fixed smoothly embedded circle ${\red S^1}$ on the boundary of~$B^3$, together with a labelling of the arcs ${\red S^1}\smallsetminus \im(T)$ by some index set $\{a,b,c,\dots\}$.
We consider tangles up to ambient isotopy of $T\cup{\red S^1}$ which keeps track of the labelling of the arcs. If the number of intervals is $n$, we call a tangle \textbf{$\boldsymbol{2n}$-ended}. The images of the intervals are called \textbf{open components}, the images of the circles are called \textbf{closed components}. An \textbf{oriented} tangle is a tangle with a choice of orientation on the tangle components. 

Throughout this paper, we will often implicitly fix an ordering of the tangle components and label them by variables $t_1,t_2,\dots$, which we call the \textbf{colours} of~$T$. In a few cases, where no ordering is needed, we will also use the variables $t$, $p$ and $q$ as colours. Unless explicitly stated otherwise, we will assume the colours of different components are distinct.

In analogy to link diagrams, we define a \textbf{tangle diagram} to be a smooth embedding $D$ of a graph whose vertices are either 1- or 4-valent into the closed 2-disc $D^2$ such that the preimage of $\partial D^2$ is exactly the set of 1-valent vertices, together with under/over information at the image of each 4-valent vertex, called a \textbf{crossing}, and a labelling of the arcs $\partial D^2\smallsetminus\im(D)$ by some index set $\{a,b,c,\dots\}$. Just as in the case of links, we consider tangle diagrams up to ambient isotopy (preserving the arc labelling) and the usual Reidemeister moves, see for example~\cite{Lickorish}. Connected components of the complement of the image of $D$ are called \textbf{regions}. Those regions that meet $\partial D^2$ are called \textbf{open}, the others are called \textbf{closed}. We call a diagram \textbf{connected} if the intersection of each open region with $\partial D^2$ is connected.
\end{definition}

\begin{Remark}\label{rem:RMmovesconnectdiagrams}
Regard $D^2$ as the intersection of $B^3$ with the plane $\{z=0\}$. Given a tangle diagram, we can obtain a tangle by pushing the two components at the image of each 4-valent vertex into $\{z>0\}$ and $\{z<0\}$, according to the under/over information. Conversely, given a tangle $T$, we can choose an embedded disc $D^2$ bounding the fixed circle ${\red S^1}$. Then, just as in the case of links, a generic projection of $B^3$ onto this disc gives rise to a well-defined tangle diagram, and any two of these are connected by a sequence of Reidemeister moves. 
\begin{figure}[b]
	\centering
	\psset{unit=0.5}
	\begin{subfigure}[b]{0.18\textwidth}\centering
		\begin{pspicture}(-2,-2)(2,2)
		\rput(0,-0.9){
			\pscircle[linestyle=dotted](0,0.9){2}
			\psecurve(1.2,-1.9)(0.9,-0.9)(0.9,2.7)(1.2,3.7)
			\psecurve(-1.2,-1.9)(-0.9,-0.9)(-0.9,2.7)(-1.2,3.7)
		}
		\rput(1.45;180){$a$}
		\rput(1.45;-90){$b$}
		\rput(1.45;360){$c$}
		\rput(1.45;450){$d$}
		\end{pspicture}
		\caption{}
		\label{figsub:rattangle0a}
	\end{subfigure}
	~
	\begin{subfigure}[b]{0.18\textwidth}\centering
		\begin{pspicture}(-2,-2)(2,2)
		\rput(0,-0.9){
			\pscircle[linestyle=dotted](0,0.9){2}
			\psline(-0.9,-0.9)(0.9,0.9)(-0.9,2.7)
			\pscircle*[linecolor=white](0,1.8){0.3}
			\pscircle*[linecolor=white](0,0){0.3}
			\psline(0.9,-0.9)(-0.9,0.9)(0.9,2.7)
		}
		\rput(1.45;180){$a$}
		\rput(1.5;-90){$b$}
		\rput(1.45;360){$c$}
		\rput(1.5;450){$d$}
		\end{pspicture}
		\caption{}
		\label{figsub:rattangle0b}
	\end{subfigure}
	~
	\begin{subfigure}[b]{0.18\textwidth}\centering
		\begin{pspicture}(-2,-2)(2,2)
		\rput(0,-0.9){
			\pscircle[linestyle=dotted](0,0.9){2}
			
			\psline(0.9,0.9)(-0.9,2.7)
			\pscircle*[linecolor=white](0,1.8){0.3}
			\psline(0.9,-0.9)(-0.9,0.9)(0.9,2.7)
			\pscircle*[linecolor=white](0,0){0.3}
			\psline(-0.9,-0.9)(0.9,0.9)(0.5,1.3)
		}
		\rput(1.45;180){$a$}
		\rput(1.5;-90){$b$}
		\rput(1.45;360){$c$}
		\rput(1.5;450){$d$}
		\end{pspicture}
		\caption{}
		\label{figsub:rattangle2}
	\end{subfigure}
	~
	\begin{subfigure}[b]{0.18\textwidth}\centering
		\begin{pspicture}(-2,-2)(2,2)
		\rput(-0.8,-0.65){
			\pscircle[linestyle=dotted](0.8,0.65){2}
			\psline(0.35,0.35)(0.65,0.65)(-0.65,1.95)
			\pscircle*[linecolor=white](0,1.3){0.3}
			\psline(2.65,1.35)(0.65,-0.65)(-0.65,0.65)(0.65,1.95)
			\pscircle*[linecolor=white](0,0){0.3}
			\psline(-0.65,-0.65)(0.65,0.65)
			\pscircle*[linecolor=white](1.95,0.65){0.3}
			\psline(0.35,1.65)(0.65,1.95)(2.65,-0.05)
		}
		\rput(1.55;200){$a$}
		\rput(1.45;-70){$b$}
		\rput(1.6;360){$c$}
		\rput(1.45;430){$d$}
		\end{pspicture}
		\caption{}\label{figsub:rattangle21}
	\end{subfigure}
	~
	\begin{subfigure}[b]{0.18\textwidth}\centering
		\begin{pspicture}(-2,-2)(2,2)
		\psrotate(0,0){90}{\rput(0,-0.9){
				\pscircle[linestyle=dotted](0,0.9){2}
				\psecurve(1.2,-1.9)(0.9,-0.9)(0.9,2.7)(1.2,3.7)
				\psecurve(-1.2,-1.9)(-0.9,-0.9)(-0.9,2.7)(-1.2,3.7)
			}}
			\rput(1.45;180){$a$}
			\rput(1.45;-90){$b$}
			\rput(1.45;360){$c$}
			\rput(1.45;450){$d$}
			\end{pspicture}
			\caption{}
			\label{figsub:rattangle0e}
		\end{subfigure}
		\caption{Some diagrams of rational tangles. (a) and (b) represent the same tangle, but only (b) is a connected diagram. (c) and (d) show some more complicated rational tangles. (e) does not represent the same tangle as (a), since the labelling is different.}\label{fig:rattangles}
	\end{figure}
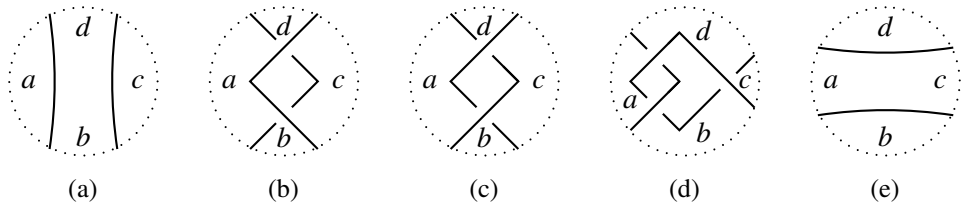
\end{Remark} 

\begin{definition}
\textbf{Rational tangles} are 4-ended tangles without any closed components obtained from the 4-ended tangle in figure \ref{figsub:rattangle0a} by repeatedly adding twists to the top and to the right.
\end{definition}

\subsection*{Alexander polynomials of knots and links. }
Next, let us recall how the Alexander polynomial of knots and links can be computed using Kauffman states and Alexander codes, following~\cite{Kauffman}. Given a diagram of a 2-ended tangle (whose closure represents a knot or link), a Kauffman state is an assignment of a marker~$\bullet$ to one of the four regions at each crossing such that each closed region is occupied by exactly one marker. 
One then applies the Alexander codes to the Kauffman states, ie one labels the markers by the monomials specified by the Alexander codes, as shown in figure~\ref{Kauffknot}. To get the multivariate Alexander polynomial, one just multiplies these labels, takes the sum over all Kauffman states and finally multiplies everything by some normalisation factor.

When trying to apply this well-known algorithm to the general case of a $2n$-ended tangle, one encounters the following problem: Say, there are $m$ crossings in the diagram. Then by an Euler characteristic argument, the diagram consists of at least $(m+n+1)$ regions, so there are at least $(n+1)$ regions more than there are markers. Moreover, we have exactly $(m+n+1)$ regions iff all regions are simply connected, so in this case, the difference between the number of regions and crossings is exactly $(n+1)$. This motivates the following definition.

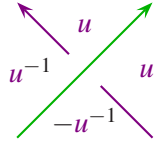
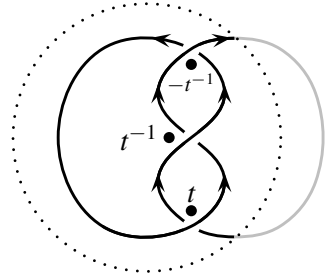
\begin{figure}[t]
	\centering
	\begin{subfigure}[b]{0.45\textwidth}\centering
		\psset{unit=1}
		\begin{pspicture}(-1.3,-1.3)(1.3,1.3)
		\psline[linecolor=violet]{->}(0.9,-0.9)(-0.9,0.9)
		\pscircle*[linecolor=white](0,0){0.3}
		\psline[linecolor=darkgreen]{->}(-0.9,-0.9)(0.9,0.9)
		\uput{0.5}[90](0,0){$\textcolor{violet}{u}$}
		\uput{0.4}[180](0,0){$\textcolor{violet}{u}^{-1}$}
		\uput{0.5}[-90](0,0){$-\textcolor{violet}{u}^{-1}$}
		\uput{0.7}[0](0,0){$\textcolor{violet}{u}$}
		\end{pspicture}
		\caption{The Alexander code from \cite[figure~33]{Kauffman}  for a positive crossing. There is also a similar one for a negative crossing. $\textcolor{violet}{u}$ is the colour of the under-strand.}
		\label{AlexCode}
	\end{subfigure}
	\quad
	\begin{subfigure}[b]{0.45\textwidth}\centering
		\psset{unit=0.59, linewidth=1.1pt}
		\begin{pspicture}[showgrid=false](-4.2,-3.1)(3.2,3.1)
		\psecurve(-2,2)(0,2)(0.75,1)(-0.75,-1)(0,-2)(0.97,-2.24)(2,-2)
		\psecurve(2,2)(0.97,2.24)(0,2)(-0.75,1)(0.75,-1)(0,-2)(-2,-2)(-3,0)(-2,2)(0,2)(0.75,1)
		
		\psecurve[linecolor=lightgray](0,-2)(0.97,-2.24)(2,-2)(2.8,-1)(2.8,1)(2,2)(0.97,2.24)(0,2)
		
		\pscircle*[linecolor=white](0,2){0.2}
		\pscircle*[linecolor=white](0,0){0.2}
		\pscircle*[linecolor=white](0,-2){0.2}
		
		\psecurve(0,2)(0.75,1)(-0.75,-1)(0,-2)
		\psecurve(2,2)(0.97,2.24)(0,2)(-0.75,1)(0.75,-1)
		\psecurve(-0.75,1)(0.75,-1)(0,-2)(-2,-2)(-3,0)
		
		\psline{->}(0.75,1.1)(0.75,1.2)
		\psline{->}(-0.75,1.11)(-0.75,1.2)
		\psline{->}(-0.75,-0.9)(-0.75,-0.8)
		\psline{->}(0.75,-0.9)(0.75,-0.8)
		
		\pscircle[linestyle=dotted](-1,0){3.05}
		
		\psline{->}(0.8,2.235)(0.9,2.25)
		\psline{->}(-0.8,2.235)(-0.9,2.25)

		\psdot(0,1.65)
		\psdot(-0.5,0)
		\psdot(0,-1.65)
		
		\uput{0.2}[-90](0,1.65){\scriptsize $-t^{-1}$}
		\uput{0.2}[180](-0.5,0){$t^{-1}$}
		\uput{0.2}[90](0,-1.65){$t$}
		
		\end{pspicture}
		\caption{A labelled Kauffman state for a 2-ended tangle, coloured by~$t$. The closure is indicated by the grey arc.}
		\label{Kauffknot}
	\end{subfigure}
	\caption{Applying Alexander codes to Kauffman states of knots and links}
\end{figure}

\begin{definition}\label{def:basic} 
Let $D$ be a diagram of an oriented $2n$-ended tangle $T$.

\begin{itemize}
	\item A \textbf{site} $s$ of $T$, or $D$, is a choice of an $(n-1)$-element subset of the set of arcs ${\red S^1}\smallsetminus \im(T)$, or equivalently $\partial D^2\smallsetminus \im(D)$. For connected tangle diagrams, this is equivalent to choosing $(n-1)$ open regions. The set of all sites of a tangle $T$ is denoted by $\mathbb{S}(T)$.
	\item A \textbf{Kauffman state} of $D$ is an assignment of a marker to one of the four regions at each crossing such that each closed region is occupied by exactly one marker, with the additional condition that there be at most one marker in each open region. Note that for a diagram without any crossings, the empty assignment is also a Kauffman state, provided that there are no closed regions. Let us denote the set of all Kauffman states of $D$ by $\mathbb{K}(D)$.
	\item Given a Kauffman state $x\in\mathbb{K}(D)$, we can construct a site $s\in\mathbb{S}(T)$ from the set of those arcs of $\partial D^2\smallsetminus \im(D)$ which lie in open regions occupied by markers of $x$ by adding for each unoccupied open region of $D$ all but one arc which lies in that region. This is indeed an $(n-1)$ element subset, since the number of unoccupied (and therefore open) regions is exactly $(n+1)$. We say that a Kauffman state $x\in\mathbb{K}(D)$ and a site $s$ obtained in this way from $x$ belong to each other. We write $\mathbb{K}(D,s)$ for the set of all Kauffman states belonging to $s$. 
	In particular, for connected diagrams $D$, any Kauffman state $x$ belongs to exactly one site, since any open region only contains a single arc.   
	\item For $x\in\mathbb{K}(D)$, let $c(x)$ be the product of the labels of the markers of $x$ according to the Alexander codes in figure~\ref{figAlexCodesForNabla}, with the convention that the empty product is equal to 1. Then for each site $s\in\mathbb{S}(T)$, let $$\hat{\nabla}_D^s:=\sum_{x\in\mathbb{K}(D,s)}c(x).$$
	Furthermore, let $\nabla_D^s$ denote the function $\hat{\nabla}_D^s$ evaluated at $h=-1$. 
\end{itemize}
\end{definition}

\begin{figure}[t]
	\centering
	\psset{unit=1.5}
	\begin{subfigure}[b]{0.3\textwidth}\centering
		\begin{pspicture}(-1.05,-1.05)(1.05,1.05)
		\psline[linecolor=violet]{->}(0.9,-0.9)(-0.9,0.9)
		\pscircle*[linecolor=white](0,0){0.3}
		\psline[linecolor=darkgreen]{->}(-0.9,-0.9)(0.9,0.9)
		\uput{0.7}[90](0,0){$\textcolor{darkgreen}{o}^{\frac{1}{2}} \textcolor{violet}{u}^{\frac{1}{2}}$}
		\uput{0.3}[180](0,0){$\textcolor{darkgreen}{o}^{\frac{1}{2}} \textcolor{violet}{u}^{-\frac{1}{2}}$}
		\uput{0.7}[-90](0,0){$h^{-1}\textcolor{darkgreen}{o}^{-\frac{1}{2}} \textcolor{violet}{u}^{-\frac{1}{2}}$}
		\uput{0.3}[0](0,0){$\textcolor{darkgreen}{o}^{-\frac{1}{2}} \textcolor{violet}{u}^{\frac{1}{2}}$}
		\end{pspicture}
		\caption{A positive crossing}
	\end{subfigure}
	\begin{subfigure}[b]{0.3\textwidth}\centering
		\begin{pspicture}(-1.05,-1.05)(1.05,1.05)
		\psline[linecolor=violet]{->}(-0.9,-0.9)(0.9,0.9)
		\pscircle*[linecolor=white](0,0){0.3}
		
		\psline[linecolor=darkgreen]{->}(0.9,-0.9)(-0.9,0.9)
		\uput{0.7}[90](0,0){$\textcolor{darkgreen}{o}^{-\frac{1}{2}} \textcolor{violet}{u}^{-\frac{1}{2}}$}
		\uput{0.3}[180](0,0){$\textcolor{darkgreen}{o}^{\frac{1}{2}} \textcolor{violet}{u}^{-\frac{1}{2}}$}
		\uput{0.7}[-90](0,0){$h\textcolor{darkgreen}{o}^{\frac{1}{2}} \textcolor{violet}{u}^{\frac{1}{2}}$}
		\uput{0.3}[0](0,0){$\textcolor{darkgreen}{o}^{-\frac{1}{2}} \textcolor{violet}{u}^{\frac{1}{2}}$}
		\end{pspicture}
		\caption{A negative crossing}
	\end{subfigure}
	\caption{The Alexander codes for definition \ref{def:basic}. The variable $\textcolor{darkgreen}{o}$ is the colour of the over-strand and $\textcolor{violet}{u}$ the colour of the under-strand.}\label{figAlexCodesForNabla}
\end{figure}
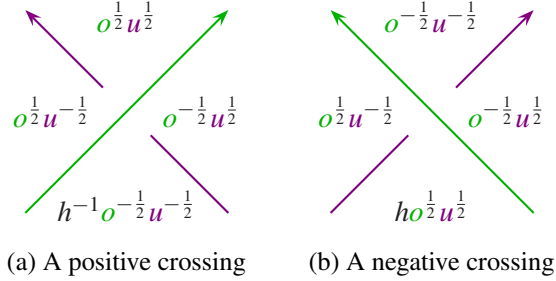 

\begin{example}
	The above definition is illustrated in figure~\ref{fig:Kaufftangle}, which shows a 4-ended tangle with a set of markers defining a Kauffman state belonging to site $c$. The markers are labelled according to the Alexander codes in figure~\ref{figAlexCodesForNabla} with $h=-1$.
\end{example}

\begin{wrapfigure}{r}{0.4\textwidth}
\centering
		\psset{unit=0.8, linewidth=1.1pt}
		\begin{pspicture}[showgrid=false](-4.1,-3.1)(2.1,3.1)
		
		\psecurve[linecolor=darkgreen](-2.5,-1.5)(0,-2)(0.75,-1)(-0.75,1)(0,2)(0.97,2.24)
		\psecurve[linecolor=darkgreen]{<-}(2,-2)(0.97,-2.24)(0,-2)(-0.75,-1)(0.75,1)(0,2)(-2.5,1.5)(-3.25,0)(-2.5,-1.5)(0,-2)(0.75,-1)
		
		\psecurve[linecolor=violet]{->}(-6,-1.5)(-3.3,-1.85)(-2.5,-1.5)(-1.85,0)(-2.5,1.5)(-3.3,1.85)(-6,1.5)
		\psline[linecolor=violet]{->}(-1.85,0.05)(-1.85,0.15)
		\pscircle*[linecolor=white](-2.5,1.5){0.2}
		\psecurve[linecolor=darkgreen](0.75,1)(0,2)(-2.5,1.5)(-3.25,0)(-2.5,-1.5)
		\psecurve[linecolor=darkgreen](0.75,-1)(0,-2)(-2.5,-1.5)(-3.25,0)(-2.5,1.5)
		
		\pscircle*[linecolor=white](0,2){0.2}
		\pscircle*[linecolor=white](0,0){0.2}
		\pscircle*[linecolor=white](0,-2){0.2}
		
		\psecurve[linecolor=darkgreen](0.75,-1)(-0.75,1)(0,2)(0.97,2.24)(2,2)
		\psecurve[linecolor=darkgreen](0,-2)(-0.75,-1)(0.75,1)(0,2)
		\psecurve[linecolor=darkgreen](-2.5,1.5)(-3.25,0)(-2.5,-1.5)(0,-2)(0.75,-1)(-0.75,1)
		
		\pscircle*[linecolor=white](-2.5,-1.5){0.2}
		\psecurve[linecolor=violet](-2.5,1.5)(-1.85,0)(-2.5,-1.5)(-3.3,-1.85)(-6,-1.5)

		\pscircle[linestyle=dotted](-1,0){3.05}
		
		%%\psdot[linecolor=violet](0.97,2.24)
		%\uput{0.2}[45](0.97,2.24){$\textcolor{violet}{q}$}
		%%\psdot[linecolor=violet](0.97,-2.24)
		%\uput{0.2}[-45](0.97,-2.24){$\textcolor{violet}{q}$}
		%%\psdot[linecolor=darkgreen](-3.3,1.85)
		%\uput{0.2}[135](-3.3,1.85){$\textcolor{darkgreen}{p}$}
		%%\psdot[linecolor=darkgreen](-3.3,-1.85)
		%\uput{0.2}[-135](-3.3,-1.85){$\textcolor{darkgreen}{p}$}
		
		%\psecurve[linecolor=violet]{<-}(2,2)(0.97,2.24)(0,2)(-0.75,1)(0.75,-1)(0,-2)(-2.5,-1.5)(-3.25,0)(-2.5,1.5)(0,2)(0.75,1)
		%\psecurve[linecolor=darkgreen]{<-}(-6,1.5)(-3.3,1.85)(-2.5,1.5)(-1.85,0)(-2.5,-1.5)
		
		\psline[linecolor=darkgreen]{->}(0.75,-1.1)(0.75,-1.2)
		\psline[linecolor=darkgreen]{->}(-0.75,-1.1)(-0.75,-1.2)
		\psline[linecolor=darkgreen]{->}(-0.75,0.9)(-0.75,0.8)
		\psline[linecolor=darkgreen]{->}(0.75,0.9)(0.75,0.8)

		\psline[linecolor=darkgreen]{->}(-3.25,0.05)(-3.25,0.15)
		
		\psline[linecolor=darkgreen]{<-}(-1.5,-2.015)(-1.4,-2.04)
		\psline[linecolor=darkgreen]{->}(-1.5,2.015)(-1.4,2.05)
		
		\uput{0.2}[45](0.97,2.24){$\textcolor{darkgreen}{p}$}
		\uput{0.2}[-45](0.97,-2.24){$\textcolor{darkgreen}{p}$}
		\uput{0.2}[135](-3.3,1.85){$\textcolor{violet}{q}$}
		\uput{0.2}[-135](-3.3,-1.85){$\textcolor{violet}{q}$}
		
		\psdot(0,1.65)
		\psdot(0.5,0)
		\psdot(0,-1.65)
		
		\psdot(-2.5,-1)
		\psdot(-2,1.4)
		
		\uput{0.2}[90](0,-1.65){$-\textcolor{darkgreen}{p}^{-1}$}
		\uput{0.2}[-90](0,1.65){$\textcolor{darkgreen}{p}$}
		\uput{0.2}[0](0.5,0){$1$}
		\uput{0.25}[88](-2.5,-1){$\textcolor{darkgreen}{p}^{\frac{1}{2}}\textcolor{violet}{q}^{\frac{1}{2}}$}
		\uput{0.15}[30](-2,1.4){$\textcolor{darkgreen}{p}^{-\frac{1}{2}}\textcolor{violet}{q}^{\frac{1}{2}}$}
		
		\uput{2.5}[180](-1,0){$a$}
		\uput{2.5}[-90](-1,0){$b$}
		\uput{2.5}[0](-1,0){$c$}
		\uput{2.5}[90](-1,0){$d$}
		\end{pspicture}
		\caption{A Kauffman state of a 4-ended tangle}\label{fig:Kaufftangle}
\end{wrapfigure}
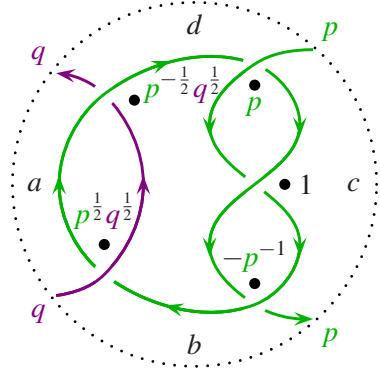
\myfixwrapfig

\begin{Remark}
The variable $h$ stands for ``homological grading''. In sections~\ref{sec:HDsForTangles} and~\ref{sec:Gradings}, we will generalise the hat version of knot and link Floer homology to tangles. The generators of these homology groups will correspond to the generalised Kauffman states above. Also, this perspective offers a more geometric interpretation of sites, see in particular definition~\ref{def:sutured3mfdForTangles}.
\end{Remark}

\begin{observation}\label{ObsAlexanderCode}
	In the Alexander codes of figure~\ref{figAlexCodesForNabla}, the exponents of $u$ in the two regions left of an under-strand are $-\tfrac{1}{2}$, and $+\tfrac{1}{2}$ in the regions on its right. For over-strands, it is the other way round. This Alexander code has the advantage over the one in figure \ref{AlexCode} that we do not need to multiply $\nabla^s_T$ by a normalisation factor to turn it into a tangle invariant.
\end{observation}

\begin{observation}
	If a diagram $D$ contains a region that is not simply connected, ie the corresponding tangle $T$ contains a split component, $\nabla_D^s$ vanishes for all sites $s$, as there are no Kauffman states of $D$. This corresponds to the fact that the Alexander polynomial of a link with a split component also vanishes.
\end{observation}

\begin{Remark}
	Let $D$ be an oriented tangle diagram and $S$ a smoothly embedded circle in the interior of the disc $D^2$ which intersects $\im(D)$ transversely away from any of the crossings. $S$~divides $D^2$ into a smaller disc and an annulus. Then the restriction of $D$ to the smaller disc defines a diagram $D'$ of some tangle. Likewise, the restriction of $D$ to the annulus can be regarded as a diagram $A$ of some ``annular'' version of a tangle. We may define the notion of a site of $A$ as a subset of arcs on both boundary components of the annulus of size equal to half the total number of arcs.
	Furthermore, regarding regions of~$A$ meeting at least one of the two boundary components of the annulus as open regions, we may define Kauffman states of~$A$, sites belonging to Kauffman states, etc. just as in definition~\ref{def:basic}. 
\end{Remark}
	
\begin{proposition}[Annular glueing/tangle replacement formula]\label{prop:annularglueing}
	With the notation from the previous remark, let us write \(\mathbb{I}\) for the set of arcs \(S\smallsetminus \im(D)\). Then, for any site \(s\) of \(D\),
	$$\hat{\nabla}_D^s=\sum_{s'\in\mathbb{S}(D')} \sum_{x^A\in\mathbb{K}(A,s\cup(\mathbb{I}\smallsetminus s'))} c(x^A) \cdot \hat{\nabla}_{D'}^{s'}.$$
	In particular, if \(D''\) is a diagram with \(\hat{\nabla}_{D''}^{s'}=\hat{\nabla}_{D'}^{s'}\) for all \(s'\in\mathbb{S}(D')\), we may replace \(D'\) by \(D''\) in \(D\) without changing the value of \(\hat{\nabla}_D^s\). The same holds if we replace \(\hat{\nabla}\) throughout by \(\nabla\).
\end{proposition}

\begin{proof}
	Let us fix $s\in\mathbb{S}(D)$ and define
	\begin{align*}
	\Phi\co\coprod_{s'\in\mathbb{S}(D')}\mathbb{K}(D',s')\times\mathbb{K}(A,s\cup(\mathbb{I}\smallsetminus s'))\longrightarrow\mathbb{K}(D,s),
	\end{align*}
	by $(x', x^A)\mapsto x=x'\cup x^A$. If $\Phi$ is a well-defined 1:1-correspondence, the observation follows.
	
	If $D$, $D'$ and $A$ are connected diagrams, it is obvious that $\Phi$ sets up a well-defined bijection.
	For the general case, let us first show that $\Phi$ is well-defined. %Let us call any marker of $x$ and any arc which is not in $s'$, respectively $s^A$, a base of the region in $D'$, respectively $A$, that it is contained in.
	Let us call any marker of $x'$ and any arc which is not in $s'$ a base of the region in $D'$ that it is contained in. Let us do the same for markers of $x^A$, arcs not in $s^A:=s\cup(\mathbb{I}\smallsetminus s')$ and regions in $A$.
	Naturally, any closed region of $D$ which is also closed in $D'$ or $A$ contains exactly one base, which is a marker of $x$.
	Furthermore, any open region of $D'$ and $A$ also contains exactly one base. For each region in $D'$, respectively $A$, let us draw an arrow from its base to any arc in $s'$, respectively $s^A$. The union of these arrows forms a graph, and the connected components of this graph correspond to the regions of $D$. Since any arc in $\mathbb{I}$ is either in $s'$ or $s^A$, but not in both, no two arrows terminate at the same vertex. So there is at most one source (ie vertex with no incoming arrows) in each component, and exactly one if the component does not contain any loops. Suppose this is the case. Any marker of $x$ is a source. So if the source of a graph component is an arc, the corresponding region of $D$ is open, unoccupied and all other arcs which lie in that region belong to $s$. In particular, any closed region contains exactly one a marker. If the source is a marker, any arcs of $\partial D^2\smallsetminus\im(D)$ are in $s$, so the arcs of any open and occupied region lie in $s$. It remains to discuss the case where a component of the graph does contain a loop. Then the  corresponding region of $D$ is not simply-connected, ie it encloses some smaller diagram. Without loss of generality, we may assume that any closed region of this smaller diagram is simply-connected. Then, by the same argument as above, such a  region is occupied by a marker. However, this is not possible by an Euler characteristic argument. 
	
	To see that $\Phi$ is in fact a bijection, we would like to define an inverse by restricting a Kauffman state $x\in\mathbb{K}(D,s)$ to $D'$ and $A$. This is indeed possible since $x$ and $s$ uniquely determine $s'$. This one can see using a similar argument to the one above, noting that the existence of a Kauffman state of $D$ belonging to $s$ implies that all regions are simply-connected. 
\end{proof}

We can show other glueing formulas in a similar way; in particular, we obtain the following generalisation of the connected sum formula for knots and links.

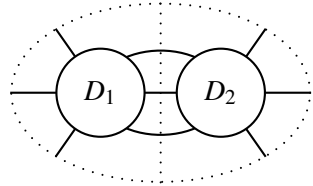
\begin{wrapfigure}{r}{0.3333\textwidth}
	\centering
	\psset{unit=0.4}
	\vspace*{-10pt}
	\begin{pspicture}(-5.1,-3.5)(5.1,3.5)
	\psellipse[linestyle=dotted](0,0)(5,3)
	
	\psline(-2,0)(2,0)
	\psline[linestyle=dotted](0,3)(0,-3)

	\rput(-2,0){
		\psline(0,0)(-3,0)
		\psline(0,0)(2.6;125)
		\psline(0,0)(2.6;-125)
		%\psarc[linestyle=dotted](0,0){1.8}{130}{175}
		%\psarc[linestyle=dotted](0,0){1.8}{-175}{-130}
		%\psarc[linestyle=dotted](0,0){1.8}{-40}{-5}
		%\psarc[linestyle=dotted](0,0){1.8}{5}{40}
	}
	
	\rput(2,0){
		\psline(0,0)(3,0)
		\psline(0,0)(2.6;55)
		\psline(0,0)(2.6;-55)
		%\psarc[linestyle=dotted](0,0){1.8}{140}{175}
		%\psarc[linestyle=dotted](0,0){1.8}{-175}{-140}
		%\psarc[linestyle=dotted](0,0){1.8}{-50}{-5}
		%\psarc[linestyle=dotted](0,0){1.8}{5}{50}
	}

	\psecurve(-2,0)(-1.5,1)(1.5,1)(2,0)
	\psecurve(-2,0)(-1.5,-1)(1.5,-1)(2,0)
	
	\pscircle[fillcolor=white,fillstyle=solid](-2,0){1.5}
	\pscircle[fillcolor=white,fillstyle=solid](2,0){1.5}
%	\rput(-6.5,0){$D=$}
	\rput(-2,0){$D_1$}
	\rput(2,0){$D_2$}
	\end{pspicture}
	\caption{Splitting a tangle diagram into two pieces}\label{fig:SplittingTangles}
	\vspace*{25pt}
\end{wrapfigure}

\myfixwrapfig

\begin{proposition}[splitting/glueing formula]\label{prop:glueing}
	Let \(D_1\) and \(D_2\) be two oriented tangle diagrams obtained by splitting an oriented tangle diagram \(D\) along some arc that does not meet any crossings; for an illustration, see figure~\ref{fig:SplittingTangles}.
	Then
	$$\hat{\nabla}_D^s=\sum\hat{\nabla}_{D_1}^{s_1} \hat{\nabla}_{D_2}^{s_2},$$
	where the sum is over all pairs 
	\((s_1,s_2)\in\mathbb{S}(D_1)\times\mathbb{S}(D_2)\) such that \(s_1\cap s_2=\emptyset\) and the set of arcs in \(\partial D^2\smallsetminus \im(D)\) which lie in \(s_1\cup s_2\) is equal to~\(s\). \qed
\end{proposition}

\begin{theorem} \label{thm:nablaisaninvariant}
For two oriented tangle diagrams \(D_1\) and \(D_2\) representing the same tangle \(T\) and \(s\in\mathbb{S}(T)\), we have \(\nabla^s_{D_1}=\nabla^s_{D_2}\). 
\end{theorem}

\begin{figure}[b]
	\centering
	{\psset{unit=0.2}
	\raisebox{-0.5cm}{
		\begin{pspicture}(-2,-3)(2,3)
		\psarc(-3,0){4.2426406871193}{-45}{45}
		\end{pspicture}
	}	
	$\longleftrightarrow$
	\raisebox{-0.5cm}{
		\begin{pspicture}(-3,-3)(4,3)
		\psline(-3,-3)(1,1)
		\pscircle*[linecolor=white](0,0){0.5}
		\psline(1,-1)(-3,3)
		\psarcn(2,0){1.4142135623731}{135}{-135}
		\end{pspicture}
	}
	}
	\qquad
	{\psset{unit=0.5}
	\raisebox{-1.15cm}{
		\begin{pspicture}(-1.5,-1.6)(1.5,3.5)
		
		\psecurve[linecolor=darkgreen]{<-}(-1.8,3.6)(-0.9,2.7)(-0.3,0.45)(-0.9,-0.9)(-1.8,-1.8)
		\psecurve[linecolor=violet]{<-}(1.8,3.6)(0.9,2.7)(0.3,0.45)(0.9,-0.9)(1.8,-1.8)
		
		\rput(0,1.8){
			\uput{1.35}[135](0,0){$\textcolor{darkgreen}{p}$}
			\uput{1.35}[45](0,0){$\textcolor{violet}{q}$}}
		\uput{1.35}[-135](0,0){$\textcolor{darkgreen}{p}$}
		\uput{1.35}[-45](0,0){$\textcolor{violet}{q}$}
		
		\end{pspicture}
	}
	$\longleftrightarrow$
	\raisebox{-1.15cm}{
		\begin{pspicture}(-1.5,-1.6)(1.5,3.5)
		\rput(0,1.8){
			\psline[linecolor=violet]{->}(-0.9,-0.9)(0.9,0.9)
			\pscircle*[linecolor=white](0,0){0.3}
			\psline[linecolor=darkgreen]{->}(0.9,-0.9)(-0.9,0.9)
		}
		\psline[linecolor=violet]{->}(0.9,-0.9)(-0.9,0.9)
		\pscircle*[linecolor=white](0,0){0.3}
		\psline[linecolor=darkgreen]{->}(-0.9,-0.9)(0.9,0.9)
		
		\rput(0,1.8){
			\uput{1.35}[135](0,0){$\textcolor{darkgreen}{p}$}
			\uput{1.35}[45](0,0){$\textcolor{violet}{q}$}}
		\uput{1.35}[-135](0,0){$\textcolor{darkgreen}{p}$}
		\uput{1.35}[-45](0,0){$\textcolor{violet}{q}$}
		
		\end{pspicture}
	}
}
	\caption{Reidemeister moves RM I (left) and RM II (right)}\label{fig:RMs}
\end{figure}
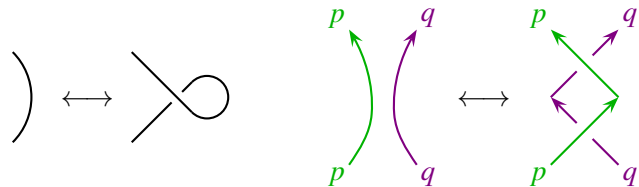

\begin{proof}
By remark~\ref{rem:RMmovesconnectdiagrams}, we just need to check that the polynomials are invariant under the Reidemeister moves RM~I--III. By proposition~\ref{prop:annularglueing}, we can check this locally, so the proof becomes exactly the same as for the usual knot and link case: We verify the theorem for the basic diagrams that appear in the Reidemeister moves, for each site separately. We only do this for RM~I and~II; for RM~III, we refer the reader to the Mathematica notebook \cite{APT.nb} which uses the package~\cite{APT.m} for calculating $\nabla_T^s$ for any connected tangle diagram $T$ and site~$s$.

Let us consider RM I first, see figure~\ref{fig:RMs}. The enclosed region on the right only has one crossing. Hence, the corresponding marker has to sit in that region in every Kauffman state. For both orientations, the labelling of this marker is 1, so we might as well remove this crossing. The same holds if we reverse the crossing; we can either check this directly, or apply proposition~\ref{prop:mirrortangle}. 

For RM~II, we only check one orientation; again, for the others, we can either check this separately or simply apply proposition~\ref{prop:reverseorientI}. In the diagram on the right, there are exactly two Kauffman states that occupy the open region on the left; they contribute $\textcolor{darkgreen}{p}$ and $h\textcolor{darkgreen}{p}$, so after setting $h=-1$, they cancel. The same is true for the open region on the right; the contribution there is $\textcolor{violet}{q}$ and $h\textcolor{violet}{q}$. Finally, for each of the open regions at the top and the bottom, there is exactly one Kauffman state and it contributes 1.
\end{proof}

\begin{definition}
	Now that we know that \(\nabla_D^s\) is a tangle invariant, we will allow ourselves to become less careful in distinguishing between tangles and their diagrams. We will use ``tangles'' and ``tangle diagrams'' synonymously, unless it is clear from the context that we do not. For example, when we talk about Kauffman states for tangles, we implicitly fix a diagram first. 
	We will from now on 
	also write \(\nabla_T^s\) for \(\nabla_D^s\) and call it the \textbf{Alexander polynomial of~$T$ at the site~$s$}. 
\end{definition} 

We have chosen the letter $\nabla$ for a reason:

\begin{theorem}\label{thm:twoended}
Let \(T\) be a diagram of an oriented $2$-ended tangle representing a link~\(L\). There is only one site of~\(T\), namely the empty set~\(\emptyset\). Let the colour of the open component be~\(c\). Then the Conway potential function \(\nabla_L\) is equal to
$$\frac{1}{c-c^{-1}}\nabla^\emptyset_T.$$
\end{theorem}
\begin{proof}
Verify that $\nabla:=\frac{1}{c-c^{-1}}\nabla^\emptyset_T$ satisfies the axioms in \cite{jiang14}, see \cite{APT.nb}.
\end{proof}

\begin{Remark}\label{rem:conwaypotfunctionandHFL}
Recall that the Conway potential function of an $n$-component oriented link $L$
is a rational function $\nabla_L(t_1,\dots,t_n)$
which is related to the multivariate Alexander polynomial $\Delta_L$ in the following way: (see for example \cite{hartley} or \cite{jiang14})
$$
\nabla_L(t_1,\dots,t_n)=
\begin{cases}
\displaystyle\frac{\Delta_L(t_1^2)}{t_1-t_1^{-1}}, & \text{if $n=1$};
\\
\Delta_L(t_1^2,\dots,t_n^2), & \text{if $n>1$}.
\end{cases}
$$
Hence, using the notation of the theorem above
$$
\nabla^\emptyset_T(t_1,\dots,t_n)=
\begin{cases}
\Delta_L(t_1^2), & \text{if $n=1$};
\\
(c-c^{-1}) \Delta_L(t_1^2,\dots,t_n^2), & \text{if $n>1$}.
\end{cases}
$$
We also note that $\nabla^\emptyset_T$ of a 2-ended tangle $T$, multiplied by a factor of $(t_i-t_i^{-1})$ for each \textit{closed} component of $T$, is equal to the Euler characteristic of Ozsváth and Szabó's link Floer homology from~\cite{OSHFL}.
\end{Remark}

	% BasicProps.tex
\section{\texorpdfstring{Basic properties of $\nabla_T^s$}{Basic properties of ∇}}\label{sec:basicpropertiesofnabla}

In this section, we study some basic properties of the tangle invariants $\nabla_T^s$, guided by the properties of the Conway potential function, which were first studied by Hartley in~\cite{hartley}. Theorem~\ref{thm:CPFprops} below summarizes some of the properties proved in that article. But first, let us state the following basic result about $\nabla_T^s$ which becomes a simple observation once we have interpreted $\nabla_T^s$ in terms of Heegaard diagrams, see page~\pageref{proof:ExponentsMod2Agree} below theorem~\ref{thm:Eulercharagreeswithnabla}. Alternatively, one can also prove the result directly by generalising Kauffman's clock theorem to tangles, which is the approach chosen in~\cite{essay}.

\begin{lemma}\label{lem:ExponentsMod2Agree}
	Let \(T\) be an oriented tangle and \(s\) a site of \(T\). Then for each colour \(t_i\), the exponents of the terms in \(t_i\) differ by multiples of 2 in~\(\nabla_T^s\).
\end{lemma}

\begin{theorem}\label{thm:CPFprops}
	{\normalfont\cite[propositions 5.6, 5.5, 5.7 and 5.3]{hartley}}
	The Conway potential function of an oriented \(r\)-component link \(L\) satisfies the following properties:
	\begin{enumerate}[(i)]
		\item If \(\m(L)\) denotes the mirror image of \(L\), then \label{CPFpropmirror}
		$$\nabla_{\m(L)}(t_1,\dots,t_r)=(-1)^{r-1}\cdot\nabla_L(t_1,\dots,t_r).$$
		\item $\nabla_L(t_1,\dots,t_r)=(-1)^{r}\cdot\nabla_L(t_1^{-1},\dots,t_r^{-1}).$\label{CPFpropssymmetry}
		\item If \(\rr(L,t_1)\) is obtained from \(L\) by reversing the orientation of the first strand, then \label{CPFpropRevOneStrand}
		$$\nabla_{\rr(L,t_1)}(t_1,\dots,t_r)=-\nabla_L(t_1^{-1},t_2,\dots,t_r).
		$$
		\item If \(r>1\) and \(L_1\) is the link obtained from \(L\) by removing the $t_1$-component, then \(\nabla_L(1,t_2,\dots,t_r)\) equals \label{CPFoneequalone}\vspace{\abovedisplayskip}\\
		\textcolor{white}{\qedsymbol}\hfill$\left(t_2^{\lk(t_1,t_2)}\cdots t_r^{\lk(t_1,t_r)}-t_2^{-\lk(t_1,t_2)}\cdots t_r^{-\lk(t_1,t_r)}\right)\cdot\nabla_{L_1}(t_2,\dots,t_r).$\hfill\qedsymbol
	\end{enumerate}
\end{theorem}

\begin{proposition}\label{prop:mirrortangle}
Let \(T\) be an oriented tangle and \(\m(T)\) its mirror image. Then for all \(s\in\mathbb{S}(T)\), 
$$\hat{\nabla}^{s}_{\m(T)}(t_1,\dots,t_r,h)=\hat{\nabla}^{s}_T(t_1^{-1},\dots ,t_r^{-1}, h^{-1}).$$
\end{proposition}
\begin{proof}
Observe that the two Alexander codes in figure~\ref{figAlexCodesForNabla} are mirror images of one another after taking the reciprocals of all variables. 
\end{proof}
\begin{definition}\label{def:linkingnumber}
We define the \textbf{linking number} $\lk_T(p,q)$ for two components $p$ and $q$ of an oriented tangle $T$ to be 
\begin{align*}
	\lk_T(p,q):=&\tfrac{1}{2}\#\{\text{positive crossings between $p$ and $q$}\}\\
	&-\tfrac{1}{2}\#\{\text{negative crossings between $p$ and $q$}\}.
\end{align*}
For a tangle with a component $t_j$, we also define 
$$\lk_T(t_j):=\sum\lk_T(t_i,t_j),$$ 
where the sum is over all $i\neq j$. We sometimes omit the subscript $T$ when there is no risk of ambiguity.
\end{definition}
\begin{Remark} Note that for two-component links, $\lk(p,q)$ coincides with the usual linking number. Also, linking numbers are invariants of tangles.
\end{Remark}

\begin{proposition}\label{prop:reverseorientI}
Let \(T\) be an oriented \(r\)-component tangle. If \(\rr(T,t_1)\) denotes the same tangle \(T\) with the orientation of the first strand reversed, then for all sites \(s\in\mathbb{S}(T)\), we have 
$$\hat{\nabla}^s_{\rr(T,t_1)}(t_1,\dots,t_r)=h^{\lk_T(t_1)} \hat{\nabla}^s_{T}(h^{-1}t_1^{-1},t_2,\dots,t_r).$$
\end{proposition}

\begin{proof}
This is easily seen by considering crossings separately: Modulo sign, the statement follows from observation~\ref{ObsAlexanderCode}. For the correct sign, note that after substituting $h^{-1}t_1^{-1}$ for $t_1$ in the Alexander code of a positive (negative) crossing involving $t_1$ and some different colour, we obtain the Alexander code of the crossing with the orientation of the $t_1$-strand reversed multiplied by $h^{-\frac{1}{2}}$ (respectively $h^{\frac{1}{2}}$). For crossings involving only $t_1$, no additional factor is necessary, and for crossings not involving $t_1$ at all, there is nothing to show.
\end{proof}
\begin{corollary}\label{cor:alloorientsrev}
Let \(T\) be an oriented \(r\)-component tangle. If \(\rr(T)\) denotes the same tangle \(T\) with the orientation of all strands reversed, then for all sites \(s\in\mathbb{S}(T)\), we have $$\hat{\nabla}^s_{\rr(T)}(t_1,\dots,t_r)=\hat{\nabla}^s_{T}(h^{-1}t_1^{-1},\dots,h^{-1}t_r^{-1}).$$
If \(\rr(\cdot)\) denotes the function which substitutes \(-t^{-1}\) for each colour~\(t\), the above implies 
$$\nabla_{\rr(T)}^s=\rr(\nabla_T^s).$$
Moreover, for an oriented link \(L\), we then have the symmetry relation
$$\nabla_{L}=\nabla_{\rr(L)}=\rr(\nabla_{L}).$$
\end{corollary}
\begin{proof}
For the first part, we successively reverse the orientation of all strands, noting that each term $\lk_T(t_i,t_j)$ appears twice in the exponent of $h$, but with different signs, because the second time it appears, the orientation of one strand has been reversed.
The second statement follows directly from the first with $h=-1$. 
The second equality of the final statement follows from theorem~\ref{thm:twoended} and the previous statement. The first part is a combination of theorem \ref{thm:CPFprops} (ii) and (iii).
\end{proof}
\begin{lemma}[one-colour skein relation]\label{onecolourskein}
Let \(T_+\), \(T_-\) and \(T_\circ\) denote the tangles 
\!\raisebox{-5pt}{\psset{unit=0.15}\begin{pspicture}[showgrid=false](-2,-2)(2,2)
	\psrotate(0,0){45}{
		\psline{<-}(2.2,0)(-2,0)
		\psline{<-}(0,2.2)(0,0.5)
		\psline(0,-2)(0,-0.5)
	}
	\end{pspicture}}\!,  
\!\raisebox{-5pt}{\psset{unit=0.15}\begin{pspicture}[showgrid=false](-2,-2)(2,2)
	\psrotate(0,0){-45}{
		\psline{->}(2,0)(-2.2,0)
		\psline{<-}(0,2.2)(0,0.5)
		\psline(0,-2)(0,-0.5)
	}
	\end{pspicture}}\! and
\!\raisebox{-5pt}{\psset{unit=0.15}
\begin{pspicture}[showgrid=false](-2,-2)(2,2)
\psrotate(0,0){45}{
	\pscustom{
		\psline{<-}(0,2)(0,1)
		\psarcn(-1,1){1}{0}{-90}
		\psline(-1,0)(-2,0)
	}%
	\pscustom{
		\psline(0,-2)(0,-1)
		\psarcn(1,-1){1}{180}{90}
		\psline{->}(1,0)(2,0)
	}%
}
\end{pspicture}}\!
respectively. Then for all sites \(s\),
$$\nabla_{T_+}^s(t,t)-\nabla_{T_-}^s(t,t)=(t-t^{-1})\cdot \nabla_{T_\circ}^s(t,t).$$
Thus, the univariate polynomial tangle invariant \(\nabla_T^s(t,\dots,t)\) satisfies the same skein relation as the Alexander polynomial. 
\end{lemma}
\begin{proof}
Straightforward.
\end{proof}
\begin{corollary}\label{corknotpmoneequalone}
Let \(T\) be a 2-ended tangle representing a knot. Then \(\nabla_{T}^s(\pm 1)=1\).
\end{corollary}
\begin{proof}
Let $T'$ be the diagram obtained from $T$ by changing some crossings such that $T'$ represents the unknot. Then, by lemma~\ref{onecolourskein}, $\nabla_{T}^s(\pm 1)=\nabla_{T'}^s(\pm 1)$, and $\nabla_{T'}^s(t)\equiv 1$.
\end{proof}

The next proposition corresponds to part (\ref{CPFoneequalone}) of theorem~\ref{thm:CPFprops}.
\begin{proposition}\label{prop:setpmone}
Let \(T\) be an oriented tangle whose \(t_1\)-component is closed and let \(T_1\) be the tangle obtained from \(T\) by removing this component. Then for all \(s\in\mathbb{S}(T)\), \(\nabla^s_{T}(\pm 1,t_2,\dots,t_r)\) equals
$$(\pm1)^{\lk(t_1)+1}\left(t_2^{\lk(t_1,t_2)}\cdots t_r^{\lk(t_1,t_r)}-t_2^{-\lk(t_1,t_2)}\cdots t_r^{-\lk(t_1,t_r)}\right)\cdot\nabla^s_{T_1}(t_2,\dots,t_r).$$
\end{proposition}

\begin{proof}
	To simplify notation, let us write $\sigma=\pm1$ throughout the proof for the substituted value of~$t_1$.
	Then, first observe that by lemma~\ref{onecolourskein}, changing a $t_1$-$t_1$-crossing in $T$ to obtain a new tangle $T'$ does not change the invariant evaluated at $t_1=\sigma$, ie
	$$\nabla_{T}^s(\sigma,t_2,\dots,t_r)=\nabla_{T'}^s(\sigma,t_2,\dots,t_r).$$
	We can therefore assume without loss of generality that the $t_1$-component is the unknot and we may choose a diagram for which there are no $t_1$-$t_1$-crossings. Moreover, by fixing the $t_1$-component and then ``pulling on all other strands'', we may choose a tangle diagram which contains the tangle $T_\circ$ shown on the left-hand side of figure~\ref{fig:setpmone}, where each of the $n$ boxes labelled $T_\pm$ is one of the two oriented tangles $T_+$ and $T_-$ shown on the right-hand side of the same figure.
	
	Let $l$, $b$, $r$ and $t$ denote the sites specified by the regions on the left, bottom, right and top of the diagrams for $T_{\pm}$ from figure~\ref{fig:setpmone}, respectively, and consider the values of $\nabla^s_{T_{\pm}}$ given in table~\ref{tab:setpmone}. 
	In particular, observe that $\nabla^{s}_{T_{\pm}}(\sigma,t_i)=0$ for $s=r$. This means that for any site $s$ of $T_{\circ}$ which contains one of the $n$ unshaded open regions of $T_{\circ}$ in figure~\ref{fig:setpmone}, $\nabla_{T_{\circ}}^s$ vanishes when setting $t_1=\sigma$. Since a site $s$ of $T_{\circ}$ consists of exactly $(n-1)$ open regions, the remaining sites are in 1:1-correspondence with the shaded open regions, where the correspondence is given by taking the complement in the set of shaded regions.

\begin{figure}[t]
	\centering
	\raisebox{-2.4cm}{\psset{unit=1}
		\begin{pspicture}(-3.5,-2.5)(2.5,2.5)
		\pscircle*[linecolor=lightgray,linewidth=0pt](0,0){2.5}
		\pscircle*[linecolor=white,linewidth=0pt](0,0){1.5}
		\psarc(0,0){1.5}{110}{70}
		
		\uput{1.5}[30](0,0){\rput{30}(0,0){
				\pscircle*[linecolor=white,linewidth=0pt](1.5,0){1.5}
				\psarc(1.5,0){1.5}{124}{236}
				\pscustom[fillstyle=solid,fillcolor=white]{\pspolygon(0.3,-0.6)(0.3,0.6)(-0.3,0.6)(-0.3,-0.6)}}
			\rput{30}(0,0){$T_{\pm}$}}
		
		\uput{1.5}[-50](0,0){\rput{-50}(0,0){
				\pscircle*[linecolor=white,linewidth=0pt](1.5,0){1.5}
				\psarc(1.5,0){1.5}{124}{236}
				\pscustom[fillstyle=solid,fillcolor=white]{\pspolygon(0.3,-0.6)(0.3,0.6)(-0.3,0.6)(-0.3,-0.6)}}
			\rput{-50}(0,0){$T_{\pm}$}}
		
		\uput{1.5}[-130](0,0){\rput{-130}(0,0){
				\pscircle*[linecolor=white,linewidth=0pt](1.5,0){1.5}
				\psarc(1.5,0){1.5}{124}{236}
				\pscustom[fillstyle=solid,fillcolor=white]{\pspolygon(0.3,-0.6)(0.3,0.6)(-0.3,0.6)(-0.3,-0.6)}}
			\rput{-130}(0,0){$T_{\pm}$}}
		
		\uput{1.5}[-210](0,0){\rput{-210}(0,0){
				\pscircle*[linecolor=white,linewidth=0pt](1.5,0){1.5}
				\psarc(1.5,0){1.5}{124}{236}
				\pscustom[fillstyle=solid,fillcolor=white]{\pspolygon(0.3,-0.6)(0.3,0.6)(-0.3,0.6)(-0.3,-0.6)}}
			\rput{-210}(0,0){$T_{\pm}$}}
		
		\psarc[linestyle=dotted](0,0){1.5}{70}{110}
		\pscircle[linestyle=dotted](0,0){2.5}
		\rput[r](-2.6,0){$T_{\circ}=$}
		\end{pspicture}}
	\hspace{1cm}
	\begin{minipage}{0.2\textwidth}
		\centering
		\raisebox{-0.91cm}{\begin{pspicture}(-0.35,-1)(0.35,1)
			\pscustom[fillstyle=solid,fillcolor=white]{\pspolygon(0.3,-0.6)(0.3,0.6)(-0.3,0.6)(-0.3,-0.6)}\rput(0,0){$T_{+}$}
			\end{pspicture}}
		$=$
		\raisebox{-1.34cm}{\psset{unit=0.5}
			\begin{pspicture}(-1,-2)(1,4)
			\rput(0,1.8){
				\psline{->}(0.9,-0.9)(-0.9,0.9)
				\pscircle*[linecolor=white](0,0){0.3}
				\psline{->}(-0.9,-0.9)(0.9,0.9)
			}
			
			\psline{->}(0.9,-0.9)(-0.9,0.9)
			\pscircle*[linecolor=white](0,0){0.3}
			\psline{->}(-0.9,-0.9)(0.9,0.9)
			
			\rput(-0.6,3.2){$t_1$}
			\rput(0.6,3.2){$t_i$}
			\end{pspicture}}
		\\
		\raisebox{-0.91cm}{\begin{pspicture}(-0.35,-1)(0.35,1)
			\pscustom[fillstyle=solid,fillcolor=white]{\pspolygon(0.3,-0.6)(0.3,0.6)(-0.3,0.6)(-0.3,-0.6)}\rput(0,0){$T_{-}$}
			\end{pspicture}}
		$=$
		\raisebox{-1.34cm}{\psset{unit=0.5}
			\begin{pspicture}(-1,-2)(1,4)
			\rput(0,1.8){
				\psline{->}(0.9,-0.9)(-0.9,0.9)
				\pscircle*[linecolor=white](0,0){0.3}
				\psline{<-}(-0.9,-0.9)(0.9,0.9)
			}
			
			\psline{<-}(0.9,-0.9)(-0.9,0.9)
			\pscircle*[linecolor=white](0,0){0.3}
			\psline{->}(-0.9,-0.9)(0.9,0.9)
			
			\rput(-0.6,3.2){$t_1$}
			\rput(0.6,3.2){$t_i$}
			\end{pspicture}}
	\end{minipage}
	\caption{The $2n$-ended tangle $T_{\circ}$ constructed in the first step of the proof of proposition~\ref{prop:setpmone}. The shaded open regions of $T_{\circ}$ are the only ones that can be occupied by a site $s$ of $T_{\circ}$ for which $\nabla_{T_{\circ}}^s$ does not vanish when setting $t_1=\sigma$.}\label{fig:setpmone}
\end{figure}
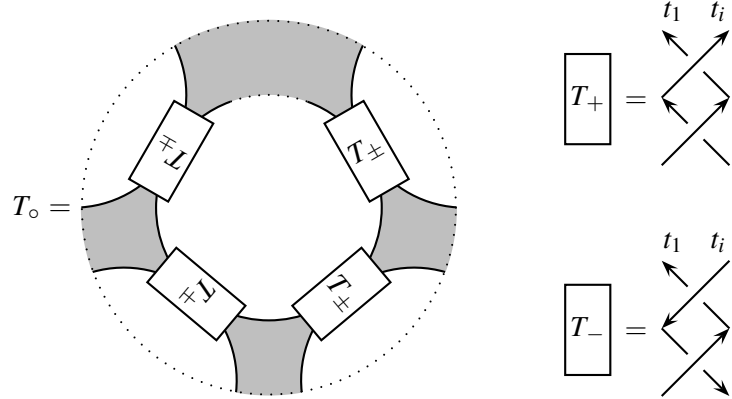
\begin{table}[t]
	\centering
	\begin{subtable}{0.4\textwidth}
		\centering
		\begin{tabular}{ccc}
			$s$ & $T_{+}$ & $T_{-}$\\ 
			\hline
			\rule{0pt}{14pt}$l$ & $(t_i-t_i^{-1})$ & $-(t_i-t_i^{-1})$ \\ 
			$b$ & $t_1^{-1} t_i^{-1}$ & $t_1^{-1} t_i$ \\ 
			$r$ & $(t_1-t_1^{-1})$ & $-(t_1-t_1^{-1})$ \\ 
			$t$ & $t_1 t_i$ & $t_1 t_i^{-1}$ \\ 
		\end{tabular} 
		\caption{}
	\end{subtable}
	\begin{subtable}{0.4\textwidth}
		\centering
		\begin{tabular}{ccc}
			$s$ & $T_{+}$ & $T_{-}$ \\ 
			\hline 
			\rule{0pt}{14pt}$l$ & $(t_i-t_i^{-1})$ & $-(t_i-t_i^{-1})$ \\ 
			$b$ & $\sigma t_i^{-1}$ & $\sigma t_i$ \\ 
			$r$ & $0$ & $0$ \\ 
			$t$ & $\sigma t_i$ & $\sigma t_i^{-1}$ \\ 
		\end{tabular}
		\caption{}
	\end{subtable}
\caption{Values of $\nabla^s_{T_{\pm}}$ before setting $t_1=\sigma$ (a) and after (b)}\label{tab:setpmone}
\end{table}

Let us fix one of the shaded regions $u$ and its corresponding site $s$. We will now compute $\nabla_{T_{\circ}}^s(\sigma,t_2,\dots,t_r)$. 
Let us write $T_j$ for the $j^\text{th}$ box $T_\pm$ in anticlockwise direction from the fixed unoccupied region $u$. Each Kauffman state of $T_{\circ}$ of the fixed site $s$ is uniquely determined by the crossing whose marker lies in the central region of $T_{\circ}$. Then, for each box $T_j$, there are two such Kauffman states which, when restricted to the box $T_k$, correspond to the site $t$ if $k<j$, the site $l$ if $k=j$ and the site $b$ if $k>j$. 
Thus, the contribution of these two Kauffman states to $\nabla_{T_{\circ}}^s(\sigma,t_2,\dots,t_r)$ is 
$$\nabla_{T_j}^l(\sigma,t_{(j)})\cdot \prod_{k<j}\nabla_{T_k}^t(\sigma,t_{(k)})\cdot
\prod_{k>j}\nabla_{T_k}^b(\sigma,t_{(k)})=\pm(t_{(j)}-t_{(j)}^{-1})\cdot \prod_{i\neq 1} (\sigma t_i)^{\alpha_j^{i}-\beta_j^{i}},$$
where
\begin{itemize}
	\item $(k)\neq 1$ denotes the index of the strand involved in $T_k$, $k=1,\dots,n$,
	\item the sign $\pm$ is determined by whether $T_j=T_+$ or $T_j=T_-$,
	\item $\alpha_j^{i}$ is the signed number of boxes $T_k$ with $(k)= i$ and $k<j$, where $T_k=T_\pm$ counts as $\pm1$, and likewise
	\item $\beta_j^{i}$ is the signed number of boxes $T_k$ with $(k)= i$ and $k>j$.
\end{itemize}
Note that by the definition of the linking number
$$
\lk(t_1,t_i)= 
\begin{cases}
\alpha_j^{i}+\beta_j^{i} &\text{ if $i\neq 1,(j)$,}\\
\alpha_j^{(j)}+\beta_j^{(j)}+1  &\text{ if $i=(j)$ and $T_j=T_{+}$,}\\
\alpha_j^{(j)}+\beta_j^{(j)}-1  &\text{ if $i=(j)$ and $T_j=T_{-}$.}\\
\end{cases}
$$
So the expression above for the contribution of the two Kauffman states corresponding to the box $T_j$ becomes
$$
\underbrace{\pm(t_{(j)}-t_{(j)}^{-1})\cdot (\sigma t_{(j)})^{\lk(t_1,t_{(j)})-2\beta_{j}^{(j)}\mp 1}}_{(\ast)}\cdot \prod_{i\neq 1,(j)} (\sigma t_i)^{\lk(t_1,t_{i})-2\beta_{j}^{i}}.
$$
Moreover, the expression $(\ast)$ can be expanded to
$$
\sigma \left((\sigma t_{(j)})^{\lk(t_1,t_{(j)})-2\beta_{j}^{(j)}}-(\sigma t_{(j)})^{\lk(t_1,t_{(j)})-2(\beta_{j}^{(j)}\pm1)}\right).
$$
By considering the two consecutive boxes $T_{j-1}$ and $T_{j}$ for $j\neq 1$, we observe
$$\beta_{j-1}^{i}=
\begin{cases*}
\beta_{j}^{i}&\text{if $i\neq1,(j)$,}\\
\beta_{j}^{(j)}\pm1&\text{if $i=(j)$.}
\end{cases*}
$$
Hence, the box $T_j$ contributes 
$$\sigma\cdot \left(\prod_{i\neq 1} (\sigma t_i)^{\lk(t_1,t_{i})-2\beta_{j}^{i}}- \prod_{i\neq 1} (\sigma t_i)^{\lk(t_1,t_{i})-2\beta_{j-1}^{i}}\right),$$
where
$$\beta_{0}^{i}:=\lk(t_1,t_i)=
\begin{cases*}
\beta_{1}^{i}&\text{if $i\neq1,(1)$,}\\
\beta_{1}^{(1)}\pm1&\text{if $i=(1)$.}
\end{cases*}
$$
So we see that when we take the sum over the contributions of all boxes $T_j$ to $\nabla_{T_{\circ}}^s(\sigma,t_2,\dots,t_r)$, most terms cancel and the only surviving ones are the second one of $T_1$ and the first one of $T_n$. Since $\beta_{n}^{i}=0$ for all $i\neq1$, we see that $\nabla_{T_{\circ}}^s(\sigma,t_2,\dots,t_r)$ is equal to
\begin{align*}
&\sigma\left((\sigma t_2)^{\lk(t_1,t_2)}\cdots (\sigma t_r)^{\lk(t_1,t_r)}-(\sigma t_2)^{-\lk(t_1,t_2)}\cdots (\sigma t_r)^{-\lk(t_1,t_r)}\right)\\
=~&\sigma^{\lk(t_1)+1}\left(t_2^{\lk(t_1,t_2)}\cdots t_r^{\lk(t_1,t_r)}-t_2^{-\lk(t_1,t_2)}\cdots  t_r^{-\lk(t_1,t_r)}\right).
\end{align*}
Finally, note that the invariant for the tangle obtained by deleting the $t_1$-component in $T_{\circ}$ is equal to 1 for the site $s$ corresponding to the region $u$ and 0 for any site containing an unshaded open region of $T_{\circ}$. Now use proposition~\ref{prop:annularglueing}.
\end{proof}

	% FourEndedTangles.tex
\section{\texorpdfstring{4-ended tangles and mutation invariance of $\nabla_T^s$}{4-ended tangles and mutation invariance of ∇}}\label{sec:4endedandmutation}

\begin{wrapfigure}{r}{0.3333\textwidth}
	\centering
	\vspace*{-10pt}
	\begin{pspicture}(-2.1,-1.5)(2.1,1.05)
	\rput(-1.1,0){
		\SpecialCoor
		\psline{<-}(0.6;45)(1;45)
		\psline{->}(0.6;-45)(1;-45)
		
		\psline{->}(0.6;135)(1;135)
		\psline{<-}(0.6;-135)(1;-135)
		
		\uput{1.1}[135](0,0){$t$}
		\uput{1.1}[-135](0,0){$t$}
		\uput{1.1}[45](0,0){$t$}
		\uput{1.1}[-45](0,0){$t$}
		
		\uput{0.6}[0](0,0){$c$}
		\uput{0.6}[90](0,0){$d$}
		\uput{0.6}[180](0,0){$a$}
		\uput{0.6}[270](0,0){$b$}
		
		\pscircle[linestyle=dotted](0,0){1}
		
		\rput(0,-1.3){\textit{type 1}}
	}
	
	\rput(1.1,0){
		\psline{->}(0.6;45)(1;45)
		\psline{<-}(0.6;-135)(1;-135)
		
		\psline{->}(0.6;135)(1;135)
		\psline{<-}(0.6;-45)(1;-45)
		
		\uput{1.1}[135](0,0){$t$}
		\uput{1.1}[-135](0,0){$t$}
		\uput{1.1}[45](0,0){$t$}
		\uput{1.1}[-45](0,0){$t$}
		
		\uput{0.6}[0](0,0){$c$}
		\uput{0.6}[90](0,0){$d$}
		\uput{0.6}[180](0,0){$a$}
		\uput{0.6}[270](0,0){$b$}
		\pscircle[linestyle=dotted](0,0){1}
		\rput(0,-1.3){\textit{type 2}}
	}
	\end{pspicture}
	\caption{Two orientations on a 4-ended tangle}\label{fig:4endedorientationsPOLY}
\end{wrapfigure}
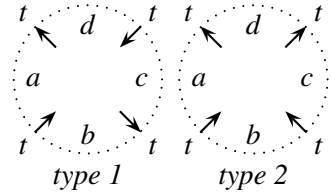
\myfixwrapfig

\begin{theorem}\label{thm:fourendedonecolour}
	Let $T$ be an oriented 4-ended tangle.
	Then up to cyclic permutation of the sites, it carries one of the two types of orientations shown in figure~\ref{fig:4endedorientationsPOLY}. Let \(\rr(T)\) denote the same tangle with the opposite orientation on all strands. After identifying the colours of the two open components of \(T\), we have
	$$\text{\(\nabla^{a}_{T}=\nabla^{c}_{\rr(T)}=\nabla^{c}_{T}\)
		and
		\(\nabla^{d}_{T}=\nabla^{b}_{\rr(T)}=\nabla^{b}_{T}\)}$$
	for type 1. For type 2, all identities but the last hold true.
\end{theorem}

\begin{Remark}\label{rem:OtherRelation}
	Recall from corollary~\ref{cor:alloorientsrev} that for all tangles $T$ and sites $s$, $\nabla_{\rr(T)}^s=\rr(\nabla_T^s)$.
	Then the proposition above shows in particular that for a 4-ended tangle, one can obtain its invariant for one site from the one for the opposite site, if we use the same colour on the open strands. This is also true for tangles with different colours on the open strands, in which case one can easily prove similar relations. Moreover, there are certain 4-term relations that can be used to obtain the invariants for all sites from the invariant of a single site. For more details, see~\cite[section~I.3]{thesis}.
\end{Remark}

\begin{proof}
	In both cases, the identity $\nabla^{c}_T=\nabla^{a}_T$ follows from theorem~\ref{thm:twoended} and the fact that the two diagrams obtained by closing at site $a$ or site $c$ both represent the same link. Next, by closing the tangle at site $c$ and applying the final part of corollary~\ref{cor:alloorientsrev}, we obtain
	$$\rr\left(\frac{\nabla_T^{c}}{t-t^{-1}}\right)=\frac{\nabla_T^{c}}{t-t^{-1}},$$
	where $t$ is the colour of the two open strands of $T$. So we immediately get $\nabla^{c}_{\rr(T)}=\rr(\nabla^{c}_{T})=\nabla^{c}_T$. 
	Similarly, the other two identities for type 1 can be seen by closing $b$ or $d$. So it remains to show $\nabla^{d}_T=\nabla^{b}_{\rr(T)}$ for type 2. For this, add a negative crossing to the right of the diagram to get a tangle~$T'$. Then $\nabla_{T'}^d=\nabla^d_T-t\cdot\nabla^c_T$ and $\nabla_{T'}^b=\nabla^b_T+t^{-1}\cdot\nabla^c_T$, respectively.
	Now, the tangle $T'$ is of type~1, so in particular, $\nabla_{T'}^d=\nabla_{\rr(T')}^b=\rr(\nabla_{T'}^b)$. Thus, $$\nabla^d_T-t\cdot\nabla^c_T=\rr(\nabla^b_T+t^{-1}\cdot\nabla^c_T)=\rr(\nabla^b_{T})-t\cdot\rr(\nabla^c_{T})=\nabla^b_{\rr(T)}-t\cdot\nabla^c_{\rr(T)}.$$ 
	Since $\nabla^c_T=\nabla^c_{\rr(T)}$, we obtain $\nabla^d_{T}=\nabla^b_{\rr(T)}$ as desired.
\end{proof}

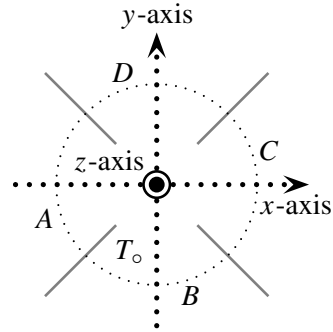
\begin{wrapfigure}{r}{0.3333\textwidth}
	\centering
	\psset{unit=0.45}
	\begin{pspicture}[showgrid=false](-4.3,-4.3)(5.3,5.4)
	
	\psline[linestyle=dotted,linewidth=2pt](0,0)(4.25,0)
	\psline[linestyle=dotted,linewidth=2pt](-4.25,0)(0,0)
	\psline[linestyle=dotted,linewidth=2pt](0,0)(0,4.25)
	\psline[linestyle=dotted,linewidth=2pt](0,-4.25)(0,0)
	
	\psline[linestyle=dotted,linewidth=2pt,arrows=->](4.3,0)(4.5,0)
	\psline[linestyle=dotted,linewidth=2pt,arrows=->](0,4.3)(0,4.5)

	{\psset{linewidth=1pt,linecolor=gray}
		\psline(1.2,1.2)(3.3,3.3)
		\psline(-1.2,-1.2)(-3.3,-3.3)
		\psline(1.2,-1.2)(3.3,-3.3)
		\psline(-1.2,1.2)(-3.3,3.3)
	}%
	
	\pscircle[linestyle=dotted](0,0){3}
	
	\rput[t](4.1,-0.3){$x$-axis}
	\rput[b](0,4.6){$y$-axis}
	
	\rput[br](-0.3,0.3){\psframebox[framesep=0pt,linecolor=white, fillstyle=solid,fillcolor=white]{$z$-axis}}
	
	\rput(-0.8,-2){$T_\circ$}
	
	\rput(-3.3,-1){$A$}
	\rput(1,-3.3){$B$}
	\rput(3.3,1){$C$}
	\rput(-1,3.3){$D$}
	
	\psdot[dotsize=11pt](0,0)
	\psdot[dotsize=9pt,linecolor=white](0,0)
	\psdot[dotsize=6pt](0,0)
	\end{pspicture}
	\caption{The three mutation axes}\label{fig:MutationAxes}
	\bigskip
\end{wrapfigure}
\myfixwrapfig
\begin{definition}\label{def:mutation}
	Let $T$ be a tangle and $T_\circ$ a 4-ended tangle obtained by intersecting $T$ with a closed 3-ball $B^3$. We may assume that all four tangle ends of $T_\circ$ lie equally spaced on a great circle on $\partial B^3$. 
	Let $T'$ be the tangle obtained from~$T$ by rotation of $T_\circ$ by $\pi$ about one of the three axes that switch pairs of endpoints of~$T_\circ$ as shown in figure~\ref{fig:MutationAxes}. We say $T'$ is obtained from $T$ by \textbf{mutation} or $T'$ is a \textbf{mutant} of $T$. $T_\circ$ is called the \textbf{mutating tangle}. 
	If $T$ is oriented, we choose an orientation of $T'$ that agrees with the one for $T$ outside of $B^3$. If this means that we need to reverse the orientation of the two open components of $T_\circ$ then we also reverse the orientation of all other components of~$T_\circ$ in $T'$; otherwise we do not change any orientation.
\end{definition}
\begin{Remark}\label{rem:defmut}
	The definition of Conway mutation given in the introduction (definition~\ref{def:INTROmutation}) is equivalent to the one above, which can be seen by twisting the ends of the mutating tangle. While from the viewpoint of the former, some symmetry relations in theorem~\ref{thm:fourendedonecolour} might seem stronger than absolutely necessary for mutation invariance, from the viewpoint of the latter they are ``exactly right''.
\end{Remark}
\begin{corollary}\label{cor:mutation}
	Let \(T\) be an oriented tangle and \(T'\) a mutant of \(T\). Suppose the colours of the two open strands of the mutating tangle agree. Then for all sites \(s\in\mathbb{S}(T)=\mathbb{S}(T')\),
	$$\nabla_T^s=\nabla_{T'}^s.$$
\end{corollary}
\begin{corollary}
	The multivariate Alexander polynomial is mutation invariant, provided that the two open strands of the mutating tangle have the same colour.
	\hfill$\square$
\end{corollary}
\begin{wrapfigure}{r}{0.3333\textwidth}
	\medskip
	\centering
	\psset{unit=0.37}
	\begin{pspicture}(-4,-4)(4,4)
	
	\pscircle[linestyle=dotted](0,0){4}
	
	\psecurve[linecolor=darkgreen,arrows=<-](6;135)(4;135)(-2,1)(-1.85,0)(-2,-1)%p
	\psecurve[linecolor=darkgreen,arrows=->](2,1)(1.85,0)(2,-1)(4;-45)(6;-45)%q
	\psecurve[linecolor=violet](-2,1)(0,1.8)(2,1)(2.4,0)(2,-1)(0,-1.8)(-2,-1)(-2.4,0)(-2,1)(0,1.8)(2,1)(2.4,0)(2,-1)%r
	
	\pscircle*[linecolor=white](2,1){0.4}
	\pscircle*[linecolor=white](2,-1){0.4}
	\pscircle*[linecolor=white](-2,-1){0.4}
	\pscircle*[linecolor=white](-2,1){0.4}
	
	\psecurve[linecolor=darkgreen](-2,1)(-1.85,0)(-2,-1)(4;-135)(6;-135)%p
	\psecurve[linecolor=darkgreen](6;45)(4;45)(2,1)(1.85,0)(2,-1)%q
	\psecurve[linecolor=violet]{->}(2,1)(2.4,0)(2,-1)(0,-1.8)(-2,-1)%r
	\psecurve[linecolor=violet]{->}(-2,-1)(-2.4,0)(-2,1)(0,1.8)(2,1)%r
	
	\rput[c](-3,0){$a$}
	\rput[c](0,-3){$b$}
	\rput[c](3,0){$c$}
	\rput[c](0,3){$d$}
	
	\rput[c](4.6;135){$\textcolor{darkgreen}{p}$}
	\rput[c](4.6;-45){$\textcolor{darkgreen}{p}$}
	\rput[c](0,1){$\textcolor{violet}{q}$}
	\end{pspicture}
	\caption{A counterexample for a relation $\nabla_{\rr(T,\textcolor{violet}{q})}^s=\nabla_{T}^s$}\label{fig:mutorient}
	\medskip
\end{wrapfigure}
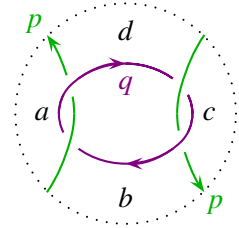

\myfixwrapfig
\begin{Remark}\label{rem:mutorient}
One might wonder, why in both definitions of Conway mutation (definitions~\ref{def:INTROmutation} and~\ref{def:mutation}), we asked that the orientations of the closed strands of the mutating tangle be reversed during mutation iff the orientations of the open strands are reversed. This is because otherwise, corollary~\ref{cor:mutation} is no longer true, for a symmetry relation $\nabla_{\rr(T,\textcolor{violet}{q})}^s=\nabla_{T}^s$ does not hold in general, where $\textcolor{violet}{q}$ is the colour of a closed strand. As a counterexample, consider the tangle $T$ shown in figure~\ref{fig:mutorient}. 
Then $\nabla_T^b=\textcolor{darkgreen}{p}^{2} \textcolor{violet}{q}  - (\textcolor{violet}{q}-\textcolor{violet}{q}^{-1})- \textcolor{darkgreen}{p}^{-2} \textcolor{violet}{q}^{-1}$.
\end{Remark}

\begin{proof}[Proof of corollary \ref{cor:mutation}]
We consider the same two cases as in theorem~\ref{thm:fourendedonecolour}, and also use its notation. Denote by $A$, $B$, $C$ and $D$ the Alexander polynomials of the corresponding counterparts of the sites $a$, $b$, $c$ and $d$ in $T\smallsetminus T_\circ$, such that
$$\nabla_T^{s}=A\nabla^{a}_{T_\circ}+B\nabla^{b}_{T_\circ}+C\nabla^{c}_{T_\circ}+D\nabla^{d}_{T_\circ},$$  
see proposition~\ref{prop:annularglueing}.
If we rotate about the $x$-axis, we have to reverse orientations in both cases of theorem~\ref{thm:fourendedonecolour}, so
$$\nabla_{T'}^s=A\nabla^{a}_{\rr(T_\circ)}+B\nabla^{d}_{\rr(T_\circ)}+C\nabla^{c}_{\rr(T_\circ)}+D\nabla^{b}_{\rr(T_\circ)}=A\nabla^{a}_{T_\circ}+B\nabla^{b}_{T_\circ}+C\nabla^{c}_{T_\circ}+D\nabla^{d}_{T_\circ}=\nabla_T^{s}.$$
Next, let us consider rotations about the $y$-axis. For type~1, we need to reverse orientations:
$$\nabla_{T'}^s=A\nabla^{c}_{\rr(T_\circ)}+B\nabla^{b}_{\rr(T_\circ)}+C\nabla^{a}_{\rr(T_\circ)}+D\nabla^{d}_{\rr(T_\circ)}
=A\nabla^{a}_{T_\circ}+B\nabla^{b}_{T_\circ}+C\nabla^{c}_{T_\circ}+D\nabla^{d}_{T_\circ}=\nabla_T^{s}$$
For type~2, we do not need to reverse orientations:
$$\nabla_{T'}^s=A\nabla^{c}_{T_\circ}+B\nabla^{b}_{T_\circ}+C\nabla^{a}_{T_\circ}+D\nabla^{d}_{T_\circ}=A\nabla^{a}_{T_\circ}+B\nabla^{b}_{T_\circ}+C\nabla^{c}_{T_\circ}+D\nabla^{d}_{T_\circ}=\nabla_T^{s}.$$
For a rotation about the $z$-axis we can argue similarly or we simply observe that it is the same as a rotation about both the $x$- and the $y$-axis (in any order).
\end{proof}
	% HDsForTangles.tex
\section{\texorpdfstring{The homological tangle invariant $\HFT$}{The homological tangle invariant HFT}}\label{sec:HDsForTangles}

In this section, we consider a more general notion of tangles. We allow tangles to live in arbitrary 3-manifolds $M$ with spherical boundary $\partial M=S^2$, ie we replace the 3-ball $B^3$ in definition~\ref{def:tangle} by any such $M$. So, throughout this section, let $M$ be a 3-manifold with spherical boundary and $T$ an oriented tangle in $M$ with $n$ open and $m$ closed components. We write $M_T$ for the tangle complement $M\smallsetminus \nu(T)$, where $\nu(T)$ is an open tubular neighbourhood of $T$.

\begin{definition}\label{def:HDsfortangles}
A \textbf{Heegaard diagram $\mathcal{H}_T$ for an oriented tangle} $T$ in $M$ is a tuple $$(\Sigma_g,\A=\Ac\cup\Aa,\B),$$ where
\begin{itemize}
\item $\Sigma_g$ is an oriented surface of genus $g$ with $2(n+m)$ boundary components, denoted by~$\Gamma$, which are partitioned into $(n+m)$ pairs,
\item $\Ac$ is a set of $(g+m)$ pairwise disjoint circles $\alpha_1,\dots, \alpha_{g+m}$ on $\Sigma_g$,
\item $\Aa$ is a set of $2n$ pairwise disjoint arcs $\aaa_1,\dots,\aaa_{2n}$ on $\Sigma_g$ which are disjoint from~$\Ac$ and whose endpoints lie on~$\Gamma$, and
\item $\B$ is a set of $(g+m+n-1)$ pairwise disjoint circles $\beta_1,\dots, \beta_{g+m+n-1}$ on~$\Sigma_g$.
\end{itemize}
We impose the following condition on the data above: The 3-manifold obtained by attaching 2-handles to $\Sigma_g\times [0,1]$ along $\Ac\times\{0\}$ and $\B\times\{1\}$ is equal to the tangle complement~$M_T$ such that under this identification, 
\begin{itemize}
\item each pair of circles in $\Gamma$ is a pair of meridional circles for the same tangle component, and each tangle component belongs to exactly one such pair, and 
\item $\Aa\times \{0\}$ is equal to the intersection of ${\red S^1}\subset \partial M$ with $M_T$ in $M$.
\end{itemize}
If the tangle $T$ is oriented, we also orient the boundary components of $\Sigma_g$ as oriented meridians of the tangle components, using the right-hand rule. Our convention on the orientation of the Heegaard surface is that its normal vector field (using the right-hand rule) points in the positive direction, ie the direction of the $\beta$-curves. However, we usually draw the Heegaard surfaces such that this normal vector field points into the plane. This slightly unusual convention is chosen because when identifying a sphere with the projection plane plus a point at infinity, it is more convenient to place this point on the back of the sphere, rather than on the front. This convention makes it easier to draw Heegaard diagrams for rational tangles. 
\end{definition}

\begin{figure}[t]
	\centering
	\begin{subfigure}[b]{0.28\textwidth}
		\centering
		\psset{unit=0.2}
		\begin{pspicture}(-10.3,-10.3)(10.3,10.3)
		\psline(6;-45)(6;135)
		\pscircle*[linecolor=white](0,0){0.7}
		\psline(6;-135)(6;45)
		\pscircle[linestyle=dotted](0,0){6}
		\end{pspicture}
		\caption{A single crossing tangle}\label{fig:1crossingTforHD}
	\end{subfigure}
	\quad
	\begin{subfigure}[b]{0.34\textwidth}
		\centering
		\psset{unit=0.2}
		\begin{pspicture}(-10.3,-10.3)(10.3,10.3)
		\psrotate(0,0){90}{
			
			\psecurve[linecolor=blue](8,-0.2)(7,0)(2,2)(0,7)(-0.2,8)
			\psecurve[linecolor=blue](-4,4)(0,7)(-8.3,8.3)(-7,0)(-4,4)
			\psecurve[linecolor=blue](-8,0.2)(-7,0)(-2,-2)(0,-7)(0.2,-8)
			\psecurve[linecolor=blue](4,-4)(0,-7)(8.3,-8.3)(7,0)(4,-4)
			
			\psarc[linecolor=red](0,0){7}{0}{360}
			
			\rput(7;45){\pscircle*[linecolor=white]{1}\pscircle[linecolor=darkgreen]{1}}
			\rput(7;135){\pscircle*[linecolor=white]{1}\pscircle[linecolor=darkgreen]{1}}
			\rput(7;225){\pscircle*[linecolor=white]{1}\pscircle[linecolor=darkgreen]{1}}
			\rput(7;315){\pscircle*[linecolor=white]{1}\pscircle[linecolor=darkgreen]{1}}
			
			%\psecurve(10,0)(9,9)(0,10)(-9,9)(-10,0)(-9,-9)(0,-10)(9,-9)(10,0)(9,9)(0,10)
		}
		\psdot(0,7)
		\psdot(7,0)
		\psdot(0,-7)
		\psdot(-7,0)
		
		\rput[b](5.5;-145){\textcolor{red}{$\Aa$}}
		\rput[b](2;-45){\textcolor{blue}{$\B$}}
		\rput[br](8;135){$\Gamma$}
		
		\rput[br](0,7.5){$D$}
		\rput[tl](7.5,0){$C$}
		\rput[tl](0,-7.5){$B$}
		\rput[br](-7.5,0){$A$}
		\end{pspicture}
		\caption{A Heegaard diagram for the tangle on the left}\label{fig:HDfor1crossing}
	\end{subfigure}
	\quad
	\begin{subfigure}[b]{0.28\textwidth}
		\centering
		\psset{unit=0.2}
		\begin{pspicture}(-5.1,-10.1)(5.1,10.1)
		\psellipticarc[linestyle=dotted,dotsep=1pt,linecolor=darkgreen](0,-8)(5,2){0}{180}
		\psellipticarc[linestyle=dotted,dotsep=1pt,linecolor=blue](0,0)(5,2){0}{180}
		\psline[linecolor=red](4,-9.2)(4,6.8)
		\psline[linecolor=red](-4,-9.2)(-4,6.8)
		\psellipticarc[linecolor=darkgreen](0,-8)(5,2){180}{0}
		\psline(4.93,-8)(4.93,8)
		\psline(-4.93,8)(-4.93,-8)
		\psellipse[linecolor=darkgreen](0,8)(5,2)
		\psellipticarc[linecolor=blue](0,0)(5,2){180}{0}
		
		\psellipse[linecolor=darkgreen](0,0)(0.5,0.6)
		
		\psellipse[linecolor=darkgreen](0,-4)(0.5,0.6)
		
		\psellipse[linecolor=red](0,-2)(2,5)
		\psdot(1.98,-1.82)
		\psdot(-1.98,-1.82)
		\psdot(4,-1.2)
		\psdot(-4,-1.2)
		\end{pspicture}
		\caption{A ladybug; this terminology is taken from \cite[figure~9c]{BaldwinLevine}.}\label{fig:ladybug}
	\end{subfigure}
	\caption{The two building blocks of tangle Heegaard diagrams in example~\ref{exa:HDforonecrossing}}\label{fig:HDBuildingBlocks}
\end{figure}
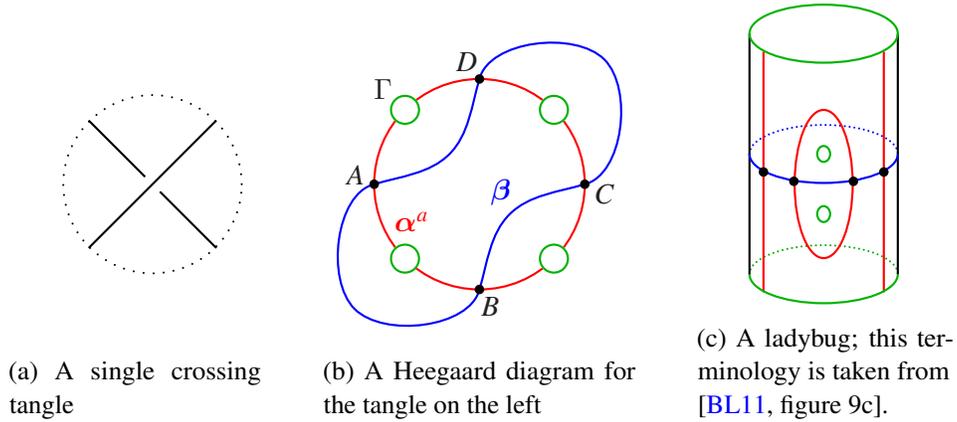

\begin{example}\label{exa:HDforonecrossing}
	For a 1-crossing tangle in $B^3$, we draw the Heegaard diagram shown in figure~\ref{fig:HDfor1crossing}. From this, we can obtain a Heegaard diagram for any tangle in $B^3$ without closed components as follows: We cut a tangle diagram of a given tangle up into 4-ended tangles with a single crossing each. Then, for each such component, we can use our Heegaard diagram from figure~\ref{fig:HDfor1crossing} and then glue these copies together along $\Gamma$ according to the tangle diagram. For tangles in $B^3$ with closed components, we can do the same except that into each closed component, we insert one copy of the ``ladybug'' from figure~\ref{fig:ladybug}.
\end{example}

\begin{example}\label{exa:RationalTangles}
	For rational tangles, we can draw Heegaard diagrams on genus 0 surfaces. As illustrated in figure~\ref{fig:HDforRatTan}, this can be seen by performing Dehn twists on the Heegaard diagram for the 1-crossing tangle in figure~\ref{fig:HDfor1crossing}. In fact, a 4-ended tangle without closed components is rational iff it has a genus~0 Heegaard diagram. Indeed, a genus 0 Heegaard diagram for such a tangle has no $\alpha$-circles and just a single $\beta$-circle. By definition, we know that performing surgery along this $\beta$-circle gives us two cylinders, so it separates two punctures from the other two. 
\end{example}

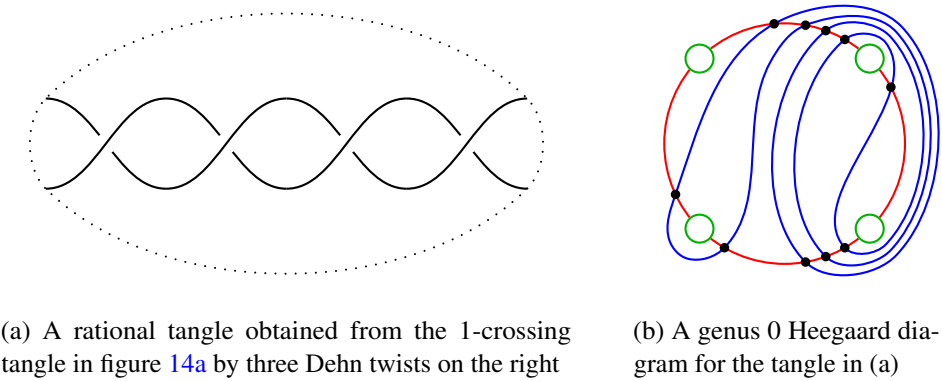
\begin{figure}[t]
	\centering
	\psset{unit=0.2}
	\begin{subfigure}[b]{0.6\textwidth}
		\centering
		\begin{pspicture}(-2,-10.3)(34,10.3)
		\psscalebox{1 -1}{
			\psecurve[linestyle=dotted](32,3)(0,3)(0,-3)(32,-3)(32,3)(0,3)(0,-3)
			
			\psecurve(-8,3)(0,-3)(8,3)(16,-3)
			\psecurve(8,3)(16,-3)(24,3)(32,-3)
			
			\psecurve(0,3)(8,-3)(16,3)(24,-3)
			\psecurve(16,3)(24,-3)(32,3)(40,-3)
			
			\pscircle*[linecolor=white](4,0){0.7}
			\pscircle*[linecolor=white](12,0){0.7}
			\pscircle*[linecolor=white](20,0){0.7}
			\pscircle*[linecolor=white](28,0){0.7}
			
			\psecurve(0,-3)(8,3)(16,-3)(24,3)
			\psecurve(16,-3)(24,3)(32,-3)(40,3)
			
			\psecurve(-8,-3)(0,3)(8,-3)(16,3)
			\psecurve(8,-3)(16,3)(24,-3)(32,3)
		}
		\end{pspicture}
		\caption{A rational tangle obtained from the 1-crossing tangle in figure~\ref{fig:1crossingTforHD} by three Dehn twists on the right}\label{fig:tangleforHDforRatTan}
	\end{subfigure}
	\quad
	\begin{subfigure}[b]{0.33\textwidth}
		\centering
		\begin{pspicture}(-10.3,-10.3)(10.3,10.3)
		\psscalebox{1 -1}{
			\psarc[linecolor=red](0,0){8}{0}{360}
			
			\SpecialCoor
			\rput(8;45){\pscircle*[linecolor=white]{1}\pscircle[linecolor=darkgreen]{1}}
			\rput(8;135){\pscircle*[linecolor=white]{1}\pscircle[linecolor=darkgreen]{1}}
			\rput(8;225){\pscircle*[linecolor=white]{1}\pscircle[linecolor=darkgreen]{1}}
			\rput(8;315){\pscircle*[linecolor=white]{1}\pscircle[linecolor=darkgreen]{1}}
			
			\psecurve[linecolor=blue]
			(10;130)(10;140)(8;155)
			(8;-95)(11;-45)
			(10.5;45)(8;80)
			(8;-70)(10;-45)
			(9.5;45)(8;60)
			(9.5;-45)(8;-60)
			(8;70)(10;45)
			(10.5;-45)(8;-80)
			(5.5;-100)
			(8;120)(10;130)(10;140)(8;155)
			
			\psdot(8;155)
			\psdot(8;120)
			\psdot(8;80)
			\psdot(8;70)
			\psdot(8;60)
			\psdot(8;-28)
			\psdot(8;-60)
			\psdot(8;-70)
			\psdot(8;-80)
			\psdot(8;-95)
		}
		\end{pspicture}
		\caption{A genus 0 Heegaard diagram for the tangle in (a)}\label{fig:HDforHDforRatTan}
	\end{subfigure}
	\caption{Illustration of example~\ref{exa:RationalTangles}}\label{fig:HDforRatTan}
\end{figure}

More generally, we can regard Heegaard diagrams for tangles as bordered sutured Heegaard diagrams, by endowing the tangle complement $M_T$ with a particular bordered sutured structure. A reader familiar with these terms can skip the following paragraphs and go directly to definition~\ref{def:TanglesAsSpecialBSMflds}; for everyone else, and also for the purpose of fixing our notation, we will give a very brief outline of what these terms mean. 

\begin{definition}\label{def:SuturedManifold}
	A \textbf{sutured manifold} $(X,\Gamma)$ is a compact oriented 3-manifold $X$ with boundary together with a set $\Gamma\subset\partial X$ of smoothly embedded simple closed curves, called the \textbf{sutures}, which divide $\partial X$ into two (not necessarily connected) subsurfaces $R_+$ and $R_-$ such that each suture lies in the boundary of both $R_+$ and $R_-$. Usually, an orientation of $R_+$ and $R_-$ is fixed to distinguish these two subsurfaces, namely the normal vector field of $R_+$ points out of $X$, whereas the normal vector field of $R_-$ points into the 3-manifold. Such an orientation of $R_+$ and $R_-$ induces an orientation on the sutures $\Gamma$ which in turn is sufficient to distinguish $R_+$ and $R_-$. 
	
	A sutured manifold $(X,\Gamma)$ is called \textbf{balanced} if $X$ has no closed component, the two surfaces $R_+$ and $R_-$ have the same Euler characteristic and each component of $\partial X$ contains at least one suture. 
\end{definition}

\begin{Remark}
	When comparing the definition above to \cite[definitions~2.1 and~2.2]{Juhasz}, note that we have slightly simplified the notation by recording only the sutures and not their tubular neighbourhoods. We have also omitted any specification of toroidal boundary components of $X$, ie $T(\gamma)$ in the notation of \cite[definition~2.1]{Juhasz}. This is because we will focus on balanced sutured manifolds for which such components are not present anyway as they do not contain any suture. 
\end{Remark}

\begin{definition}
	A \textbf{Heegaard diagram of a balanced sutured manifold} $(M,\Gamma)$ is a tuple $(\Sigma, \A, \B)$, where $\Sigma$ is a compact oriented surface with boundary and $\A$ and $\B$ are
	two sets of pairwise disjoint simple closed curves in the interior of $\Sigma$, such that the 3-manifold $M$ is obtained from $\Sigma\times [0,1]$ by attaching 3–dimensional 2–handles
	along the curves $\A \times \{ 0 \}$ and $\B\times\{1 \}$ and the sutures are defined as $\partial\Sigma\times \{\frac{1}{2}\}$ with the orientation induced by $\Sigma$. Thus, $R_{-}$ is the result of surgery of $\Sigma\times\{0\}$ along the 2-handles corresponding to the $\alpha$-curves, which agrees with Zarev's and Juhász's convention.
\end{definition}

\begin{Remark}
	By~\cite[proposition~2.9]{Juhasz}, the number of curves in $\A$ and $\B$ are the same in a Heegaard diagram of a balanced sutured manifold. Furthermore, Juhász shows in~\cite[proposition~2.14]{Juhasz}, using a standard argument from Morse theory, that any balanced sutured manifold has a Heegaard diagram. 
\end{Remark}

\begin{definition}
	An \textbf{arc diagram} $\mathcal{Z}$ is a triple $(Z, \mathbf{a},M)$, where $Z$ is a (possibly empty) set of oriented line segments, $\mathbf{a}$ an even number of points on $Z$ and $M$ a matching of points in $\mathbf{a}$. The \textbf{graph $G(\mathcal{Z})$ of an arc diagram} $\mathcal{Z}$ is the graph obtained from the line segments $Z$ by adding an edge between matched points in $\mathbf{a}$. An arc diagram is called degenerate if surgery of the 1-manifold $Z$ along all matched points contains closed components. 
\end{definition}
\begin{definition}
	A \textbf{bordered sutured manifold} is a tuple 
	$(X,\Gamma, \mathcal{Z},\phi),$
	where
	\begin{itemize}
		\item $X$ is a sutured manifold with sutures $\Gamma$;
		\item $\mathcal{Z}=(Z, \mathbf{a},M)$  is a non-degenerate arc diagram;
		\item $\phi$ is a (smooth) embedding of $G(\mathcal{Z})$ into the closure of $R_-$ such that $\im(\phi)\cap\Gamma=\phi(Z)$.
	\end{itemize}
\end{definition}

\begin{definition}\label{def:HDforBorderedSuturedMfdls}
	A \textbf{Heegaard diagram of a bordered sutured manifold} is obtained from a Heegaard diagram of the underlying sutured manifold by adding the graphs of the arc diagrams to it. To be more precise, consider a Heegaard diagram of the underlying sutured manifold. Then we can embed the graph $G(\mathcal{Z})$ into $R_-$ in such a way that it misses the 2-handles corresponding to the $\alpha$-curves, simply by sliding them off those 2-handles. This gives us an embedding of $G(\mathcal{Z})$ into the Heegaard surface such that its image does not intersect the $\alpha$-curves. We view the images of the edges connecting points in $\mathbf{a}$ as $\alpha$-arcs.
\end{definition}

\begin{Remark}
	In~\cite{ZarevThesis}, Zarev developed a theory which simultaneously generalises Juhász's Heegaard Floer homology $\SFH$ for sutured 3-manifolds~\cite{Juhasz} and Lipshitz, Ozsváth and Thurston's bordered Heegaard Floer homology~\cite{LOT}. In general, given an arc diagram~$\mathcal{Z}$, Zarev defines an algebra~$\mathcal{A}(\mathcal{Z})$. Moreover, with a bordered sutured Heegaard diagram $\mathcal{H}_X$ of a bordered sutured manifold $X$, he associates (up some analytic choices) two different algebraic structures over~$\mathcal{A}(\mathcal{Z})$, so-called type~A and type~D structures, denoted by $\BSA(\mathcal{H}_X)$ and $\BSD(\mathcal{H}_X)$, respectively. In some very precise sense (see for example~\cite[appendix~A]{thesis}), these algebraic structures can be regarded as generalisations of chain complexes, and one can define notions of homotopy equivalence for them in the usual way. It turns out that up to homotopy equivalence, $\BSA(\mathcal{H}_X)$ and $\BSD(\mathcal{H}_X)$ are invariants of the bordered sutured manifold $X$. Hence, Zarev denotes them instead by $\BSA(X)$ and $\BSD(X)$, respectively. 
	Since type~A and type~D structures are dual to each other and type~D structures are often easier to understand, we will restrict ourselves in the following to type~D structures. 
\end{Remark}

\begin{definition}\label{def:TanglesAsSpecialBSMflds}
	We can endow the tangle complement $M_T$ with the structure of a bordered sutured manifold as follows: Each closed component gets two oppositely oriented meridional circles and around each tangle end, we have a single suture such that the boundary of $M$ minus some contractible neighbourhoods of the tangle ends lies in $R_-$. Furthermore, the arcs ${\red S^1}\smallsetminus \nu(\partial T)$ together with small neighbourhoods of the endpoints on the sutures $\Gamma$ constitute the arc diagram, which is indeed non-degenerate. Then, Heegaard diagrams for tangles can be viewed as bordered sutured Heegaard diagrams. In particular, it follows that every tangle $T$ in a 3-manifold $M$ with spherical boundary has a Heegaard diagram. 
\end{definition}

Of course, Heegaard diagrams for tangles are far from being unique. As a special case of \cite[proposition~4.1.5]{ZarevThesis}, we note the following result.

\begin{lemma}\label{lem:HeegaardMoves}
	Any two diagrams for the same tangle can be obtained from one another by a sequence of the following \textbf{Heegaard moves}:
	\begin{itemize}
		\item an isotopy of an \(\alpha\)- or \(\beta\)-circle or an isotopy of an \(\alpha\)-arc relative to its endpoints,
		\item a handleslide of a \(\beta\)-circle over another \(\beta\)-circle,
		\item a handleslide of an \(\alpha\)-curve over an \(\alpha\)-circle and
		\item stabilisation.\qed
	\end{itemize} 
\end{lemma}

\begin{Remark}
	In general, the arc diagrams in bordered sutured Heegaard Floer theory play the following role: Given a sutured 3-manifold $X$ and a surface $F$ which divides $X$ into two pieces (and satisfies certain other sufficiently general conditions), one can choose a handle decomposition of $F$ which corresponds to an arc diagram such that the embedding of this arc diagram onto the boundary of each piece turns this piece into a bordered sutured manifold. A central result of Zarev's thesis~\cite[theorem~12.3.2]{ZarevThesis} is that the sutured Floer homology $\SFH(X)$ can be computed from the bordered sutured invariants of the two pieces. 
		
	Returning to our case, the arc diagram of the bordered sutured structure on $M_T$ parametrizes $2n$ disjoint strips connecting the tangle ends. For glueing tangles together to obtain knots or links, in general, one needs to choose arc diagrams parametrizing the $2n$-punctured sphere, see for example~\cite[construction~1.7]{AlishahiLipshitz}. 
	Therefore, the homological invariant $\HFT$ defined below
	from the bordered sutured structure chosen above on $M_T$ does not satisfy a glueing formula of the desired form. Nonetheless, by extending the chosen arc diagram on $M_T$, I show in~\cite[section~II.3]{thesis} that one only needs to modify the differential on $\HFT$ to obtain a glueable tangle invariant.
\end{Remark}

\begin{definition}\label{def:HFTfromBSD}
For the tangle complement $M_T$, the arc diagrams are particularly simple, and so are the algebras associated with them. More precisely, the algebra~$\mathcal{A}(\mathcal{Z})$ agrees with the ring of idempotents~$\mathcal{I}(\mathcal{Z})$ over which it is defined. $\mathcal{I}(\mathcal{Z})$ is equal to the direct sum of $\mathbb{Z}/2.\iota_{s}$ over all subsets $s$ of the set of $\alpha$-arcs~$\Aa$. A type~D structure over this algebra~$\mathcal{A}(\mathcal{Z})$ is equivalent to a set of vector spaces over $\mathbb{Z}/2$, indexed by subsets $s$ of~$\Aa$; so we obtain
$$\BSD(M_T)=\bigoplus_s \iota_s.\BSD(M_T),$$
where $\iota_s.\BSD(M_T)$ is now an honest chain complex over $\mathbb{Z}/2$. 

If we fix a Heegaard diagram $\mathcal{H}_T=(\Sigma_g,\A,\B)$ for $T$, $\BSD(\mathcal{H}_T)$ is generated by tuples $\x=(x_1,\dots,x_{g+m+n-1})$ of points $x_1,\dots,x_{g+m+n-1}\in\A\cap\B$ such that there is exactly one point $x_i$ on each $\alpha$- and $\beta$-circle, and at most one point on each $\alpha$-arc. We denote the set of these \textbf{generators} by $\mathcal{G}=\mathcal{G}_{\mathcal{H}_T}$. Note that each generator occupies exactly $(n-1)$ $\alpha$-arcs. So like in definition~\ref{def:basic}, we define a \textbf{site} of $T$ to be an $(n-1)$-element subset $s$ of $\Aa$ and denote the set of all sites of $T$ by $\mathbb{S}(T)$. Then, we can associate with each $\x\in\mathcal{G}$ the site $s(\x)$ consisting of all those $\alpha$-arcs that are occupied by an intersection point in $\x$. (Note that unlike sites of general Kauffman states, the site of a generator is always unique.) We denote the set of all generators corresponding to a given site~$s$ by $\mathcal{G}^s$. Thus, we obtain a partition 
$$\mathcal{G}=\coprod_{s\in\mathbb{S}(T)}\mathcal{G}^s.$$

From the discussion above, it follows that $\iota_s.\BSD(\mathcal{H}_T)$ vanishes unless $s$ is a site. So we define the chain complex $\CFT(\mathcal{H}_T,s)$ as $\iota_s.\BSD(\mathcal{H}_T)$ and 
$$\CFT(\mathcal{H}_T):=\bigoplus_{s\in\mathbb{S}(T)}\CFT(\mathcal{H}_T,s).$$
By construction, both the chain homotopy types of $\CFT(\mathcal{H}_T,s)$ and $\CFT(\mathcal{H}_T)$ are invariants of $M_T$ and therefore of the tangle~$T$. We denote their homologies by $\HFT(T,s)$ and $\HFT(T)$, respectively. 
\end{definition}

While the generators of a type D structure corresponding to a bordered sutured Heegaard diagram are easy to describe combinatorially, the definition of the differential is more involved; in particular, it requires a substantial amount of analytic machinery, which we will not discuss. However, we use the remainder of this section to describe the combinatorics that go into the definition of the differential in our special case. Let us fix a Heegaard diagram $\mathcal{H}=\mathcal{H}_T=(\Sigma_g,\A,\B)$ for our tangle $T$.

\begin{definition}\label{def:HDtoCFTbasics}
	We define $D_\mathcal{H}$ to be the free Abelian group generated by the connected components of $\Sigma_g\smallsetminus(\A\cup\B\cup\Gamma )$, which we call \textbf{regions}. In other words,
	$$D_{\mathcal{H}}:=H_2(\Sigma_g,\A\cup\B\cup\Gamma;\mathbb{Z}).$$
	Elements of this group are called \textbf{domains}. Given two generators $\x,\y\in\mathcal{G}$, we define $\pi_2(\x,\y)$ to be the subset of those domains $\phi$ which satisfy 
	$$\partial(\partial\phi\cap\B)=\x-\y.$$
	We call elements in $\pi_2(\x,\x)$ \textbf{periodic domains}. Note that this does not depend on the choice of $\x\in\mathcal{G}$.
	Furthermore, let 
	$$\pip(\x,\y):=\{\phi\in\pi_2(\x,\y)\vert\,\Gamma\cap\,\phi=\emptyset\}.$$
	A Heegaard diagram is called \textbf{admissible} if 
	every non-zero periodic domain in $\pip(\x,\x)$ has positive and negative multiplicities. 
\end{definition}

\begin{lemma}\label{lem:HeegaardMovesAdmissibility}
	Every tangle diagram can be made admissible by isotopies of \(\B\). Furthermore, any two such diagrams for the same tangle can be transformed into one another by a sequence of Heegaard moves from lemma~\ref{lem:HeegaardMoves} through admissible diagrams.
\end{lemma}
\begin{proof}
	This is a special case of \cite[proposition~4.4.2 and corollary~4.4.3]{ZarevThesis}. Note that our terminology differs slightly from Zarev's: he does not include regions near basepoints in $\pi_2(\x,\y)$, so in our special case, his definition of $\pi_2(\x,\y)$ coincides with our $\pip(\x,\y)$. Thus, the distinction in his terminology between periodic and provincial periodic domains becomes irrelevant. 
\end{proof}

We can now describe the differential on $\CFT(\mathcal{H}_T)$. Given a generator $\x\in\mathcal{G}$, 
\begin{equation}
	\partial \x=\sum_{\y\in \mathcal{G}}\sum_{\substack{\phi\in\pip(\x,\y)\\ \mu_{\x,\y}(\phi)=1}}\#\widehat{\mathcal{M}}(\x,\y;\phi)\cdot\y, \label{eq:differential}\tag{$\dagger$}
\end{equation}
where $\mu_{\x,\y}(\phi)$ is the Maslov index of $\phi$, see definition~\ref{def:MaslovIndex}, and  $\#\widehat{\mathcal{M}}(\x,\y;\phi)$ denotes the count (modulo 2) of holomorphic curves associated with domains $\phi$ connecting $\x$ to $\y$, see for example \cite{LipshitzCyl}. Note that the sum in~\eqref{eq:differential} is over domains in $\pip(\x,\y)$, ie those that avoid~$\Gamma$. Since there are no such domains between generators in distinct sites, the chain complex $\CFT(\mathcal{H}_T)$ admits a splitting into summands $\CFT(\mathcal{H}_T,s)$.

	% HDsForTangles.tex
\section{\texorpdfstring{Gradings on $\HFT$}{Gradings on HFT}}\label{sec:Gradings}

The chain complexes $\CFT(T,s)$ from the previous section inherit various gradings from bordered sutured theory. In general, the gradings in this theory are rather complicated; for instance, the grading groups are not necessarily Abelian. However, in our case, the gradings allow a much simpler description, in particular if one imposes an additional condition on $M$, namely that its reduced homology vanishes: Throughout this section, let $M$ be a $\mathbb{Z}$-homology 3-ball with spherical boundary and $T\subset M$ an oriented tangle with $n$ open and $m$ closed components and $s$~a~site of~$T$. Again, we write $M_T$ for the tangle complement. We also fix a Heegaard diagram $\mathcal{H}=\mathcal{H}_T=(\Sigma_g,\A,\B)$ for our tangle $T$.

\begin{definition}\label{def:boundaryOfM}
	From the definition of a Heegaard diagram of a tangle, it follows that the surface $S_{\B}(\Sigma_g)$ obtained by surgery along the curves in $\B$ is a disjoint union of $(n+m)$ annuli, each of whose boundary is a pair in $\Gamma$. Similarly, the surface $S_{\Ac}(\Sigma_g)$ obtained by surgery along the curves in $\Ac$ is a disjoint union of $m$ annuli, each of whose boundary is a pair in $\Gamma$, and a 2-sphere with $2n$ boundary components, which is cut into two discs by the $\alpha$-arcs. We call these $(n+2m)$ annuli and two discs \textbf{elementary periodic domains}. This terminology is justified by the following lemma.
\end{definition}

\begin{lemma}\label{lem:noadmissibilityissues}
There is a canonical identification \(\pi_2(\x,\x)\cong H_2(M_T,\Gamma\cup\Aa;\mathbb{Z})\). The right hand side is the free Abelian group generated by the elementary periodic domains, modulo the relation that the sum of all of them vanishes. Furthermore, under this identification, \(\pip(\x,\x)\) is freely generated by the \(m\) tori which are the boundaries of the closed components of \(T\); in particular, any Heegaard diagram for a tangle without closed components is admissible.
\end{lemma}
\begin{proof}
First, we contract each boundary component of $\Sigma_g$ such that we obtain a closed surface $\overline{\Sigma}_g$. On this surface, there are still the $\beta$-circles and $\alpha$-circles, but the $\alpha$-arcs have become a single $\alpha$-circle $\alpha^\ast$. Let us write $\overline{\A}=\Ac\cup\{\alpha^\ast\}$. \(\pi_2(\x,\x)\) can then be expressed as the subgroup of $H_2(\overline{\Sigma}_g,\overline{\A}\cup\B)$ consisting of all domains $\phi$ satisfying $\partial(\partial\phi\cap\B)=0$. This subgroup can be rewritten as the second homology of $\overline{\Sigma}_g\cup_{\overline{\A}\cup\B}\{\text{2-cells}\}$, which is a deformation retract of the tangle complement $M_T$ with 2-handles attached along each component of $\Gamma$ and ${\red S^1} \subset\partial M$. This, in turn, is just another description of $M$ with $(n+2m+1)$ embedded 3-balls removed. Since $H_2(M;\mathbb{Z})=0$, the homology is freely generated by the boundaries of these 3-balls. Each elementary periodic domain now corresponds to such a generator, except for one disc, which is the boundary of $M$. This proves the first two statements.

For the third claim, let us write a given periodic domain in~$\pip(\x,\x)$ as a linear combination of the elementary periodic domains such that the coefficient of one of the two discs, say, is zero. Then the coefficients of the two annuli of each closed component must be the same and all other coefficients must be zero. In fact, the sum of the two annuli for each closed component form a basis of $\pip(\x,\x)$.
\end{proof}
\begin{lemma}\label{lem:pixynonempty}
\(\pi_2(\x,\y)\) is non-empty for all pairs \((\x,\y)\in\mathcal{G}^2\).
 \end{lemma}
\begin{proof}
For any pair $(\x,\y)\in\mathcal{G}^2$, there exists a 1-cycle $\gamma$ in $C_1(\A\cup\B\cup\Gamma)$ such that $\partial(\gamma\cap\B)=\x-\y$. Indeed: First, we choose a 1-chain on $C_1(\B)$ with this property. Then, since there is exactly one intersection point on every $\alpha$-circle for each $\x$ and $\y$, we can add 1-chains in $C_1(\Ac)$ such that the boundary of the new 1-chain lies on the $\alpha$-arcs only. But since $\gamma$ is allowed to have $\Gamma$-components, we can get rid of these intersection points, too, and obtain our cycle $\gamma$.

Next, we can add $\Gamma$-cycles to $\gamma$ such that the resulting 1-cycle is 0 in $H_1(M_T)$. Adding $\alpha$- and $\beta$-cycles gives us another 1-cycle $\gamma'$ which is 0 in $H_1(\Sigma_g)$ and also satisfies $\partial(\gamma'\cap\B)=\x-\y$.
\end{proof}
Our next goal is to define a relative Alexander grading on generators. We do this by counting $\Gamma$-components of domains which connect two generators. 
In the following, let us denote the set of circles in $\Gamma$ which meet the $\alpha$-curves by $\Gamma^\alpha$.

\begin{figure}[t]
	\setlength\mathsurround{0pt}
	$$
	\begin{tikzcd}[column sep=0.63cm,row sep=1.2cm]    
	D_\mathcal{H}\arrow{r}{\partial} &H_1(\A\cup\B\cup\Gamma) \arrow[color=white]{dr}[description]{\quad\textcolor{black}{\square}}
	\arrow{r}{\iota}\arrow{d}{f} & H_1(\Sigma_g)\arrow[color=white,near start]{drr}[description]{\textcolor{black}{\square}}\arrow{d}{\cong}\arrow{rr}{j}& &H_1(\Sigma_g)/\langle\Ac,\B\rangle&\hspace{-0.74cm}\cong H_1(M_T)\\
	& H_1((\A\cup\B\cup\Gamma)^f) \arrow{r}{\iota^f}\arrow[swap]{d}{\cong}& H_1(\Sigma_g^f)\arrow[swap,bend right=5]{urr}{j^{f}}\arrow[color=white, very near start]{drrr}[description]{\textcolor{black}{\square}}&&\phantom{X}\\
	& H_1(\A^f\cup\B)\oplus H_1(\Gamma^f)\arrow{r}{p_\Gamma}& H_1(\Gamma^f)\arrow[bend right=35,swap,start anchor=east,end anchor=south]{uurr}{k=j^f\circ i}\arrow{u}{i}&&&\phantom{X}
	\end{tikzcd}
	$$
	\caption{Diagram for the definition of $A^f$ in definition~\ref{def:AlexGradingOnDomains}. The maps $\iota$, $\iota^f$, $i$ and $j$ are induced by inclusions, and $p_\Gamma$ is the projection onto the second summand.}\label{fig:AlexgradingOnDomains}
\end{figure}

\begin{definition}\label{def:AlexGradingOnDomains}
By definition, the endpoints of the $\alpha$-arcs divide each circle in $\Gamma^\alpha$ into two components. In the following, we will suggestively call them the front and the back component, but eventually it will not matter which is which. Let $\Sigma_g^f$, $(\A\cup\B\cup\Gamma)^f$ and $\A^f$ denote the spaces obtained from $\Sigma_g$, $(\A\cup\B\cup\Gamma)$ and $\A$ respectively by contracting the back component of each circle in $\Gamma^\alpha$ to a point. 
Note that $\B\cap\Gamma=\emptyset$, $\Ac\cap\Gamma=\emptyset$ and that the images of $\alpha$-arcs become a single circle in $\A^f$. Let $f\co(\A\cup\B\cup\Gamma)\rightarrow(\A\cup\B\cup\Gamma)^f$ be the quotient map and $\partial\co D_\mathcal{H}\rightarrow H_1(\A\cup\B\cup\Gamma)$ the boundary map of the long exact sequence of the pair $(\Sigma_g,\A\cup\B\cup\Gamma)$. Now, we consider the diagram from figure~\ref{fig:AlexgradingOnDomains} and define 
$$A^f\co D_\mathcal{H}\rightarrow H_1(M_T)$$
as the composition $k\circ p_\Gamma\circ f\circ\partial$. 
Similarly, by contracting the front components of $\Gamma^\alpha$, we obtain a homomorphism
$$A^b\co D_\mathcal{H}\rightarrow H_1(M_T).$$
Finally, note that the orientation of~$T$ induces an orientation of the meridians using the right-hand rule,  which gives rise to a canonical identification $H_1(M_T)\cong \mathbb{Z}^{n+m}$.
\end{definition}
\begin{lemma}\label{lem:AlexFconstant}
\(A^f\) and \(A^b\) are constant on \(\pi_2(\x,\y)\) for each pair \((\x,\y)\in\mathcal{G}^2\).
\end{lemma}
\begin{proof}
It suffices to show that $A^f$ and $A^b$ vanish on periodic domains. But this is obvious from the description of periodic domains in lemma \ref{lem:noadmissibilityissues} in terms of elementary periodic domains. The annuli have cancelling $\Gamma$-components of the corresponding tangle component and the $\Gamma$-components of the two discs are the sums of all front/back $\Gamma$-components.
\end{proof}

Combining lemmas \ref{lem:pixynonempty} and \ref{lem:AlexFconstant} enables us to define a relative grading on $\mathcal{G}$.
\begin{definition} \label{def:AlexgradingonGenerators}
The homomorphism $A=A^f+A^b\co D_\mathcal{H}\rightarrow \mathbb{Z}^{n+m}$ induces a relative grading $A\co\mathcal{G}\rightarrow \mathbb{Z}^{n+m}$ by setting $$A(\y)-A(\x)=A(\phi)\quad\text{for } \phi\in\pi_2(\x,\y).$$ 
We call $A$ the \textbf{Alexander grading}. 
\end{definition}

\begin{lemma}\label{lem:pidxynonempty}
 For each pair \((\x,\y)\in\mathcal{G}^2\), \(\pip(\x,\y)\) is non-empty iff \(\x\) and \(\y\) are in the same Alexander grading and belong to the same site.
 \end{lemma}
\begin{proof}
The only-if part is clear. The opposite direction follows from a refinement of the proof of \ref{lem:pixynonempty}: We can now get a 1-cycle~$\gamma$ in $C_1(\A\cup\B)$ such that $\partial(\gamma\cap\B)=\x-\y$, because the generators belong to the same site. This 1-cycle is already zero in $H_1(M_T)$, since the generators are in the same Alexander grading. Then we might have to add $\alpha$- and $\beta$-cycles as before and we are done.
\end{proof} 

\begin{Remark}\label{rem:AlexanderIdentification}
	The type D structure $\BSD(X)$ of a bordered sutured manifold~$X$ comes with an absolute grading by relative $\Spinc$-structures $\Spinc(X,\partial X\smallsetminus F)$, where $F$ is the surface parametrized by the arc diagram on the boundary of~$X$. In~\cite[section~4.5]{ZarevThesis}, Zarev identifies $\Spinc(X,\partial X\smallsetminus F)$ with an affine copy of $H^2(X,\partial X\smallsetminus F)$, which by Lefschetz duality is isomorphic to $H_1(X,F)$. Then, in~\cite[proof of proposition~4.5.2]{ZarevThesis}, Zarev associates with a pair of generators $\x,\y\in\mathcal{G}$ a homology class $\epsilon(\x,\y)=[a-b]\in H_1(X,F)$, where $a$ and $b$ are 1-chains on the $\alpha$- and $\beta$-curves, respectively, such that $\partial b=\y-\x$ and $\partial a=\y-\x+\z$ for some 0-chain $\z$ in $Z$.

	If the bordered sutured manifold in question is $M_T$, $F$ is a set of $2n$ strips connecting the sutures at the tangle ends. Then, by lemma~\ref{lem:pixynonempty}, we may choose a domain $\phi\in\pi_2(\x,\y)$ and take $a=\partial\phi\cap\A$ and $b=-\partial\phi\cap\B$. If $c=\partial\phi\cap\Gamma$, then $\epsilon(\x,\y)=[a-b]=[-c]\in H_1(M_T,F)$. 
	Like in definition~\ref{def:AlexGradingOnDomains}, we can define $(M_T^f,F^f)$ to be the pair of spaces obtained by contracting the front components of $\Gamma^\alpha$. Then the quotient map induces a homomorphism
	$$\tilde{A}^f\co H_1(M_T,F)\rightarrow H_1(M_T^f,F^f)\cong H_1(M_T)$$
	such that, with the same notation as in said definition, $$\tilde{A}^f(\epsilon(\x,\y))=\tilde{A}^f([-c])=\tilde{A}^f([-\partial\phi\cap\Gamma])=-k(p_\Gamma(f(\partial\phi)))=-A^f(\phi).$$
	Similarly, we can define a homomorphism
	$$\tilde{A}^b\co H_1(M_T,F)\rightarrow H_1(M_T^b,F^b)\cong H_1(M_T)$$
	such that $\tilde{A}^b(\epsilon(\x,\y))=-A^b(\phi)$. Thus, we may write our Alexander grading as a quotient of the relative $H_1(M_T,F)$-grading via the homomorphism $-(\tilde{A}^f+\tilde{A}^b)$.
\end{Remark}
\begin{lemma}
	The differential on~$\CFT(\mathcal{H}_T,s)$ preserves the Alexander grading.
\end{lemma}
\begin{proof}
	The sum in~\eqref{eq:differential} on page~\pageref{eq:differential} is over domains in $\pip(\x,\y)$, ie those domains that avoid~$\Gamma$, so the Alexander grading of any such domain vanishes.
\end{proof}

Next, we describe a second grading on $\CFT(\mathcal{H}_T)$, using the Maslov index $\mu_{\x,\y}$, which already appeared in the definition of the differential on $\CFT(\mathcal{H}_T)$. The Maslov index plays the role of the formal dimension of the moduli space $\mathcal{M}(\x,\y;\phi)$ of holomorphic curves in the homology class $\phi$ connecting two generators $\x$ and $\y$. The moduli space $\mathcal{M}(\x,\y;\phi)$ comes with a natural $\mathbb{R}$-action. If the expected dimension of $\mathcal{M}(\x,\y;\phi)$ is equal to 1, ie $\mu_{\x,\y}(\phi)=1$, the quotient of the moduli space by this $\mathbb{R}$-action is just a set of points; this is the set $\widehat{\mathcal{M}}(\x,\y;\phi)$ appearing in the definition of the differential from equation~\eqref{eq:differential}.

The Maslov index can be computed combinatorially, as shown in \cite[corollary~4.10]{LipshitzCyl}. We will take this combinatorial formula as a definition.
\begin{definition}\label{def:MaslovIndex}
Let $\phi\in\pi_2(\x,\y)$ for some $\x,\y\in\mathcal{G}$. We define the \textbf{Maslov index} by
$$
\mu_{\x,\y}(\phi)=e(\phi)+m_{\x}(\phi)+m_{\y}(\phi),
$$
where $e(\phi)$ is the Euler measure of $\phi$ and $m_{\x}(\phi)$ and $m_{\y}(\phi)$ are the multiplicities of $\phi$ at $\x$ and~$\y$, respectively. More explicitly, given a region $\psi$ of the Heegaard diagram, let $m_\psi(\phi)$ denote the coefficient of $\psi$ in $\phi$. Then 
$$e(\phi)=\sum_{\text{regions~} \psi}m_\psi(\phi)\left(\chi(\psi)-\tfrac{1}{4}\#\{\text{acute corners of }\psi\}+\tfrac{1}{4}\#\{\text{obtuse corners of }\psi\}\right).$$
Furthermore, for any $\x\in\mathcal{G}$, let
$$m_{\x}(\phi)=\sum_{x_i\in\x}m_{x_i}(\phi),$$ 
where $m_x(\phi)$ is the average of the $m_{\psi_i}(\phi)$ in the four quadrants $\psi_1,\dots,\psi_4$ at $x$. 
\end{definition}
\begin{lemma}\label{lem:muisadditive}
Given \(\phi\in\pi_2(\x,\y)\) and \(\psi\in\pi_2(\y,\z)\), \(\mu_{\x,\y}(\phi)+\mu_{\y,\z}(\psi)=\mu_{\x,\z}(\phi+\psi)\).
\end{lemma}
\begin{proof}
This follows from basically the same arguments as \cite[theorems~3.1 and~3.3]{Sarkar06}. We give some details nonetheless. First of all, note that the Euler measure is additive. Hence, all we need to show is that 
$$m_{\x}(\phi)+m_{\y}(\phi)+m_{\y}(\psi)+m_{\z}(\psi)=m_{\x}(\phi+\psi)+m_{\z}(\phi+\psi).$$
This simplifies to 
$$m_{\y}(\phi)+m_{\y}(\psi)=m_{\x}(\psi)+m_{\z}(\phi).$$
Theorem 3.1 from \cite{Sarkar06} for $n=i=2$, $\eta^1=\A$ and $\eta^2=\B$ gives us
\begin{align*}
m_{\y}(\phi)-m_{\z}(\phi)&=\partial\phi \cdot\partial_{\B}(\psi)\quad\text{and similarly}\\
m_{\x}(\psi)-m_{\y}(\psi)&=\partial\psi \cdot\partial_{\B}(\phi),
\end{align*}
where the product $\cdot$ denotes the ``average'' intersection number from \cite{Sarkar06}. 
So we need to see that 
$$\partial\psi \cdot\partial_{\B}(\phi)+\partial_{\B}(\psi)\cdot\partial\phi=0. $$
The boundaries of the domains lie in $\A\cup\B\cup\Gamma$. However, $\B\cap\Gamma=\emptyset$, so the left-hand side equals $\partial_{\A}(\psi) \cdot\partial_{\B}(\phi)+\partial_{\B}(\psi)\cdot\partial_{\A}(\phi)=\partial_{\A\cup\B}(\psi) \cdot\partial_{\A\cup\B}(\phi)$. To see that this is zero, we modify the Heegaard surface by contracting all boundary components. Then the left-hand side is equal to $\partial(\psi) \cdot\partial(\phi)$, and this is indeed zero.
\end{proof}

\begin{lemma}\label{lem:homgradingconstant}
\(\mu_{\x,\y}\) is constant on \(\pi_2(\x,\y)\) for each pair \((\x,\y)\in\mathcal{G}^2\). 
\end{lemma}
\begin{proof}
Applying the previous lemma to $\z=\y$, we see that it suffices to show that
$$\mu_{\y,\y}(\phi)=e(\phi)+2m_{\y}(\phi)$$
vanishes for all periodic domains $\phi\in\pi_2(\y,\y)$. In fact, we only need to show this for every domain $\phi$ which corresponds to an elementary periodic domain under the identification of lemma~\ref{lem:noadmissibilityissues}. If $\phi$ corresponds to an annulus which is a component of $S_{\B}(\Sigma_g)$, we may regard $\phi$ as a subsurface of $\Sigma_g\times\{1\}$, and the annulus is obtained by performing surgery along all $\beta$-circles contained in the interior of $\phi$ and attaching discs along those $\beta$-curves that lie on the boundary of $\phi$. Any other $\beta$-curves stay away from $\phi$. Performing surgery on the relevant $\beta$-circles increases the Euler measure by 2. However, each such $\beta$-circle is occupied by an intersection point of $\y$, and the contribution of this point to $2m_{\y}(\phi)$ is $2$, since both sides of this $\beta$-circle belong to $\phi$. Likewise, attaching a disc along those $\beta$-curves that lie on the boundary of $\phi$ increases the Euler measure by 1; this is cancelled by the contribution of the intersection point occupying this curve to $2m_{\y}(\phi)$, since $\phi$ lies only to one side of this curve. So we see that removing all intersection points of $\y$ and simultaneously performing surgery/attaching discs does not change the Maslov index. Therefore  $\mu_{\y,\y}(\phi)=e(\text{annulus})=0$.

We can argue similarly if $\phi$ corresponds to an annulus which is a component of $S_{\Ac}(\Sigma_g)$, noting that $m_{y}(\phi)=0$ for any intersection point $y$ of $\y$ which lies on the $\alpha$-arcs. Finally, if $\phi$ corresponds to one of the two discs of $S_{\Ac}(\Sigma_g)$, we see by the same argument as above that $\mu_{\y,\y}(\phi)$ is equal to the Euler measure of a disc with $4n$ corners plus 1 for each of the $(n-1)$ intersection points of $\y$ which lie on an $\alpha$-arc. Hence $\mu_{\y,\y}(\phi)=0$ in this case, too.
\end{proof}
Combining lemmas \ref{lem:muisadditive} and \ref{lem:homgradingconstant}, we can now define a relative grading on generators induced by the Maslov index, just as for the Alexander grading.
\begin{definition}
The \textbf{$\delta$-grading} on generators is a relative $\tfrac{1}{2}\mathbb{Z}$-grading defined by
$$\delta(\y)-\delta(\x)=\mu_{\x,\y}(\phi),$$
where $\phi\in\pi_2(\x,\y)$.
We also define a relative $\mathbb{Z}$-grading, the \textbf{homological grading}, by
$$h(\y)-h(\x)=h(\phi):=\tfrac{1}{2}\overline{A}(\phi)-\mu_{\x,\y}(\phi),$$
where $\phi\in\pi_2(\x,\y)$ and $\overline{A}$ is the composition of $A$ with the map $\mathbb{Z}^{n+m}\rightarrow\mathbb{Z}$ that adds all components. In short,
$$h=\tfrac{1}{2}\overline{A}-\delta.$$
Although we now have three different gradings $\overline{A}$, $\delta$ and $h$ on~\(\CFT(\mathcal{H}_T,s)\), we call their union the \textbf{bigrading}, since any one of them is determined by the other two.
\end{definition}

\begin{Remark}\label{rem:conventions2}
	When comparing these gradings with those in link Floer homology, note that we are using a Heegaard surface with punctures instead of marked points. For two-ended tangles, our conventions agree with those in \cite[section~3.1, equations~(3.2)--(3.4)]{BaldwinLevine} (up to a factor 2 in the Alexander grading), noting that tangle ends that point into the 3-manifold are represented by labels $X$ and outgoing ones by labels $O$. See also figure~\ref{fig:AlexCodesForOneCrossingsWithDelta}.
\end{Remark}

\begin{lemma}
	The homological grading is well-defined.
\end{lemma}
\begin{proof}
	We have claimed in the definition above that the homological grading is a relative $\mathbb{Z}$-grading, but \textit{a priori}, it is only a relative $\frac{1}{2}\mathbb{Z}$-grading. To prove the assertion, let $\phi\in\pi_2(\x,\y)$ be a domain connecting two generators $\x,\y\in\mathcal{G}$. Let us contract each boundary component of the Heegaard surface to a single point. We obtain a closed surface $\Sigma'$ with a certain number of $\beta$-curves and $(n-2)$ fewer $\alpha$-curves, where we regard the curve obtained from the $2n$ $\alpha$-arcs as a single new $\alpha$-circle $\alpha^\ast$, as in the proof of lemma~\ref{lem:noadmissibilityissues}. $\phi$ induces some domain $\phi'$ in this new diagram, and we can still regard $\x$ and $\y$ as tuples of intersection points, even though they do not define generators in the usual sense if $n\neq2$. We claim that $h(\phi)\equiv \mu_{\x,\y}(\phi')\mod 1$. Indeed, contracting a closed component missing the $\alpha$-arcs changes $\mu_{\x,\y}$ by an integer; the contribution of this component to $\tfrac{1}{2}\overline{A}(\phi)$ is an integer, too. Similarly, contracting a component which corresponds to a tangle end increases the Euler measure of each of the two adjacent regions by~$\tfrac{1}{2}$. Hence, if such a region has multiplicity~$m$ in~$\phi$, its contribution to $\mu_{\x,\y}$ increases by~$\tfrac{m}{2}$ while its contribution to $\tfrac{1}{2}\overline{A}(\phi)$ is~$\pm\tfrac{m}{2}$. 
	
	So it remains to show that $\mu_{\x,\y}(\phi')$ is an integer. If $n=2$, this follows from the fact that $\x$ and $\y$ are well-defined generators of a Heegaard diagram and $\mu_{\x,\y}(\phi')$ is equal to the expected dimension of the moduli space $\mathcal{M}(\x,\y;\phi')$. For the general case, we can adapt an argument from~\cite[proof of lemma~4.1]{LipshitzCyl}: Let us regard each region $r$ of the closed Heegaard diagram as an oriented surface $S_r$ with boundary. By adding the fundamental class of $H_2(\Sigma')$ sufficiently many times, we may assume without loss of generality that the multiplicity $m_r$ of each region $r$ in $\phi'$ is non-negative. Then, for each region $r$ in $\phi'$, consider the $m_r$-sheeted trivial cover of $S_r$. Each component of the boundary of these covering spaces minus the corners maps to a part of an $\alpha$- or $\beta$-curve in $\Sigma'$. We can now glue the covering spaces together along matching components to construct a surface $S$ and a map $S\rightarrow \Sigma'$. As in~\cite[proof of lemma~4.1]{LipshitzCyl}, we do this first for all components corresponding to $\alpha$-curves until no more glueing is possible, and then do the same for the $\beta$-curves. Thereby, we ensure that the corners of the surface $S$ are mapped bijectively to non-degenerate points of $\x$ and $\y$, ie points in $(\x\cup\y)\smallsetminus (\x\cap\y)$. Let us call a point in $S$ in the preimage of $\x$ or $\y$ a marked point. Then we may compute $\mu_{\x,\y}(\phi')$ as the sum of the Euler measure of $S$ and for each marked point $x$ on $S$
	$$
	m_x(S):=
	\begin{cases}
	\tfrac{1}{4} & \text{if $x$ is an acute corner of $S$,}\\
	\tfrac{3}{4} & \text{if $x$ is an obtuse corner of $S$,}\\
	\tfrac{1}{2} & \text{if $x$ is lies in the interior of the boundary of $S$,}\\
	1 &	\text{if $x$ lies in the interior of $S$,}
	\end{cases}
	$$
	where we count marked points corresponding to degenerate points twice. 
	The sum of the Euler measure with the sum of $m_x(S)$ over all corners $x$ of $S$ is an integer. Also, we may ignore all marked points corresponding to degenerate points as those points are not corners of $S$ by construction. So it remains to be seen that the number of marked points which lie in the interior of the boundary of $S$ is even. Let us consider each component of the boundary of $S$ minus the corners separately. By construction, each component corresponds to an $\alpha$- or $\beta$-curve, which is occupied by an even number of points in $\x$ and $\y$. So if a component is closed, it contains an even number of marked points. If a component is open then by construction, its two boundary points correspond to a pair of non-degenerate points in $\x$ and $\y$. Unless these two points lie on the special $\alpha$-curve $\alpha^\ast$ which was obtained from the $\alpha$-arcs in the original Heegaard diagram, this component contains an even number of marked points, too. 
	
	It remains to consider the open components corresponding to $\alpha^\ast$. We may assume without loss of generality that the image of each such component is an \textit{embedded} interval on~$\alpha^\ast$, since the number of points in $\x$ and $\y$ on~$\alpha^\ast$ is even. Consider two such intervals. Then, they are either disjoint or they intersect in one or two intervals. In all three cases, the number of endpoints of the two intervals that are contained in the interior of the other interval is either 0, 2 or 4; in particular, it is even. So the total number of marked points in the interior of all such components is even. 
\end{proof}

\begin{Remark}\label{rem:HomologicalIdentification}
	The type D structure $\BSD(X)$ of a bordered sutured manifold $X$ comes with a relative grading $\gru$ by a certain group $\Gru(\mathcal{Z})$, which can be identified with a $\frac{1}{2}\mathbb{Z}$-extension of $H_1(F)$, where $F$ is the surface parametrized by the arc diagram on $X$. In general, this group can be complicated, eg it need not be Abelian. For $X=M_T$ however, $F$ is just a collection of strips, so $H_1(F)$ vanishes. Hence, $\Gru(\mathcal{Z})=\frac{1}{2}\mathbb{Z}$. Moreover, the grading $\gru$ agrees with the unrefined grading $\gr$ in this particular case, and the definition of $\gr$ on domains agrees with our formula for the Maslov index~\cite[6.1.1]{ZarevThesis} up to a sign. We conclude that the grading $\gru$ agrees with our $\delta$-grading.
	Note however, that Zarev only defines $\gru$ for each $\Spinc$-structure separately. We have defined a \textit{uniform} relative $\frac{1}{2}\mathbb{Z}$-grading for all $\Spinc$-structures/Alexander gradings simultaneously. 
\end{Remark}

Let us now summarize the discussion in this and the preceding section.

\begin{theorem}\label{thm:HFTiswelldefandinvariant}
	Let $\mathcal{H}_T$ be a Heegaard diagram for an oriented tangle $T$ inside a $\mathbb{Z}$-homology 3-ball $M$ with spherical boundary and $s$~a~site of~$T$. Then \(\CFT(\mathcal{H}_T,s)\) carries an Alexander grading for each component of the tangle and a homological grading, each of which is a relative $\mathbb{Z}$-grading. Furthermore, the differential \eqref{eq:differential} is well-defined, preserves Alexander gradings and decreases the homological grading by 1. Finally, the bigraded chain homotopy type of \(\CFT(\mathcal{H}_T,s)\) is an invariant of the tangle~\(T\).
\end{theorem}
\begin{proof}
	In lemmas~\ref{lem:AlexFconstant} and~\ref{lem:homgradingconstant}, we showed that the gradings are well-defined. The fact that the differential from~\eqref{eq:differential} is well-defined and respects the gradings follows from the definition of \(\CFT(\mathcal{H}_T,s)\) as the bordered sutured invariant $\iota_s.\BSD(\mathcal{H}_T)$, see definition~\ref{def:HFTfromBSD}, and the identification of the gradings with the $\Spinc$- and $\Gru(\mathcal{Z})$-grading on~$\iota_s.\BSD(\mathcal{H}_T)$, see remarks~\ref{rem:AlexanderIdentification} and~\ref{rem:HomologicalIdentification}. 
	Invariance of \(\CFT(\mathcal{H}_T,s)\) as an Alexander graded chain complex and a $\delta$-graded complex in each fixed Alexander grading follows likewise. It only remains to show that the $\delta$-grading between generators of different Alexander gradings and sites is preserved by Heegaard moves. These are purely local arguments, so the proof of this fact follows from the same argument as for closed Heegaard surfaces. See for example~\cite[theorem~3.4]{Sarkar06} for invariance under isotopies; invariance under handleslides and stabilisation follows similarly. We leave the details to the reader.
\end{proof}

\begin{definition}\label{def:CFT}
	Let us write $\CFT_h(\mathcal{H}_T,s,a)$ for the subgroup of  $\CFT(\mathcal{H}_T,s)$ generated by those elements in $\mathcal{G}^s$ of Alexander grading $a\in\mathbb{Z}^{n+m}$ and homological grading $h$. We then define the \textbf{graded Euler characteristic of \(\CFT(\mathcal{H}_T,s)\)} as 
	$$
	\chi(\CFT(\mathcal{H}_T,s))=\sum_{h, a} (-1)^h\operatorname{rk}(\CFT_h(\mathcal{H}_T,s,a))\cdot t_{1\phantom{\vert}}^{a_{1\phantom{\vert}}}\!\!\cdots t_{\vert T\vert}^{a_{\vert T\vert}}\in\mathbb{Z}[t_1^{\pm1},\dots,t_{\vert  T\vert}^{\pm1}]
	$$
	which is well-defined up to multiplication by a unit in \(\mathbb{Z}[t_1^{\pm1},\dots,t_{\vert  T\vert}^{\pm1}]\).
\end{definition}

\begin{theorem}\label{thm:Eulercharagreeswithnabla}
	For an oriented tangle \(T\) in \(M=B^3\), \(\chi(\CFT(\mathcal{H}_T,s))\) coincides with
	$$\nabla_T^s(t_1,\dots,t_{\vert  T\vert})\cdot\prod(t_i-t_i^{-1}) $$
	up to multiplication by a unit, where the product is over all closed components of~\(T\). 	
\end{theorem}

\begin{proof}[Proof of lemma~\ref{lem:ExponentsMod2Agree}]\label{proof:ExponentsMod2Agree}
	Assuming the theorem above, we can now easily prove lemma~\ref{lem:ExponentsMod2Agree}:
	For any two generators of the same site $s$ of $T$, we can find a connecting domain $\phi$, such that $\partial\phi\cap\partial\Sigma_g$ consists of closed components only. Thus the difference of the Alexander gradings between the two generators is even for each color.  
\end{proof}

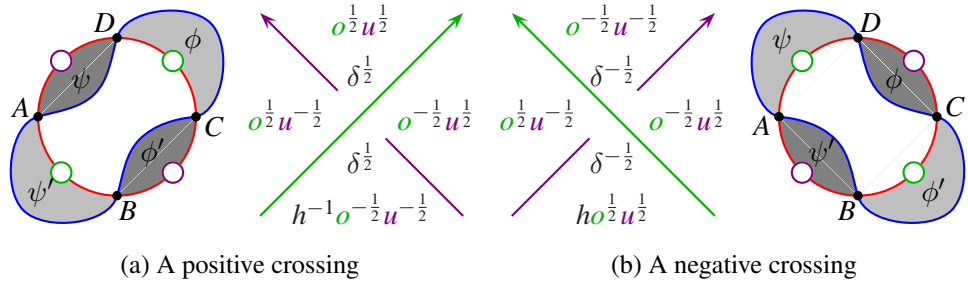
\begin{figure}[t]
	\centering
	\begin{subfigure}[b]{0.49\textwidth}\centering
		{\psset{unit=0.15}
			\begin{pspicture}(-10.3,-10.3)(10.3,10.3)
			% outer
			%\psecurve*[linecolor=lightgray](10,0)(9,9)(0,10)(-9,9)(-10,0)(-9,-9)(0,-10)(9,-9)(10,0)(9,9)(0,10)
			%\pscircle*[linecolor=white]{7}
			
			% top right
			\pscustom*[linecolor=lightgray,linewidth=0pt]{\psecurve(4,4)(7,0)(8.3,8.3)(0,7)(4,4)}
			\pscustom*[linecolor=white,linewidth=1pt]{
				\psarc(0,0){7}{0}{45}
				\psarc(0,0){7}{45}{90}
				\psarc(7,7){7}{180}{-90}}
			
			% bottom right
			\pscustom*[linecolor=gray,linewidth=0pt]{
				\psarc(0,0){7}{-90}{-45}
				\psarc(0,0){7}{-45}{0}
				\psecurve(-0.2,-8)(0,-7)(2,-2)(7,0)(8,-0.2)}
			
			% top left
			\pscustom*[linecolor=gray,linewidth=0pt]{
				\psarc(0,0){7}{90}{135}
				\psarc(0,0){7}{135}{180}
				\psecurve(0.2,8)(0,7)(-2,2)(-7,0)(-8,-0.2)}
			
			% bottom left
			\pscustom*[linecolor=lightgray,linewidth=0pt]{\psecurve(-4,-4)(-7,0)(-8.3,-8.3)(0,-7)(-4,-4)}
			\pscustom*[linecolor=white,linewidth=1pt]{
				\psarc(0,0){7}{180}{-90}
				\psarc(-7,-7){7}{0}{90}}
			
			% inner
			%\pscircle*[linecolor=lightgray,linewidth=0pt]{7}

			\psecurve[linecolor=blue](-0.2,-8)(0,-7)(2,-2)(7,0)(8,-0.2)
			\psecurve[linecolor=blue](4,4)(7,0)(8.3,8.3)(0,7)(4,4)
			\psecurve[linecolor=blue](0.2,8)(0,7)(-2,2)(-7,0)(-8,-0.2)
			\psecurve[linecolor=blue](-4,-4)(-7,0)(-8.3,-8.3)(0,-7)(-4,-4)
			
			\psarc[linecolor=red](0,0){7}{0}{360}
			
			\SpecialCoor
			\rput(7;45){\pscircle*[linecolor=white]{1}\pscircle[linecolor=darkgreen]{1}}
			\rput(7;135){\pscircle*[linecolor=white]{1}\pscircle[linecolor=violet]{1}}
			\rput(7;225){\pscircle*[linecolor=white]{1}\pscircle[linecolor=darkgreen]{1}}
			\rput(7;315){\pscircle*[linecolor=white]{1}\pscircle[linecolor=violet]{1}

				%\psecurve(10,0)(9,9)(0,10)(-9,9)(-10,0)(-9,-9)(0,-10)(9,-9)(10,0)(9,9)(0,10)
			}
			\psdot(0,7)
			\psdot(7,0)
			\psdot(0,-7)
			\psdot(-7,0)
			
			\SpecialCoor
			\rput(9.5;45){$\phi$}
			\rput(4.5;135){$\psi$}
			\rput(4.5;-45){$\phi'$}
			\rput(9.5;-135){$\psi'$}
			
			\rput[br](0,7.5){$D$}
			\rput[tl](7.5,0){$C$}
			\rput[tl](0,-7.5){$B$}
			\rput[br](-7.5,0){$A$}
			\end{pspicture}}
		{\psset{unit=1.5}
			\begin{pspicture}(-1.05,-1.05)(1.05,1.05)
			\psline[linecolor=violet]{->}(0.9,-0.9)(-0.9,0.9)
			\pscircle*[linecolor=white](0,0){0.3}
			\psline[linecolor=darkgreen]{->}(-0.9,-0.9)(0.9,0.9)
			\uput{0.7}[90](0,0){$\textcolor{darkgreen}{o}^{\frac{1}{2}} \textcolor{violet}{u}^{\frac{1}{2}}$}
			\uput{0.3}[180](0,0){$\textcolor{darkgreen}{o}^{\frac{1}{2}} \textcolor{violet}{u}^{-\frac{1}{2}}$}
			\uput{0.7}[-90](0,0){$h^{-1}\textcolor{darkgreen}{o}^{-\frac{1}{2}} \textcolor{violet}{u}^{-\frac{1}{2}}$}
			\uput{0.3}[0](0,0){$\textcolor{darkgreen}{o}^{-\frac{1}{2}} \textcolor{violet}{u}^{\frac{1}{2}}$}
			\uput{0.25}[90](0,0){$\delta^{\frac{1}{2}}$}
			\uput{0.25}[-90](0,0){$\delta^{\frac{1}{2}}$}
			\end{pspicture}}
		\caption{A positive crossing}
	\end{subfigure}
	\begin{subfigure}[b]{0.49\textwidth}\centering
		{\psset{unit=1.5}
			\begin{pspicture}(-1.05,-1.05)(1.05,1.05)
			\psline[linecolor=violet]{->}(-0.9,-0.9)(0.9,0.9)
			\pscircle*[linecolor=white](0,0){0.3}
			
			\psline[linecolor=darkgreen]{->}(0.9,-0.9)(-0.9,0.9)
			\uput{0.7}[90](0,0){$\textcolor{darkgreen}{o}^{-\frac{1}{2}} \textcolor{violet}{u}^{-\frac{1}{2}}$}
			\uput{0.3}[180](0,0){$\textcolor{darkgreen}{o}^{\frac{1}{2}} \textcolor{violet}{u}^{-\frac{1}{2}}$}
			\uput{0.7}[-90](0,0){$h\textcolor{darkgreen}{o}^{\frac{1}{2}} \textcolor{violet}{u}^{\frac{1}{2}}$}
			\uput{0.3}[0](0,0){$\textcolor{darkgreen}{o}^{-\frac{1}{2}} \textcolor{violet}{u}^{\frac{1}{2}}$}
			\uput{0.25}[90](0,0){$\delta^{-\frac{1}{2}}$}
			\uput{0.25}[-90](0,0){$\delta^{-\frac{1}{2}}$}
			\end{pspicture}}
		{\psset{unit=0.15}
			\begin{pspicture}(-10.3,-10.3)(10.3,10.3)
			% outer
			%\psecurve*[linecolor=lightgray](10,0)(9,9)(0,10)(-9,9)(-10,0)(-9,-9)(0,-10)(9,-9)(10,0)(9,9)(0,10)
			%\pscircle*[linecolor=white]{7}

			% top left
			\pscustom*[linecolor=lightgray,linewidth=0pt]{\psecurve(-4,4)(-7,0)(-8.3,8.3)(0,7)(-4,4)}
			\pscustom*[linecolor=white,linewidth=1pt]{
				\psarcn(0,0){7}{180}{135}
				\psarcn(0,0){7}{135}{90}
				\psarc(-7,7){7}{-90}{0}}
			
			% bottom left
			\pscustom*[linecolor=gray,linewidth=0pt]{
				\psarcn(0,0){7}{-90}{-135}
				\psarcn(0,0){7}{-135}{-180}
				\psecurve(0.2,-8)(0,-7)(-2,-2)(-7,0)(-8,-0.2)}
			
			% top right
			\pscustom*[linecolor=gray,linewidth=0pt]{
				\psarcn(0,0){7}{90}{45}
				\psarcn(0,0){7}{45}{0}
				\psecurve(-0.2,8)(0,7)(2,2)(7,0)(8,-0.2)}
			
			% bottom right
			\pscustom*[linecolor=lightgray,linewidth=0pt]{\psecurve(4,-4)(7,0)(8.3,-8.3)(0,-7)(4,-4)}
			\pscustom*[linecolor=white,linewidth=1pt]{
				\psarcn(0,0){7}{0}{-45}
				\psarcn(0,0){7}{-45}{-90}
				\psarc(7,-7){7}{90}{180}}
			
			% inner
			%\pscircle*[linecolor=lightgray,linewidth=0pt]{7}

			\psecurve[linecolor=blue](0.2,-8)(0,-7)(-2,-2)(-7,0)(-8,-0.2)
			\psecurve[linecolor=blue](-4,4)(-7,0)(-8.3,8.3)(0,7)(-4,4)
			\psecurve[linecolor=blue](-0.2,8)(0,7)(2,2)(7,0)(8,-0.2)
			\psecurve[linecolor=blue](4,-4)(7,0)(8.3,-8.3)(0,-7)(4,-4)
			
			\psarc[linecolor=red](0,0){7}{0}{360}
			
			\SpecialCoor
			\rput(7;45){\pscircle*[linecolor=white]{1}\pscircle[linecolor=violet]{1}}
			\rput(7;135){\pscircle*[linecolor=white]{1}\pscircle[linecolor=darkgreen]{1}}
			\rput(7;225){\pscircle*[linecolor=white]{1}\pscircle[linecolor=violet]{1}}
			\rput(7;315){\pscircle*[linecolor=white]{1}\pscircle[linecolor=darkgreen]{1}

				%\psecurve(10,0)(9,9)(0,10)(-9,9)(-10,0)(-9,-9)(0,-10)(9,-9)(10,0)(9,9)(0,10)
			}
			\psdot(0,7)
			\psdot(7,0)
			\psdot(0,-7)
			\psdot(-7,0)
			
			\SpecialCoor
			\rput(4.5;45){$\phi$}
			\rput(9.5;135){$\psi$}
			\rput(9.5;-45){$\phi'$}
			\rput(4.5;-135){$\psi'$}
			
			\rput[bl](0,7.5){$D$}
			\rput[bl](7.5,0){$C$}
			\rput[tr](0,-7.5){$B$}
			\rput[tr](-7.5,0){$A$}
			
			\end{pspicture}}
		\caption{A negative crossing}
	\end{subfigure}
	\caption{The gradings of the generators of the two 1-crossing tangles. The over-strand is coloured by $\textcolor{darkgreen}{o}$ and the under-strand by $\textcolor{violet}{u}$. Compare this to the Alexander codes from figure~\ref{figAlexCodesForNabla}. Our conventions agree with~\cite[figure~5 and~6]{OSrescube} and \cite[figure~11]{BaldwinLevine}, up to a factor 2 in the Alexander grading.}\label{fig:AlexCodesForOneCrossingsWithDelta}
\end{figure}
\begin{proof}[Proof of theorem~\ref{thm:Eulercharagreeswithnabla}]
	We first calculate the gradings of the generators for the 1-crossing diagrams, see figure~\ref{fig:AlexCodesForOneCrossingsWithDelta}.
	In each case, we have four connecting domains $\psi$, $\phi$, $\psi'$ and~$\phi'$. The $\delta$-grading of all these domains is $+\frac{1}{2}$. This gives us the correct relative $\delta$-grading on generators, noting that the normal vector field of the Heegaard diagram, determined by the right-hand rule, points into the plane. (For example, for the positive crossing, $\psi$ is in~$\pi_2(A,D)$ and the $\delta$-grading increases along this domain by~$+\frac{1}{2}$.) Using the right-hand rule convention from definition~\ref{def:AlexGradingOnDomains}, we similarly obtain the correct relative Alexander gradings. This determines the homological grading. 
	
	For a general tangle in $B^3$, we can consider the Heegaard diagram induced by a tangle diagram as discussed in example~\ref{exa:HDforonecrossing}. Suppose there are no closed components. Then there is an obvious 1:1-correspondence between Kauffman states and generators: intersection points of generators correspond to markers of Kauffman states, $\beta$-curves to crossings and $\alpha$-curves to regions of the tangle diagram. 
	So it suffices to check that the gradings agree on both sides of this correspondence, up to an overall shift. Locally, at the crossings, we have already done this. For Kauffman states, the gradings can be computed as the sum of the local gradings. For generators in our chosen Heegaard diagram, the same is true, which can be checked on the level of domains. For the Alexander grading, this is obvious. For the Maslov grading, it follows from the fact that the Euler measure has this property, too, which is an elementary exercise. 
	
	If the tangle has a closed component, we need to insert a ladybug into the Heegaard diagram for each such component, see figure~\ref{fig:ladybug}; this multiplies the number of generators by two, since there are two intersection points of the $\alpha$-circle in a ladybug. It is straightforward to compute the grading difference between corresponding generators of the two intersection points: the $\delta$-gradings agree and the Alexander gradings differ by~2. Hence, we get an extra factor $(t_i-t_i^{-1})$ in the decategorified invariant.
\end{proof}

\begin{Remark}
	For tangles in $B^3$ without closed components, the Mathematica program \cite{APT.m} explicitly computes the generators of the categorified tangle invariant from a standard Heegaard diagram as in the proof above. Note that the relative gradings of the generators computed by the program are lifted to absolute ones in such a way that they agree with the gradings of the corresponding Kauffman states. For tangles in $B^3$ with closed components, we need to introduce a factor $(t_i+h^{-1}t_i^{-1})$ for each closed component. 
\end{Remark}

	% HFTviaSFH.tex
\section{\texorpdfstring{Basic properties of $\HFT$}{Basic properties of HFT}}\label{sec:Properties}

\begin{wrapfigure}{r}{0.3333\textwidth}
	\centering
	\psset{unit=0.2}
	\begin{pspicture}(-10,-10)(10,10)
	\pscircle(0,0){10}
	\psarc[linecolor=red](0,0){7}{45}{135}
	\psarc[linecolor=red,linestyle=dotted](0,0){7}{142}{215}
	\psarc[linecolor=red,linestyle=dotted](0,0){7}{235}{305}
	\psarc[linecolor=red,linestyle=dotted](0,0){7}{325}{398}
	
	\pscircle[linecolor=lightgray](0,0){4}
	
	\psline[linecolor=white,linewidth=4pt](3;45)(6.3;45)
	\psline[linecolor=white,linewidth=4pt](3;135)(6.3;135)
	\psline[linecolor=white,linewidth=4pt](3;-135)(6.3;-135)
	\psline[linecolor=white,linewidth=4pt](3;-45)(6.3;-45)
	
	\psline[linecolor=gray](3;45)(6.3;45)
	\psline[linecolor=gray](3;135)(6.3;135)
	\psline[linecolor=gray](3;-135)(6.3;-135)
	\psline[linecolor=gray](3;-45)(6.3;-45)

	\rput{-135}(7;-135){\psellipse[linecolor=darkgreen](0,0)(0.5,1)}
	\rput{-45}(7;-45){\psellipse[linecolor=darkgreen](0,0)(0.5,1)}
	
	\pscustom[linecolor=darkgreen]{
		\psarc(0,0){6.5}{45}{135}
		\psecurve(7;130)(6.5;135)(7;140)(7.5;135)(7;130)
		\psarcn(0,0){7.5}{135}{45}
		\psecurve(7;50)(7.5;45)(7;40)(6.5;45)(7;50)
	}

	\rput(-5.5,0){$a$}
	\rput(0,-5.5){$b$}
	\rput(5.5,0){$c$}
	\rput(0,5){$d$}
	
	\rput(0,0){\textcolor{darkgray}{$T$}}
	
	\end{pspicture}
	\caption{The set of sutures (\textcolor{darkgreen}{green} curves) on $M_T^s$ for a 4-ended tangle $T$ and $s=d$. Any closed components of $T$ get two meridional sutures as in the case of knots and links.}\label{fig:HFTviaSFHsutmfd}
\end{wrapfigure}
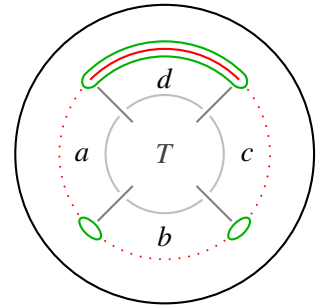

In this section, $M$ still denotes a $\mathbb{Z}$-homology 3-ball with spherical boundary, $T$~is an oriented tangle in~$M$ and $s$ a site of $T$. 
Before discussing basic properties of our homological invariant $\HFT$, let us interpret it in terms of sutured Floer homology $\SFH$. $\SFH$ was originally developed by Juhász in~\cite{Juhasz} as a generalisation of the hat version of Heegaard Floer homology of closed three manifolds and links therein to balanced sutured manifolds, see definition~\ref{def:SuturedManifold}. 
He used $\SFH$ to give short proofs of a number of known results, eg that link Floer homology detects the Thurston norm and fibredness. Juhász also proved a surface decomposition formula, which says that $\SFH$ behaves very nicely under splitting a balanced sutured manifold along certain embedded surfaces. For basic definitions and properties of $\SFH$, we refer the reader to Juhász's original papers \cite{Juhasz,SurfaceDecomposition,polytope} and Altman's introductory article~\cite{Altman}.

\begin{definition}\label{def:sutured3mfdForTangles}
	With a tangle $T$ in $M$ and a site~$s$ of~$T$, we associate the sutured 3-manifold $M_T^s$ defined as follows: The underlying 3-manifold with boundary is the tangle complement $M_T$. The sutures on $\partial M_T$ are obtained by placing two oppositely oriented meridional circles around closed components of the tangle and meridional circles around the ends of the open components and performing surgery along the arcs in $s$, see figure~\ref{fig:HFTviaSFHsutmfd}. We orient the sutures such that one component of $R_-$ is contained in the boundary of the 3-ball.
\end{definition}

\begin{lemma}
	$M_T^s$ is balanced.
\end{lemma}
\begin{proof}
	Say $T$ has $n$ open components and without loss of generality, we may assume that there are no closed components. The site $s$ consists of $(n-1)$ open regions, so there are $(n-1)$ arcs which we have performed surgery along. Hence, $R_-$ is a sphere with $(n+1)$ punctures, so it has Euler characteristic $(1-n)$. Each annulus around an open tangle component contributes 0 to the Euler characteristic, but each surgery decreases the Euler characteristic by 1.
\end{proof}

\begin{theorem}\label{thm:HFTasSFT}
	$\HFT(T,s)=\SFH(M_T^s)$.
\end{theorem}

\begin{Remark}
	As in bordered sutured theory, the homological grading in sutured theory is usually only defined for each $\Spinc$-structure separately. However, like in the proof of theorem~\ref{thm:HFTiswelldefandinvariant}, one can show that in our case, $\SFH(M_T^s)$ is well-defined as a \textit{uniformly} $\delta$-graded invariant.
\end{Remark}

\begin{proof}
	The main idea is to modify a Heegaard diagram~$\mathcal{H}$ for the tangle $T$ in the following way, as illustrated in figure~\ref{fig:relationHFTandSFH}: 
	Let us consider a 2-torus with a fixed longitude and $2n$ disjoint meridians. Puncture the torus $2n$ times along the longitude such that any two meridians are no longer homotopic. We consider the remaining segments of the longitude as $\alpha$-arcs and the meridians as $\beta$-circles. We may assume that each $\beta$-circle intersects exactly one $\alpha$-arc in a single point and there are exactly $2n$ connected components in their complement on the punctured torus. We place a puncture in each of these components. Finally, we attach the (now $4n$-punctured) torus to $\Sigma$ in such a way that each $\alpha$-arc in the torus closes an $\alpha$-arc in the Heegaard surface $\Sigma_g$ for our tangle. This gives us a sutured Heegaard diagram $\overline{\mathcal{H}}$ consisting of a $2(n+m)$-punctured surface $\overline{\Sigma}$ with $(g+2n+m)$ $\alpha$-circles and $(g+3n+m-1)$ $\beta$-circles. However, $\overline{\mathcal{H}}$ is not balanced unless $n=1$.
	
	So, let us fix a site $s$. By definition, $s$ is a set of $\alpha$-arcs in $\mathcal{H}$ that are occupied by generators in $\mathcal{G}^s$. An $\alpha$-arc in $\mathcal{H}$ corresponds to an $\alpha$-arc in the punctured torus which in turn corresponds to the $\beta$-circle that it intersects. Thus, a site $s$ gives rise to a collection of $\beta$-circles on the punctured torus. For each such circle $\beta_i$, we pick a path $\gamma_i$ between the two adjacent punctures (the dashed line in figure~\ref{fig:relationHFTandSFH}) which intersects no $\alpha$-circle and no $\beta$-circle except $\beta_i$. We delete $\beta_i$ and cut the surface along $\gamma_i$. This gives us a new sutured Heegaard diagram which \emph{is} balanced. Let us denote it by $\overline{\mathcal{H}}^s$. 
	
	Now observe that generators in $\mathcal{G}^s$ correspond to generators in $\overline{\mathcal{H}}^s$ and that domains in~$\mathcal{H}$ that avoid $\Gamma$ correspond to domains in $\overline{\mathcal{H}}^s$ that avoid the boundary. Furthermore, if we started with an admissible Heegaard diagram $\mathcal{H}$, then $\overline{\mathcal{H}}^s$ is also admissible. Hence, the sutured Floer homology $\SFH(\overline{\mathcal{H}}^s)$ is well-defined and identical to $\HFT(T,s)$. Now, $\overline{\mathcal{H}}^s$ is a Heegaard diagram for $M_T^s$.
	Finally, $\SFH(M_T^s)$ also comes with two gradings, a $\Spinc(M_T^s,\Gamma)$-grading, which gives a relative $H_1(M_T)$-grading, and a homological grading. In~\cite[comment after definition~4.6]{Juhasz}, Juhász describes the first grading in terms of a homology class $\epsilon(\x,\y)$ like in remark~\ref{rem:AlexanderIdentification}. The homological grading uses the Maslov index, so it also agrees with our definition.  
\end{proof}

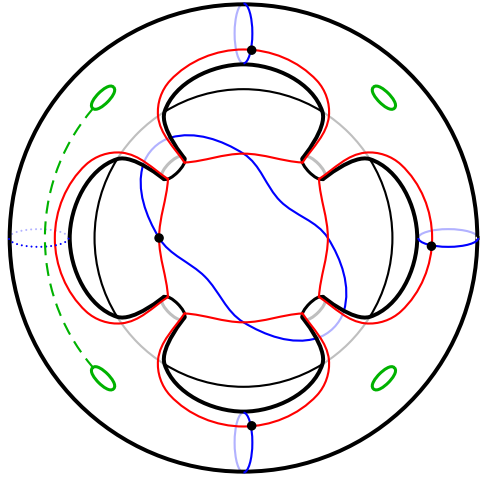
\begin{figure}[t]
	\centering
	\psset{unit=0.33}
	\begin{pspicture}(-9.55,-9.55)(9.55,9.55)
	
	%% beta curve on original tangle HD
	\psecurve[linecolor=blue](5;-36)(5;-54)(3.4;-90)(-2.2;45)(3.4;180)(5;144)(5;126)
	\psecurve[linecolor=blue](5;144)(5;126)(3.4;90)(2.1;45)(3.4;0)(5;-36)(5;-54)
	\psecurve[linecolor=blue!30!white](3.4;180)(5;144)(5;126)(3.4;90)
	\psecurve[linecolor=blue!30!white](3.4;0)(5;-36)(5;-54)(3.4;-90)

	%% sphere of original tangle HD
	%\pscircle(0,0){6}
	
	\rput{0}(0,0){\psarc(0,0){6}{-32}{32}}
	\rput{90}(0,0){\psarc(0,0){6}{-32}{32}}
	\rput{-90}(0,0){\psarc(0,0){6}{-32}{32}}
	\rput{180}(0,0){\psarc(0,0){6}{-32}{32}}
	
	\rput{45}(0,0){\psarc[linecolor=lightgray](0,0){6}{-13}{13}}
	\rput{-45}(0,0){\psarc[linecolor=lightgray](0,0){6}{-13}{13}}
	\rput{135}(0,0){\psarc[linecolor=lightgray](0,0){6}{-13}{13}}
	\rput{-135}(0,0){\psarc[linecolor=lightgray](0,0){6}{-13}{13}}

	%% punctures of original tangle HD
	\rput{45}(4;45){
		\psecurve[linewidth=1.15pt,linecolor=lightgray](-0.25,0)(0,-0.6)(0.25,0)(0,0.6)(-0.25,0)
		\psecurve[linewidth=1.15pt](0.25,0)(0,0.6)(-0.25,0)(0,-0.6)(0.25,0)
	}
	\rput{-45}(4;135){
		\psecurve[linewidth=1.15pt](-0.25,0)(0,-0.6)(0.25,0)(0,0.6)(-0.25,0)
		\psecurve[linewidth=1.15pt,linecolor=lightgray](0.25,0)(0,0.6)(-0.25,0)(0,-0.6)(0.25,0)}
	\rput{45}(4;-135){
		\psecurve[linewidth=1.15pt](-0.25,0)(0,-0.6)(0.25,0)(0,0.6)(-0.25,0)
		\psecurve[linewidth=1.15pt,linecolor=lightgray](0.25,0)(0,0.6)(-0.25,0)(0,-0.6)(0.25,0)}
	\rput{-45}(4;-45){
		\psecurve[linewidth=1.15pt,linecolor=lightgray](-0.25,0)(0,-0.6)(0.25,0)(0,0.6)(-0.25,0)
		\psecurve[linewidth=1.15pt](0.25,0)(0,0.6)(-0.25,0)(0,-0.6)(0.25,0)}
	
	%% blue beta circles on punctured torus (back)
	\psellipticarc[linecolor=blue!30!white](8.25,0)(1.25,0.4){0}{180}
	\rput{90}(0,0){\psellipticarc[linecolor=blue!30!white](8.25,0)(1.25,0.4){0}{180}}
	\rput{180}(0,0){\psellipticarc[linecolor=blue!30!white,linestyle=dotted,dotsep=1pt](8.25,0)(1.25,0.4){180}{360}} %% this one is only for without green curve
	\rput{270}(0,0){\psellipticarc[linecolor=blue!30!white](8.25,0)(1.25,0.4){180}{360}}
	
	%%%%%%%%%%%%%%%%%%%%%%%%%%%%%%%%%%
	%--------------------------------%
	%%%%%%%%%%%%%%%%%%%%%%%%%%%%%%%%%%
	%% transparent punctured torus
	%\pscustom[fillstyle=solid,opacity=0.7,fillcolor=white,linecolor=white,linewidth=0pt]{\pscurve(7;0)(6;32)(4;36.4)
	%\pscurve[liftpen=1](4;53.5)(6;58)(7;90)(6;122)(4;126.5)
	%\pscurve[liftpen=1](4;143.5)(6;148)(7;180)(6;212)(4;216.5)
	%\pscurve[liftpen=1](4;233.5)(6;238)(7;270)(6;302)(4;306.5)
	%\pscurve[liftpen=1](4;-36.4)(6;-32)(7;0)
	%\psarcn[liftpen=1](0,0){9.5}{360}{0}
	%\closepath}
	%%%%%%%%%%%%%%%%%%%%%%%%%%%%%%%%%%
	%--------------------------------%
	%%%%%%%%%%%%%%%%%%%%%%%%%%%%%%%%%%
	
	%% puncured torus
	\pscurve[linewidth=1.5pt,linecap=1](3.98;-36.5)(6;-32)(7;0)(6;32)(3.98;36.5)
	\pscurve[linewidth=1.5pt,linecap=1](3.98;53.5)(6;58)(7;90)(6;122)(3.98;126.5)
	\pscurve[linewidth=1.5pt,linecap=1](3.98;143.5)(6;148)(7;180)(6;212)(3.98;216.5)
	\pscurve[linewidth=1.5pt,linecap=1](3.98;233.5)(6;238)(7;270)(6;302)(3.98;306.5)
	\pscircle[linewidth=1.5pt](0,0){9.5}
	
	%% alpha circles (former alpha-arcs on torus and tangle HD)
	\rput{0}(0,0){\pscustom[linecolor=red]{
			\psecurve(5,-50)(3.86;-38)(3.4;0)(3.86;38)(5,50)
			\psecurve(3;38)(3.86;38)(6;35)(7;25)(7.6;0)(7;-25)(6;-35)(3.86;-38)(3;-38)}
		\psdots[linecolor=red,dotsize=1pt](3.86;-38)(3.86;38)
	}
	\rput{90}(0,0){\pscustom[linecolor=red]{
			\psecurve(5,-50)(3.86;-38)(3.4;0)(3.86;38)(5,50)
			\psecurve(3;38)(3.86;38)(6;35)(7;25)(7.6;0)(7;-25)(6;-35)(3.86;-38)(3;-38)}
		\psdots[linecolor=red,dotsize=1pt](3.86;-38)(3.86;38)
	}
	\rput{180}(0,0){\pscustom[linecolor=red]{
			\psecurve(5,-50)(3.86;-38)(3.4;0)(3.86;38)(5,50)
			\psecurve(3;38)(3.86;38)(6;35)(7;25)(7.6;0)(7;-25)(6;-35)(3.86;-38)(3;-38)}
		\psdots[linecolor=red,dotsize=1pt](3.86;-38)(3.86;38)
	}
	\rput{-90}(0,0){\pscustom[linecolor=red]{
			\psecurve(5,-50)(3.86;-38)(3.4;0)(3.86;38)(5,50)
			\psecurve(3;38)(3.86;38)(6;35)(7;25)(7.6;0)(7;-25)(6;-35)(3.86;-38)(3;-38)}
		\psdots[linecolor=red,dotsize=1pt](3.86;-38)(3.86;38)
	}
	
	%% remaining puncures on torus (opposite the glued ones)
	\rput{45}(8;45){
		\psecurve[linecolor=darkgreen,linewidth=1.15pt](0.25,0)(0,0.6)(-0.25,0)(0,-0.6)(0.25,0)(0,0.6)(-0.25,0)}
	\rput{-45}(8;135){
		\psecurve[linecolor=darkgreen,linewidth=1.15pt](0.25,0)(0,0.6)(-0.25,0)(0,-0.6)(0.25,0)(0,0.6)(-0.25,0)}
	\rput{45}(8;-135){
		\psecurve[linecolor=darkgreen,linewidth=1.15pt](0.25,0)(0,0.6)(-0.25,0)(0,-0.6)(0.25,0)(0,0.6)(-0.25,0)}
	\rput{-45}(8;-45){
		\psecurve[linecolor=darkgreen,linewidth=1.15pt](0.25,0)(0,0.6)(-0.25,0)(0,-0.6)(0.25,0)(0,0.6)(-0.25,0)}
	
	%% beta circles on punctured torus (front)
	\psellipticarc[linecolor=blue](8.25,0)(1.25,0.4){180}{360}
	\rput{90}(0,0){\psellipticarc[linecolor=blue](8.25,0)(1.25,0.4){180}{360}}
	\rput{180}(0,0){\psellipticarc[linecolor=blue,linestyle=dotted,dotsep=1pt](8.25,0)(1.25,0.4){0}{180}}
	\rput{270}(0,0){\psellipticarc[linecolor=blue](8.25,0)(1.25,0.4){0}{180}}
	
	%% generator 
	\psdot(3.4;180)
	\psdot(7.58;87.5)
	\psdot(7.58;-2.5)
	\psdot(7.58;-87.5)
	
	\psarc[linecolor=darkgreen,linestyle=dashed](0,0){8}{139}{-139}
	\end{pspicture}
	\caption{From a Heegaard diagram for a tangle to one for a sutured manifold, illustrating the proof of theorem~\ref{thm:HFTasSFT}. This example shows the Heegaard diagram for a 1-crossing tangle.}\label{fig:relationHFTandSFH}
\end{figure}

The following results can be viewed as the categorified counterparts of some results from sections~\ref{sec:basicpropertiesofnabla} and~\ref{sec:4endedandmutation}.
\begin{proposition}[compare~\ref{prop:mirrortangle}]\label{prop:HFTmirror}
	Given a tangle \(T\) in $M$, let \(\m(T)\) denote its mirror image in the mirror of~$M$. Let \(\widehat{\operatorname{CFT}^\ast\!\!}\,\,(T,s)\) denote the dual chain complex of \(\CFT(T,s)\), with the usual convention that all gradings are reversed. Then 
	$$\CFT(\m(T),s)\cong\widehat{\operatorname{CFT}^\ast\!\!}\,\,(T,s),$$
	as (relatively) bigraded invariants.
\end{proposition}
\begin{proof}
	This follows from the previous theorem and \cite[proposition~2.14]{DecatSFH}.
\end{proof}
\begin{proposition}[compare~\ref{prop:reverseorientI} and~\ref{cor:alloorientsrev}]\label{prop:HFTreverseorientI}
	Let \(T\) be an oriented \(r\)-component tangle and \(s\) a site of \(T\). If \(\rr(T,t_1)\) denotes the same tangle \(T\) with the orientation of the first strand reversed, then for all Alexander gradings \(a=(a_1,\dots,a_r)\in\mathbb{Z}^r\),
	$$\CFT(\rr(T,t_1),s,a)\cong\CFT(T,s,(-a_1,a_2,\dots,a_r))$$
	as (relatively) bigraded invariants. Similarly, if \(\rr(T)\) denotes the tangle \(T\) with the orientation of all strands reversed, then 
	$$\CFT(\rr(T),s,a)\cong\CFT(T,s,-a).$$
\end{proposition}
\begin{proof}
	The orientation of a tangle component is just a choice of an orientation of its meridian. This does not affect the relative $\delta$-grading.
\end{proof}
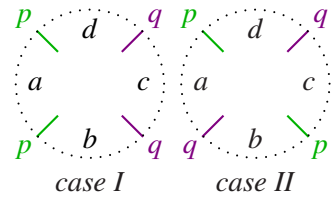
\begin{wrapfigure}{r}{0.3333\textwidth}
	\centering
	\vspace*{-10pt}
	\begin{pspicture}(-2.1,-1.5)(2.1,1.05)
	\rput(-1.1,0){
		\SpecialCoor
		\psline[linecolor=violet](0.6;45)(1;45)
		\psline[linecolor=violet](0.6;-45)(1;-45)
		
		\psline[linecolor=darkgreen](0.6;135)(1;135)
		\psline[linecolor=darkgreen](0.6;-135)(1;-135)
		
		\uput{0.6}[0](0,0){$c$}
		\uput{0.6}[90](0,0){$d$}
		\uput{0.6}[180](0,0){$a$}
		\uput{0.6}[270](0,0){$b$}
		
		\uput{1.05}[135](0,0){$\textcolor{darkgreen}{p}$}
		\uput{1.05}[-135](0,0){$\textcolor{darkgreen}{p}$}
		\uput{1.05}[45](0,0){$\textcolor{violet}{q}$}
		\uput{1.05}[-45](0,0){$\textcolor{violet}{q}$}
		\pscircle[linestyle=dotted](0,0){1}
		
		\rput(0,-1.3){\textit{case I}}
	}
	
	\rput(1.1,0){
		\psline[linecolor=violet](0.6;45)(1;45)
		\psline[linecolor=violet](0.6;-135)(1;-135)
		
		\psline[linecolor=darkgreen](0.6;135)(1;135)
		\psline[linecolor=darkgreen](0.6;-45)(1;-45)
		
		\uput{0.6}[0](0,0){$c$}
		\uput{0.6}[90](0,0){$d$}
		\uput{0.6}[180](0,0){$a$}
		\uput{0.6}[270](0,0){$b$}
		
		\uput{1.05}[135](0,0){$\textcolor{darkgreen}{p}$}
		\uput{1.05}[-45](0,0){$\textcolor{darkgreen}{p}$}
		\uput{1.05}[45](0,0){$\textcolor{violet}{q}$}
		\uput{1.05}[-135](0,0){$\textcolor{violet}{q}$}
		\pscircle[linestyle=dotted](0,0){1}
		\rput(0,-1.3){\textit{case II}}
	}
	
	\end{pspicture}
	\caption{Two cases for theorem~\ref{thm:fourendedHFT}}\label{fig:4endedcasesCAT}	
	\bigskip\bigskip
\end{wrapfigure}
\myfixwrapfig

\begin{theorem}[compare~\ref{thm:fourendedonecolour} and~\ref{rem:OtherRelation}]\label{thm:fourendedHFT}
	Let \(T\) be an oriented 4-ended tangle in $M$. We distinguish the two cases shown in figure~\ref{fig:4endedcasesCAT}.
	In both cases, 
	\begin{equation}
	\CFT(T,b)\cong\CFT(\rr(T),d)\label{eqn:BD}\tag{B-D}
	\end{equation}
	as (relatively) bigraded invariants.
	In case I, we also have
	\begin{equation}
	V_{\textcolor{darkgreen}{p}}\otimes\CFT(T,a)\cong V_{\textcolor{violet}{q}}\otimes\CFT(\rr(T),c), \tag{A-C}\label{eqn:AC}
	\end{equation}
	where \(V_t\) denotes a 2-dimensional vector space supported in consecutive Alexander and homological gradings.
	In case II, the second identity holds if we drop the tensor factors \(V_{\textcolor{darkgreen}{p}}\) and \(V_{\textcolor{violet}{q}}\).
\end{theorem}

\begin{Remark}\label{rem:SymRelCAT}
	Theorem~\ref{thm:INTROfourendedHFT} from the introduction, ie the univariate version of the theorem above, follows immediately, since two chain complexes are the same iff they are the same after tensoring with a free Abelian group such as $V_{\textcolor{darkgreen}{p}}$.
\end{Remark}

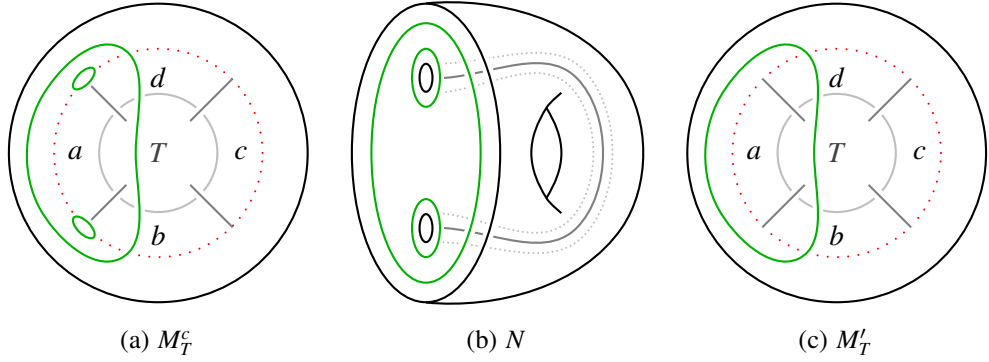
\begin{figure}[t]
	\centering
	\begin{subfigure}[b]{0.31\textwidth}
		\centering
		\psset{unit=0.2}
		\begin{pspicture}(-10,-10.3)(10,10.3)
		\pscircle(0,0){10}
		\pscircle[linecolor=red,linestyle=dotted](0,0){7}

		\pscircle[linecolor=lightgray](0,0){4}
		
		\psline[linecolor=white,linewidth=4pt](3;45)(6.3;45)
		\psline[linecolor=white,linewidth=4pt](3;135)(6.3;135)
		\psline[linecolor=white,linewidth=4pt](3;-135)(6.3;-135)
		\psline[linecolor=white,linewidth=4pt](3;-45)(6.3;-45)
		
		\psline[linecolor=gray](3;45)(7;45)
		\psline[linecolor=gray](3;135)(6.3;135)
		\psline[linecolor=gray](3;-135)(6.3;-135)
		\psline[linecolor=gray](3;-45)(7;-45)

		\rput{-135}(7;-135){\psellipse[linecolor=darkgreen,fillcolor=white,fillstyle=solid](0,0)(0.5,1)}
		\rput{135}(7;135){\psellipse[linecolor=darkgreen,fillcolor=white,fillstyle=solid](0,0)(0.5,1)}
		
		\psecurve[linecolor=white,linewidth=4pt](7.5;110)(1.5;180)(7.5;-110)(8.7;-155)
		(8.6;155)(7.5;110)(1.5;180)(7.5;-110)
		\psecurve[linecolor=darkgreen](7.5;110)(1.5;180)(7.5;-110)(8.7;-155)
		(8.6;155)(7.5;110)(1.5;180)(7.5;-110)
		
		\rput(-5.5,0){$a$}
		\rput(0,-5.5){$b$}
		\rput(5.5,0){$c$}
		\rput(0,5){$d$}
		
		\rput(0,0){\textcolor{darkgray}{$T$}}
		
		\end{pspicture}
		\caption{$M_T^c$}\label{fig:SurfaceDecompBefore}
	\end{subfigure}
	\quad
	\begin{subfigure}[b]{0.3\textwidth}
		\centering
		\psset{unit=0.2}
		\begin{pspicture}(-10,-10.3)(9.5,10.3)
		
		\psecurve[linecolor=gray](-10,6)(-5,5)(5,5)(5,-5)(-5,-5)(-10,-6)
		
		\psellipse*[linecolor=white,linewidth=4pt](-5,5)(0.5,1)
		\psellipse*[linecolor=white,linewidth=4pt](-5,-5)(0.5,1)
		
		\psecurve[linecolor=lightgray,linestyle=dotted,dotsep=1pt](-10,5)(-5,4)(4.5,4.5)(4.5,-4.5)(-5,-4)(-10,-5)
		\psecurve[linecolor=lightgray,linestyle=dotted,dotsep=1pt](-10,7)(-5,6)(5.5,5.5)(5.5,-5.5)(-5,-6)(-10,-7)
		
		\psellipse[linecolor=white,linewidth=4pt](-5,0)(5.3,10.3)
		
		\psecurve(-13,5)(-5,9.95)(8,5)(8,-5)(-5,-9.95)(-13,-5)
		\psellipse(-5,0)(5,10)
		\rput(-2,0){
			\psecurve(4,4)(5,3)(6,0)(5,-3)(4,-4)
			\pscurve(6,4)(5,3)(4,0)(5,-3)(6,-4)
		}
		
		\psellipse[linecolor=white,linewidth=4pt](-5,0)(4,9)
		\psellipse[linecolor=darkgreen](-5,0)(3.7,8.7)
		
		\psellipse[linecolor=white,linewidth=4pt](-5,5)(1.2,2.2)
		\psellipse[linecolor=white,linewidth=4pt](-5,-5)(1.2,2.2)
		\psellipse[linecolor=darkgreen,fillcolor=white,fillstyle=solid](-5,5)(1,2)
		\psellipse[linecolor=darkgreen,fillcolor=white,fillstyle=solid](-5,-5)(1,2)
		
		\psellipse(-5,5)(0.5,1)
		\psellipse(-5,-5)(0.5,1)
		
		\end{pspicture}
		\caption{$N$}\label{fig:SurfaceDecompN}
	\end{subfigure}
	\quad
	\begin{subfigure}[b]{0.31\textwidth}
		\centering
		\psset{unit=0.2}
		\begin{pspicture}(-10,-10.3)(10,10.3)
		\pscircle(0,0){10}
		\pscircle[linecolor=red,linestyle=dotted](0,0){7}

		\pscircle[linecolor=lightgray](0,0){4}
		
		\psline[linecolor=white,linewidth=4pt](3;45)(7;45)
		\psline[linecolor=white,linewidth=4pt](3;135)(7;135)
		\psline[linecolor=white,linewidth=4pt](3;-135)(7;-135)
		\psline[linecolor=white,linewidth=4pt](3;-45)(7;-45)
		
		\psline[linecolor=gray](3;45)(7;45)
		\psline[linecolor=gray](3;135)(7;135)
		\psline[linecolor=gray](3;-135)(7;-135)
		\psline[linecolor=gray](3;-45)(7;-45)

		%\rput{-135}(7;-135){\psellipse[linecolor=darkgreen,fillcolor=white,fillstyle=solid](0,0)(0.5,1)}
		%\rput{135}(7;135){\psellipse[linecolor=darkgreen,fillcolor=white,fillstyle=solid](0,0)(0.5,1)}
		
		\psecurve[linecolor=white,linewidth=4pt](7.5;110)(1.5;180)(7.5;-110)(8.7;-155)
		(8.6;155)(7.5;110)(1.5;180)(7.5;-110)
		\psecurve[linecolor=darkgreen](7.5;110)(1.5;180)(7.5;-110)(8.7;-155)
		(8.6;155)(7.5;110)(1.5;180)(7.5;-110)

		\rput(-5.5,0){$a$}
		\rput(0,-5.5){$b$}
		\rput(5.5,0){$c$}
		\rput(0,5){$d$}
		
		\rput(0,0){\textcolor{darkgray}{$T$}}
		
		\end{pspicture}
		\caption{$M_T^\prime$}\label{fig:SurfaceDecompX}
	\end{subfigure}
	\caption{The surface decomposition used in the proof of theorem~\ref{thm:fourendedHFT}}\label{fig:SurfaceDecomp}
\end{figure}

\begin{proof}[Proof of theorem~\ref{thm:fourendedHFT}]
	Let us consider relation \eqref{eqn:BD} first. The underlying sutured manifolds are the same after switching the roles of $R_-$ and $R_+$ on one side. This can be easily seen by pushing the two meridional sutures of the open tangle components through the tangle. Then, by \cite[proposition~2.14]{DecatSFH}, the sutured Floer homologies are identical, except that the $\Spinc$-gradings are opposite to each other. Now apply proposition~\ref{prop:HFTreverseorientI}.
	
	In case~II, \eqref{eqn:AC} (without the tensor factors) follows from the same arguments. 
	In case~I, relation~\eqref{eqn:AC} is an exercise in applying the surface decomposition formula for sutured Floer homology, see figure~\ref{fig:SurfaceDecomp}. Consider $M_T^c$. Let $N$ be a tubular neighbourhood of the union of~$R_-$ and the component of~$R_+$ corresponding to the $\textcolor{darkgreen}{p}$-strand. Let $S$ be the surface obtained as the intersection of~$N$ with the closure of $M_T^\prime:=M_T^c\smallsetminus N$. Note that $M_T^\prime$ is diffeomorphic to~$M_T$. We turn it into a balanced sutured manifold by adding a single suture on the boundary of $M$, separating the $\textcolor{darkgreen}{p}$-ends from the ${\textcolor{violet}{q}}$-ends. We get a decomposition
	$$M_T^c\rightsquigarrow^S M_T^\prime\cup N,$$
	which satisfies all conditions of \cite[proposition~8.6]{SurfaceDecomposition}, except that $M'_T$ might not be taut. If $M'_T$ is not taut because it is not irreducible, we may apply the connected sum formula~\cite[proposition~9.15]{Juhasz} on both sides first. So we may assume that $M'_T$ is irreducible. Then, by definition, if $M'_T$ is not taut, one of the components of $\partial M'_T\smallsetminus\Gamma$ is not norm minimizing. Hence, there is an embedded disk separating the $\textcolor{darkgreen}{p}$- and ${\textcolor{violet}{q}}$-strand, so the sutured Floer homologies of both $M'_T$ and $M_T$ vanish by~\cite[proposition~9.18]{Juhasz}. So, in any case,
	$$\SFH(M_T^c)\cong\SFH(M_T^\prime)\otimes \SFH(N).$$
	It is now straightforward to calculate $\SFH(N)$ which gives $V_{\textcolor{darkgreen}{p}}$. Thus,
	$$\HFT(T,c)\cong\SFH(M_T^\prime)\otimes V_{\textcolor{darkgreen}{p}}$$
	and similarly
	$$\HFT(T,a)\cong\SFH(M_T^{\prime\prime})\otimes V_{\textcolor{violet}{q}},$$
	where $M_T^{\prime\prime}$ agrees with $M_T^\prime$, except that the roles of $R_-$ and $R_+$ are interchanged. 
	To get relation~\eqref{eqn:AC}, we now argue just as before.
\end{proof}

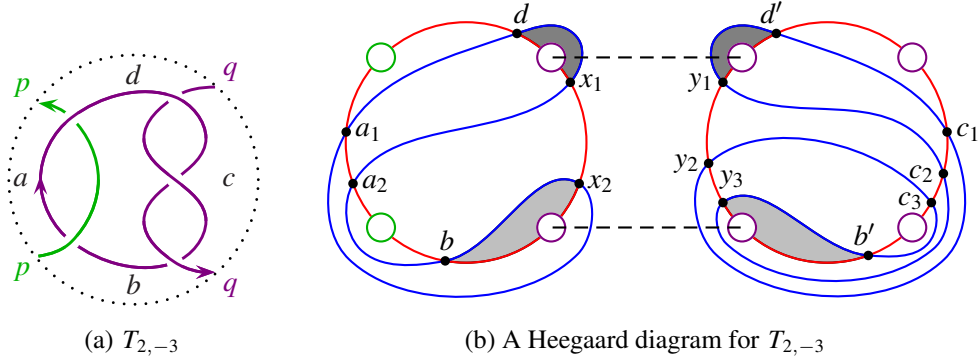
\begin{figure}[t]
	\centering
	\begin{subfigure}[b]{0.29\textwidth}\centering
		\psset{unit=0.55, linewidth=1.1pt}
		\begin{pspicture}[showgrid=false](-4.2,-3.1)(2.2,3.1)
		\psecurve[linecolor=violet]{->}(-2.5,1.5)(0,2)(0.75,1)(-0.75,-1)(0,-2)(0.97,-2.24)(2,-2)
		\psecurve[linecolor=violet](2,2)(0.97,2.24)(0,2)(-0.75,1)(0.75,-1)(0,-2)(-2.5,-1.5)(-3.25,0)(-2.5,1.5)(0,2)(0.75,1)
		\psecurve[linecolor=darkgreen]{<-}(-6,1.5)(-3.3,1.85)(-2.5,1.5)(-1.85,0)(-2.5,-1.5)(-3.3,-1.85)(-6,-1.5)
		\pscircle*[linecolor=white](-2.5,1.5){0.2}
		
		\psecurve[linecolor=violet](0.75,-1)(0,-2)(-2.5,-1.5)(-3.25,0)(-2.5,1.5)
		
		\pscircle*[linecolor=white](0,2){0.2}
		\pscircle*[linecolor=white](0,0){0.2}
		\pscircle*[linecolor=white](0,-2){0.2}
		
		\psecurve[linecolor=violet]{->}(0.75,1)(-0.75,-1)(0,-2)(0.97,-2.24)(2,-2)
		\psecurve[linecolor=violet](0,2)(-0.75,1)(0.75,-1)(0,-2)
		\psecurve[linecolor=violet](-2.5,-1.5)(-3.25,0)(-2.5,1.5)(0,2)(0.75,1)(-0.75,-1)
		
		\pscircle*[linecolor=white](-2.5,-1.5){0.2}
		\psecurve[linecolor=darkgreen](-2.5,1.5)(-1.85,0)(-2.5,-1.5)(-3.3,-1.85)(-6,-1.5)
		
		\psline[linecolor=violet]{->}(-3.25,-0.1)(-3.25,0.1)
		\pscircle[linestyle=dotted](-1,0){3.05}
		
		%\psdot[linecolor=violet](0.97,2.24)
		\uput{0.2}[45](0.97,2.24){$\textcolor{violet}{q}$}
		%\psdot[linecolor=violet](0.97,-2.24)
		\uput{0.2}[-45](0.97,-2.24){$\textcolor{violet}{q}$}
		%\psdot[linecolor=darkgreen](-3.3,1.85)
		\uput{0.2}[135](-3.3,1.85){$\textcolor{darkgreen}{p}$}
		%\psdot[linecolor=darkgreen](-3.3,-1.85)
		\uput{0.2}[-135](-3.3,-1.85){$\textcolor{darkgreen}{p}$}

		\uput{2.5}[180](-1,0){$a$}
		\uput{2.3}[-90](-1,0){$b$}
		\uput{2.1}[0](-1,0){$c$}
		\uput{2.3}[90](-1,0){$d$}

		\end{pspicture}
		\caption{$T_{2,-3}$}\label{fig:mutationpretzeltangleT}
	\end{subfigure}
	\quad
	\begin{subfigure}[b]{0.67\textwidth}\centering
		\psset{unit=0.2}
		\begin{pspicture}(-22,-11)(22,11)
		%%left
		\rput(-12,0){\psrotate(0,0){-90}{

				\pscustom[fillstyle=solid,fillcolor=lightgray]{\psecurve[linewidth=0pt](10.5;45)(8;70)(8;-10)(10;-45)
					\psarc[linewidth=0pt](0,0){8}{-10}{70}
				}
				
				\pscustom[fillstyle=solid,fillcolor=gray]{\psecurve[linewidth=0pt](8;-70)(8;120)(10;130)(10;140)(8;155)(8;-95)
					\psarcn[linewidth=0pt](0,0){8}{155}{120}
				}

				\psarc[linecolor=red](0,0){8}{0}{360}
				
				\SpecialCoor
				\rput(8;45){\pscircle*[linecolor=white]{1}\pscircle[linecolor=violet]{1}}
				\rput(8;135){\pscircle*[linecolor=white]{1}\pscircle[linecolor=violet]{1}}
				\rput(8;225){\pscircle*[linecolor=white]{1}\pscircle[linecolor=darkgreen]{1}}
				\rput(8;315){\pscircle*[linecolor=white]{1}\pscircle[linecolor=darkgreen]{1}}
				
				%\psecurve[linecolor=yellow](8;-65)(9.5;132)(8;-70)(11;-45)(10.5;45)(8;65)(8;-60)(10;-45)(9.5;45)(8;55)(9.5;-45)(8;-55)(8;60)(10;45)(10.5;-45)(8;-65)(9.5;132)(8;-70)
				
				\psecurve[linecolor=blue]%
				(10;130)(10;140)(8;155)%
				(8;-95)(11;-45)%
				(10.5;45)(8;70)%
				(8;-10)%
				(10;-45)(8;-70)%
				(8;120)(10;130)(10;140)(8;155)%

				\psdot(8;155)%
				\psdot(8;120)%
				\psdot(8;70)%
				\psdot(8;-10)%
				\psdot(8;-70)%
				\psdot(8;-95)%
			}
			\rput[b](8.7;65){$d$}
			\rput[l](8;30){~$x_1$}
			\rput[l](8;-20){~$x_2$}
			\rput[b](7.3;-100){$b$}
			\rput[l](8;-160){~$a_2$}
			\rput[l](8;-185){~$a_1$}
		}

		%%right
		\rput(12,0){\psrotate(0,0){-90}{
				
				\pscustom[fillstyle=solid,fillcolor=lightgray]{\psecurve[linewidth=0pt](9.5;45)(8;20)(8;-60)(9.5;-45)
					\psarc[linewidth=0pt](0,0){8}{-60}{20}
				}
				
				\pscustom[fillstyle=solid,fillcolor=gray]{\psecurve[linewidth=0pt](5.5;100)(8;-120)(10;-130)(10;-140)(8;-155)(8;95)
					\psarc[linewidth=0pt](0,0){8}{-155}{-120}
				}
				
				\psarc[linecolor=red](0,0){8}{0}{360}
				
				\SpecialCoor
				\rput(8;45){\pscircle*[linecolor=white]{1}\pscircle[linecolor=violet]{1}}
				\rput(8;135){\pscircle*[linecolor=white]{1}\pscircle[linecolor=violet]{1}}
				\rput(8;225){\pscircle*[linecolor=white]{1}\pscircle[linecolor=violet]{1}}
				\rput(8;315){\pscircle*[linecolor=white]{1}\pscircle[linecolor=violet]{1}}

				%\psecurve[linecolor=yellow](8;-65)(9.5;132)(8;-70)(11;-45)(10.5;45)(8;65)(8;-60)(10;-45)(9.5;45)(8;55)(9.5;-45)(8;-55)(8;60)(10;45)(10.5;-45)(8;-65)(9.5;132)(8;-70)
				
				\psecurve[linecolor=blue]%
				(10;-130)(10;-140)(8;-155)%
				(8;95)(11;45)%
				(10.5;-45)(8;-80)%
				(8;60)
				(9.5;45)(8;20)
				(8;-60)(9.5;-45)(10;45)
				(8;75)%
				(5.5;100)%
				(8;-120)(10;-130)(10;-140)(8;-155)%
				
				\psdot(8;-155)%
				\psdot(8;-120)%
				\psdot(8;-80)%
				\psdot(8;-60)%
				\psdot(8;20)%
				\psdot(8;60)%
				\psdot(8;75)%
				\psdot(8;95)%
			}
			\rput[b](8.7;-245){$d'$}
			\rput[r](8;-210){$y_1$~}
			\rput[r](8;-170){$y_2$~}
			\rput[b](7;-154){$y_3$}
			\rput[b](7.3;-70){$b'$}
			\rput[r](8;-29){$c_3$~}
			\rput[r](8;-15){$c_2$~}
			\rput[l](8;5){~$c_1$}
		}
		
		\psline[linestyle=dashed](-6.25,5.65685424949236)(6.25,5.65685424949236)
		
		\psline[linestyle=dashed](-6.25,-5.65685424949236)(6.25,-5.65685424949236)
		
		\end{pspicture}
		\caption{A Heegaard diagram for $T_{2,-3}$}
		\label{fig:mutationpretzeltangleHD}
	\end{subfigure}
	\\
	\vspace*{11pt}
	\begin{subfigure}[b]{0.95\textwidth}\centering
		\begin{tabular}{c|c|c|c}
			
			\textbf{site a} & \textbf{site b} & \textbf{site c} & \textbf{site d} \\ 
			$a_1y_1: \textcolor{darkgreen}{p}^{0} \textcolor{violet}{q}^{+3} \delta^{-\frac{1}{2}} $
			&
			$by_1: \textcolor{darkgreen}{p}^{-1} \textcolor{violet}{q}^{+1} \delta^{0} $ 
			&
			$x_1c_1: \textcolor{darkgreen}{p}^{+1} \textcolor{violet}{q}^{+2} \delta^{-\frac{1}{2}} $ 
			&
			\cancel{$dy_1: \textcolor{darkgreen}{p}^{+1} \textcolor{violet}{q}^{+3} \delta^{0}$} 
			\\ 
			$a_1y_2: \textcolor{darkgreen}{p}^{0} \textcolor{violet}{q}^{+1} \delta^{-\frac{1}{2}} $ 
			&
			$by_2: \textcolor{darkgreen}{p}^{-1} \textcolor{violet}{q}^{-1} \delta^{0} $ 
			& 
			$x_1c_2: \textcolor{darkgreen}{p}^{+1} \textcolor{violet}{q}^{0} \delta^{-\frac{1}{2}} $ 
			& 
			$dy_2: \textcolor{darkgreen}{p}^{+1} \textcolor{violet}{q}^{+1} \delta^{0} $ 
			\\ 
			$a_1y_3: \textcolor{darkgreen}{p}^{0} \textcolor{violet}{q}^{-1} \delta^{-\frac{1}{2}} $ 
			& 
			\cancel{$by_3: \textcolor{darkgreen}{p}^{-1} \textcolor{violet}{q}^{-3} \delta^{0}$} 
			& 
			$x_1c_3: \textcolor{darkgreen}{p}^{+1} \textcolor{violet}{q}^{-2} \delta^{-\frac{1}{2}} $ 
			& 
			$dy_3: \textcolor{darkgreen}{p}^{+1} \textcolor{violet}{q}^{-1} \delta^{0} $ 
			\\ 
			$a_2y_1: \textcolor{darkgreen}{p}^{0} \textcolor{violet}{q}^{+1} \delta^{-\frac{1}{2}} $ 
			& 
			$x_1b': \textcolor{darkgreen}{p}^{+1} \textcolor{violet}{q}^{-3} \delta^{-1} $  
			& 
			$x_2c_1: \textcolor{darkgreen}{p}^{-1} \textcolor{violet}{q}^{+2} \delta^{-\frac{1}{2}} $ 
			& 
			\cancel{$x_1d': \textcolor{darkgreen}{p}^{+1} \textcolor{violet}{q}^{+3} \delta^{-1}$} 
			\\ 
			$a_2y_2: \textcolor{darkgreen}{p}^{0} \textcolor{violet}{q}^{-1} \delta^{-\frac{1}{2}} $ 
			& 
			\cancel{$x_2b': \textcolor{darkgreen}{p}^{-1} \textcolor{violet}{q}^{-3} \delta^{-1}$}  
			& 
			$x_2c_2: \textcolor{darkgreen}{p}^{-1} \textcolor{violet}{q}^{0} \delta^{-\frac{1}{2}} $ 
			& 
			$x_2d': \textcolor{darkgreen}{p}^{-1} \textcolor{violet}{q}^{+3} \delta^{-1} $   
			\\ 
			$a_2y_3: \textcolor{darkgreen}{p}^{0} \textcolor{violet}{q}^{-3} \delta^{-\frac{1}{2}} $ 
			& 
			& 
			$x_2c_3: \textcolor{darkgreen}{p}^{-1} \textcolor{violet}{q}^{-2} \delta^{-\frac{1}{2}} $ 
			&
			\\ 
		\end{tabular} 
		\caption{A table of the generators of $\HFT(T_{2,-3})$ and their gradings. The generators that can be cancelled are crossed out.}\label{fig:mutationpretzeltangleTresult}
	\end{subfigure}
	\caption{The calculation of $\HFT$ for the $(2,-3)$-pretzel tangle $T_{2,-3}$, see example~\ref{exa:pretzeltangle}}\label{fig:mutationexample}
\end{figure}

\begin{example}[the $(2,-3)$-pretzel tangle]\label{exa:pretzeltangle}
	In figure~\ref{fig:mutationexample}, we compute $\HFT$ for the $(2,-3)$-pretzel tangle $T_{2,-3}$. Note that the relative gradings of the generators have been lifted to absolute ones in such a way that they agree with the gradings of the corresponding Kauffman states of the diagram in figure~\ref{fig:mutationpretzeltangleT}. The shaded regions in the Heegaard diagram in figure~\ref{fig:mutationpretzeltangleHD} show the only two domains that contribute to the differential. It is interesting to note that if we set $t:=\textcolor{darkgreen}{p}=\textcolor{violet}{q}$, then the results for the sites $a$ and $c$ are the same without modification, and those for the sites $b$ and $d$ are the same after reversing the orientation $t\leftrightarrow t^{-1}$. 
	
	If bigraded knot Floer homology were mutation invariant, one might expect that also $\HFT(T_{2,-3},b)\cong\HFT(T_{2,-3},d)$. This, however, is only true for the $\delta$-graded version. One might want to interpret the fact that the symmetry relations for $\HFT$ from theorem~\ref{thm:fourendedHFT} are slightly weaker than those for $\nabla_T^s$ from theorem~\ref{thm:fourendedonecolour} as a first indication \textit{why} mutation invariance of bigraded knot Floer homology does not hold, but might be expected for the $\delta$-graded version. In fact, in~\cite{pqMod}, I prove that mutation about the tangle $T_{2,-3}$ preserves $\delta$-graded knot Floer homology, using a slightly more sophisticated invariant based on $\HFT$. 
\end{example}

	%%%%%%%%%%%%%%%%%%%%   End of main body of article
	%
	%                             References

\end{document}